\documentclass[11pt]{article}
\usepackage{latexsym,amsfonts,amssymb,amsmath,amsthm}
\usepackage{graphicx}

\usepackage[usenames,dvipsnames]{color}
\usepackage{ulem}

\usepackage{color}

\begin{document}

\newtheorem{tm}{Theorem}[section]
\newtheorem{prop}{Proposition}[section]
\newtheorem{defin}{Definition}[section]
\newtheorem{coro}{Corollary}[section]
\newtheorem{lem}{Lemma}[section]
\newtheorem{assumption}{Assumption}[section]
\newtheorem{rk}{Remark}[section]
\newtheorem{nota}{Notation}[section]
\numberwithin{equation}{section}

    \newcommand{\lb}{\label}
  \newcommand{\beq}{\begin{equation}}
    \newcommand{\eeq}{\end{equation}}

\newcommand{\stk}[2]{\stackrel{#1}{#2}}
\newcommand{\dwn}[1]{{\scriptstyle #1}\downarrow}
\newcommand{\upa}[1]{{\scriptstyle #1}\uparrow}
\newcommand{\nea}[1]{{\scriptstyle #1}\nearrow}
\newcommand{\sea}[1]{\searrow {\scriptstyle #1}}
\newcommand{\csti}[3]{(#1+1) (#2)^{1/ (#1+1)} (#1)^{- #1
 / (#1+1)} (#3)^{ #1 / (#1 +1)}}
\newcommand{\RR}[1]{\mathbb{#1}}

\newcommand{\rd}{{\mathbb R^d}}
\newcommand{\ep}{\varepsilon}
\newcommand{\rr}{{\mathbb R}}
\newcommand{\alert}[1]{\fbox{#1}}
\newcommand{\eqd}{\sim}
\def\p{\partial}
\def\R{{\mathbb R}}
\def\N{{\mathbb N}}
\def\Q{{\mathbb Q}}
\def\C{{\mathbb C}}
\def\l{{\langle}}
\def\r{\rangle}
\def\t{\tau}
\def\k{\kappa}
\def\a{\alpha}
\def\la{\lambda}
\def\De{\Delta}
\def\de{\delta}
\def\ga{\gamma}
\def\Ga{\Gamma}
\def\ep{\varepsilon}
\def\eps{\varepsilon}
\def\si{\sigma}
\def\Re {{\rm Re}\,}
\def\Im {{\rm Im}\,}
\def\E{{\mathbb E}}
\def\P{{\mathbb P}}
\def\Z{{\mathbb Z}}
\def\D{{\mathbb D}}
\def\p{\partial}
\newcommand{\ceil}[1]{\lceil{#1}\rceil}

\def\calK{{\mathcal K}}
\def\calR{{\mathcal R}}
\def\calM{{\mathcal M}}
\def\bB{{\mathbb B}}
\def\calT{{\mathcal T}}

\newcommand{\qq}{\hfill \ensuremath{\Box}}

\title{Global existence of classical solutions of chemotaxis systems with logistic source and consumption or linear signal production on $\mathbb{R}^{n}$}

\author{
Zulaihat Hassan,  Wenxian Shen, and Yuming Paul Zhang  \\
Department of Mathematics and Statistics\\
Auburn University, AL 36849\\
U.S.A. }

\date{}
\maketitle

\begin{abstract}
While much literature on chemotaxis systems focuses on bounded domains, this paper emphasizes the global existence of classical solutions for  the following three primary chemotaxis systems with a logistic source on $\mathbb{R}^n$,
\begin{equation}
\label{abstract-eq1}
\begin{cases}
u_t= \Delta u-\chi\nabla \cdot(u\nabla v)+u(a-bu),\quad x\in \mathbb{R}^n\cr
\tau v_t= \Delta v-uv,\quad x\in\mathbb{R}^n,
\end{cases}
\end{equation}
\begin{equation}
\label{abstract-eq2}
\begin{cases}
u_t= \Delta u-\chi\nabla \cdot(u\nabla v)+u(a-bu),\quad x\in \mathbb{R}^n\cr
\tau v_t= \Delta v-\lambda v + \mu u,\quad x\in\mathbb{R}^n,
\end{cases}
\end{equation}
and
\begin{equation}
\label{abstract-eq3}
\begin{cases}
u_t= \Delta u-\chi\nabla \cdot(u\nabla v)+u(a-bu),\quad x\in \mathbb{R}^n\cr
0= \Delta v-\lambda v + \mu u,\quad x\in\mathbb{R}^n,
\end{cases}
\end{equation}
where $\chi$ is a nonzero number and $a,b,\lambda,\mu, \tau$ are positive constants.
Suppose that $(u,v)$ is a nonnegative classical solution of \eqref{abstract-eq1} (resp. \eqref{abstract-eq2}, \eqref{abstract-eq3}) on the maximal interval $(0,T_{\max})$. We first prove in a unified way  that if {there is $p>\max\{1,\frac{n}{2}\}$ such that}
$$
\limsup_{t\to T_{\max}}\sup_{x_0\in\mathbb{R}^n}\int_{B(x_0,1)}u^p(t,x+x_0)dx<\infty,
$$
where {$B(x_0,1)$} is a ball centered at $x_0$ with radius $1$, then
 $$T_{\max}=\infty\quad {\rm and}\quad
\limsup_{t\to\infty}\|u(t,\cdot)\|_{L^\infty(\mathbb{R}^n)}<\infty.
$$
We then provide sufficient conditions for the global existence and boundedness of
classical solutions  of \eqref{abstract-eq1},  \eqref{abstract-eq2}, and \eqref{abstract-eq3} with nonnegative initial functions. It follows that nonnegative classical solutions of \eqref{abstract-eq1},  \eqref{abstract-eq2}, and \eqref{abstract-eq3}  exist globally and stays bounded in one- and two-dimensional settings for any chemotaxis sensitivity $\chi$.
Notably, the methods developed in this paper can be adapted to the bounded domain case without much effort.  Several existing results for \eqref{abstract-eq1},  \eqref{abstract-eq2}, and \eqref{abstract-eq3}  on bounded domains are then improved.
\end{abstract}

\medskip


\medskip


\tableofcontents

\section{Introduction}

\subsection{Overview}

Movement of living organisms or biological  species are partially oriented in response to certain
chemicals  in many
 biological phenomena such as bacteria aggregation, immune system response or angiogenesis
in the embryo formation and in tumor development.
 Theoretical and mathematical modeling
of chemotaxis phenomenon dates back to the pioneering works of Patlak in the 1950s (\cite{Pat})  and Keller-Segel in 1970s (\cite{KeSe0, KeSe1, KeSe2}).
 Since these  pioneering works, a rich variety of mathematical models for studying chemotaxis have been introduced, and  the area  remains as one of the most active
ones in nonlinear PDEs. The reader is referred to \cite{ArTy, BeBeTaWi, HiPa, Hor, Pai}  for some detailed introduction into the mathematics and applications
of chemotaxis models and  some survey on the recent developments.

The following  is a general chemotaxis system to describe the movement of cells under the influence of some chemical signal,
\begin{equation}
\label{general-eq}
\begin{cases}
\p_t u=\nabla (D(u,v)\nabla u-\chi(u,v)\nabla v)+f(u,v),\quad x\in\Omega\cr
\tau \p_t v=d \Delta v+g(u,v)-h(u,v)v,\quad x\in\Omega,
\end{cases}
\end{equation}
where $u$ represents the cell  density on a given domain
 $\Omega\subset\mathbb{R}^n$ and $v$ denotes
the concentration of the chemical signal. The  function $D(u,v)$ describes the diffusivity
of the cells and  $\chi(u,v)$ represents the chemotaxis  sensitivity.   The function $f(u,v)$ models  cell growth and death, and
functions $g(u, v)$ and $h(u, v)$  describe production and degradation
of the chemical signal, respectively.
The parameters $\tau$ and $d$ are   linked to the speed of diffusion of the chemical substance.  Biologically, $\tau=0$ indicates that the chemical substance  diffuses much faster than the cells, while $d=0$ indicates that the chemical substance diffuses much slower than the cells. One of the central problems for chemotaxis models is whether finite-time blow-up occurs and  whether any nonnegative solution exists globally and stays bounded.

There are many studies on the finite-time blow-up and  global existence of  nonnegative solutions of various special cases of \eqref{general-eq}  on bounded domain $\Omega$ with Neumann boundary conditions (see (1.3)-(1.81) and Tables 1 and 2 in \cite{ArTy} for  various special cases of \eqref{general-eq}).  For example,  consider  the following three special cases of \eqref{general-eq},
\begin{equation}
\label{special-eq1}
\begin{cases}
u_t=\Delta u-\chi\nabla\cdot(u\nabla v)+u(a-bu),\quad x\in\Omega\cr
\tau v_t=\Delta v-uv,\quad x\in\Omega\cr
u(0,x)=u_0(x),\,\, v(0,x)=v_0(x),\quad x\in\Omega,
\end{cases}
\end{equation}
\begin{equation}
\label{special-eq2}
\begin{cases}
u_t=\Delta u-\chi\nabla \cdot(u\nabla v)+u(a-bu),\quad x\in\Omega\cr
\tau v_t=\Delta v+\mu u-\lambda v,\quad x\in\Omega\cr
u(0,x)=u_0(x),\,\, v(0,x)=v_0(x),\quad x\in\Omega,
\end{cases}
\end{equation}
and
\begin{equation}
\label{special-eq3}
\begin{cases}
u_t=\Delta u-\chi\nabla \cdot(u\nabla v)+u(a-bu),\quad x\in\Omega\cr
0= v+\mu u-\lambda v,\quad x\in\Omega\cr
u(0,x)={u_0(x)},\quad x\in\Omega
\end{cases}
\end{equation}
complemented with Neumann boundary conditions when $\Omega$ is a  bounded domain in $\R^n$, where $\tau>0$, $\chi\in\R$, and $a,b\ge 0$ are constants.
Note that
 $\chi>0$ corresponds to the positive taxis
and $\chi<0$ corresponds to the negative taxis.  The reaction function $u(a-bu)$ is referred to as {\it logistic source}. When $a=b=0$, \eqref{special-eq1}-\eqref{special-eq3} are referred to as {\it minimal chemotaxis systems}.
  Biologically, \eqref{special-eq1} describes such cases  that  the chemical substance is consumed by the cells, and is referred to as a {\it chemotaxis system with  consumption}.  \eqref{special-eq2} and \eqref{special-eq3}  describe such cases that the chemical substance is produced
by the cells at the rate $\mu$ and degrades at the rate $\lambda$, and are referred to as {\it chemotaxis systems with linear signal production}.  There are many studies of finite-time blow-up and global existence of classical solutions of \eqref{special-eq1}-\eqref{special-eq3}
on bounded domain $\Omega$ complemented with Neumann boundary conditions.
Here are some existing results.

\begin{itemize}

\item Consider \eqref{special-eq1} with $\Omega\subset\R^n$ being a bounded smooth domain and complemented with the Neumann boundary condition,
\begin{equation}
\label{special-boundary-cond}
\frac{\partial u}{\partial n}=\frac{\partial v}{\partial n}=0,\quad x\in\partial \Omega.
\end{equation}

\item[{\bf --}] Assume that   $a=b=0$ and $\tau=1$. It is proved in \cite[Theorem 1.1]{Tao} that
 \eqref{special-eq1}+\eqref{special-boundary-cond} has a unique globally defined bounded classical solution with initial functions
$(u_0,v_0)\in \big(W^{1,p}(\Omega)\big)^2$ for some $p>n$  ($u_0,v_0\ge 0$)  provided
that
\begin{equation}
\label{bounded-domain-cond1}
0<\|v_0\|_{L^\infty(\Omega)}\cdot  \chi<\frac{1}{6(n+1)}.
\end{equation}
It is proved in \cite[Theorem 4.4]{ZhLi}  that   any globally defined bounded positive classical solutions of  \eqref{special-eq1}+\eqref{special-boundary-cond} converges to $\big(\frac{1}{|\Omega|}\int_\Omega u_0,0\big)$ as $t\to\infty$  (see also \cite{TaWi, Win1}
for the global existence and convergence of weak solutions of \eqref{special-eq1}+\eqref{special-boundary-cond} in the case that $n\ge 3$).

\item[{\bf --}]  Assume that $a,b>0$ and $\tau=1$.  It is proved in \cite[Theorem 3.3]{WaKhKh} that
 \eqref{special-eq1}+\eqref{special-boundary-cond} has a unique globally defined bounded classical solution with initial functions
$(u_0,v_0)\in \big(W^{1,p}(\Omega)\big)^2$ for some $p>n$  ($u_0,v_0\ge 0$)  provided
that
 \eqref{bounded-domain-cond1}  holds.  It is proved in \cite{LaWa}   that
 \eqref{special-eq1}+\eqref{special-boundary-cond} has a unique globally defined bounded classical solution with initial functions $(u_0,v_0)\in (C(\bar \Omega)\times C^1(\bar\Omega)$ ($u_0,v_0\ge 0$)  provided that $\chi\|v_0\|_{L^\infty(\Omega)}$ is small relative to $b$ (see \cite[Theorem 1.1]{LaWa}), and that any positive bounded globally defined classical solution of \eqref{special-eq1}+\eqref{special-boundary-cond} converges to $\big(\frac{a}{b},0\big)$ as $t\to\infty$ (see \cite[Theorem 1.2]{LaWa}). The reader is also referred to \cite{XLi} for the global existence of classical solutions of \eqref{special-eq1}+\eqref{special-boundary-cond} when $n=2$.

\item[{\bf --}] It remains open whether positive classical solutions of
 \eqref{special-eq1}+\eqref{special-boundary-cond} exist globally without any conditions on $\chi$ and $\|v_0\|_\infty$.

\item Consider \eqref{special-eq2}+\eqref{special-boundary-cond}  with $\Omega\subset\R^n$ being a bounded smooth domain.

\item[{\bf --}] If  $a=b=0$, it  is known that finite-time blow-up does not occur when $n=1$, and may occur when $n\ge 2$ (see \cite{HeVe2}, \cite{Nag}, etc.).

\item[{\bf --}]  If  $a, b>0$, it is known that finite-time blow-up does not occur when $n=1,2$ (see \cite{OsTsYa, OsYa1}). Hence the logistic source prevents finite-time blow-up to occur in certain sense.

\item[{\bf --}] Assume that $a,b>0$ and $\tau=1$. It is known that finite-time blow-up does not occur provided
\begin{equation}
\label{bounded-domain-cond2-1}
b>\frac{n|\chi|\mu}{4}
\end{equation}
(see \cite{IsSh, Win}).

\item[{\bf --}]  Assume that $a,b>0$ and $\tau$ is any positive number. It is proved in \cite[Theorem 1.1]{IsSh} (see also \cite[Theorem 2.2]{ZhLiBaZo} for the case $\tau=1$)
 that finite-time blow-up does not occur provided
\begin{equation}
\label{bounded-domain-cond2-2}
b>\inf_{\gamma>\max\{1,\frac{n}{2}\}}\Big(\frac{\gamma-1}{\gamma} \big(\widetilde C_{\gamma+1, n}\big)^{\frac{1}{\gamma+1}}\Big)|\chi|\mu,
\end{equation}
where $\widetilde C_{\gamma+1,n}$ is  a positive constant which is corresponding to the maximal  regularity of  the parabolic equation
$$
\begin{cases}
\tau v_t=\Delta v-v+g(t),\quad x\in\Omega\cr
{\frac{\partial v}{\partial n}=0,\quad x\in\partial\Omega}
\end{cases}
$$
(see \cite{Cao, HiPr, ZhLiBaZo}, and see  Lemma \ref{maximal-regularity-lm} about the maximal regularity of parabolic equations on the whole space).
Note that
$$
\inf_{\gamma>\max\{1,\frac{n}{2}\}}\Big(\frac{\gamma-1}{\gamma} \big(\widetilde C_{\gamma+1,n}\big)^{\frac{1}{\gamma+1}}\Big)|\chi|\mu\le \frac{(n-2)_+}{n} \Big(\widetilde C_{\frac{n}{2}+1,n}\Big)^{\frac{2}{n+2}} |\chi|\mu.
$$
It then also follows that finite-time blow-up does not occur when $n=1,2$.

\item[{\bf --}] It remains open whether finite-time blow-up does not occur when $b>0$.

\item  Consider \eqref{special-eq3}+\eqref{special-boundary-cond} with $\Omega\subset\R^n$ being a bounded smooth domain.

\item[{\bf --}]  Assume  that $a=b=0$.  It is known that finite-time blow-up does not occur when $n=1$
 and may occur when $n\ge 2$ (see \cite{HeVe1}, \cite{JaLu}, \cite{Nag}, etc.).

\item[{\bf --}] Assume that $a,b>0$.   $\lambda=\mu=1$, and $\Omega$ is bounded,  if $b>\frac{(n-2)_+}{n}\chi$, then
classical solutions of \eqref{special-eq3} with positive initials exist globally and stay bounded (see \cite[Theorem 2.5]{TeWi}). So the global existence is automatic if $n=1$ and $n=2$. The proof involves the Gagliardo-Nirenberg inequality, see (2.24) in \cite{TeWi}. If $\lambda,\mu$ are positive constants, the condition becomes
\begin{equation}
\label{bounded-domain-cond3}
b>\frac{(n-2)_+}{n}\chi\mu.
\end{equation}

\item[{\bf --}]   It remains open whether finite-time blow-up does not occur when $b>0$.

\end{itemize}

\medskip

In contrast to the bounded domain case,  the study of \eqref{special-eq1}-\eqref{special-eq3} with unbounded $\Omega$ is much less. To the best of our knowledge, regarding global well-posedness on the whole space  {with non-integrable initial functions}, only the problem of \eqref{special-eq2} with $\tau=1$ was discussed before in \cite{ShXu1, ShXu2}, see Remark \ref{main-eq2-rk}, {and the problem of \eqref{special-eq3} was discussed before in \cite{SaSh0,SaSh1,SaSh2}, see Remark \ref{main-eq3-rk}}. Furthermore, the recent paper \cite{HeRe} about travelling waves of \eqref{special-eq3} highlighted the complexity of this Cauchy problem.
The current paper is devoted to the study of global existence and boundedness  of classical solutions
of \eqref{special-eq1}-\eqref{special-eq3} when $\Omega=\R^n$, that is,
 the following chemotaxis models on the whole space $\R^n$,
\begin{equation}
\label{main-eq1}
\begin{cases}
u_t= \Delta u-\chi \nabla \cdot(u\nabla v)+u(a-bu),\quad x\in \R^n\cr
\tau v_t= \Delta v-uv,\quad x\in\R^n\cr
u(0,x)=u_0(x),\quad v(0,x)=v_0(x),\quad x\in \R^n,
\end{cases}
\end{equation}
\begin{equation}
\label{main-eq2}
\begin{cases}
u_t= \Delta u-\chi \nabla \cdot(u\nabla v)+u(a-bu),\quad x\in \R^n\cr
\tau v_t= \Delta v-\lambda v +\mu u,\quad x\in\R^n\cr
u(0,x)=u_0(x), \quad v(0,x)=v_0(x),\quad x\in\R^n,
\end{cases}
\end{equation}
and
\begin{equation}
\label{main-eq3}
\begin{cases}
u_t= \Delta u-\chi \nabla \cdot(u\nabla v)+u(a-bu),\quad x\in \R^n\cr
0= \Delta v-\lambda v +\mu u,\quad x\in\R^n\cr
u(0,x)=u_0(x),\quad x\in\R^n,
\end{cases}
\end{equation}
where $\chi\in\R$ and  $a,b,\tau>0$ are constants, and $u_0,v_0$ are proper non-negative functions on $\R^n$.   In the rest of this introduction, we introduce some definitions and  standing notations, state the main results, and provide some remarks about the main results and their proofs in subsections \ref{notations}, \ref{main-thms}, and \ref{remarks}, respectively.

\subsection{Definitions and notations}
\label{notations}

We first give the definition of classical solutions of \eqref{main-eq1} (resp. \eqref{main-eq3}, \eqref{main-eq2}). We call  $(u,v)$ a classical solution of \eqref{main-eq1} (resp. \eqref{main-eq3}, \eqref{main-eq2}) on $[0,T)$ if
$(u,v)\in {C^{2,1}((0,T)\times \R^n)}$ and satisfies \eqref{main-eq1}
(resp. \eqref{main-eq3}, \eqref{main-eq2}) for $(t,x)\in (0,T)\times \R^n$ in the classical sense. A classical solution of \eqref{main-eq1} (resp. \eqref{main-eq3}, \eqref{main-eq2})
$(0, T)$ is called nonnegative if  $u(t,x)\ge 0$ and $v(t,x)\ge 0$ for all
$(t,x)\in(0,T)\times\R^n$. A global classical solution of \eqref{main-eq1} (resp. \eqref{main-eq3}, \eqref{main-eq2}) is a classical solution on $(0,\infty)$.

Let
$$
C_{\rm unif}^b(\R^n)=\{u\in C(\R^n)\,|\, u\,\, \text{it uniformly continuous in}\,\, x\in\R^n\,\, \text{and}\,\, \sup_{x\in\R^n}|u(x)|<\infty\}
$$
equipped with the norm $\|u\|_{C_{\rm unif}^b(\R^n)}:=\|u\|_\infty=\sup_{x\in\R^n}|u(x)|$,  and
$$
C_{\rm unif}^{m,b}=\{u\in C_{\rm unif}^b(\R^n)\,|\, \frac{\partial^k u}{\partial x_{i_1} \partial x_{i_2}\cdots \partial x_{i_k}} \in C_{\rm unif}^b(\R^n)\,\,
{\rm for}\,\, k=1,2,\cdots, m,  1\le i_1,i_2,\cdots,i_k\le n\}
$$
equipped with the norm $\|u\|_{C_{\rm unif}^{m,b}}:=\|u\|_\infty+\sum_{k=1}^m \sum_{1\le i_1,i_2,\cdots,i_k\le n} \|\frac{\partial^k u}{\partial x_{i_1} \partial x_{i_2}\cdots \partial x_{i_k}}\|_\infty$, where $m\in\mathbb{N}$.

For given $0<\nu<1$ and $m\ge 1$, let
\begin{equation*}
\label{holder-cont-space}
C^{\nu, b}_{\rm unif}(\R^n)=\{u\in C_{\rm unif}^b(\R^n)\,|\, \sup_{x,y\in\R^n,x\not =y}\frac{|u(x)-u(y)|}{|x-y|^\nu}<\infty\}
\end{equation*}
with the norm $\|u\|_{\infty,\nu}=\sup_{x\in\R^n}|u(x)|+\sup_{x,y\in\R^n,x\not =y}\frac{|u(x)-u(y)|}{|x-y|^\nu}$, and
\begin{equation*}
C^{m,\nu,b}_{\rm unif}(\R^n)=\{u\in C_{\rm unif}^{m,b}(\R^n)\,|\,  \frac{\partial^m u}{\partial x_{i_1} \partial x_{i_2}\cdots \partial x_{i_m}}\in C^{\nu,b}_{\rm unif}(\R^n),\,\, 1\le i_1,i_2,\cdots,i_m\le n\}
\end{equation*}
with the norm $\|u\|_{m,\nu,b}=\|u\|_{C_{\rm unif}^{m,b}(\R^n)}+\sum_{1\le i_1,i_2,\cdots,i_m\le n}\| \frac{\partial^m u}{\partial x_{i_1} \partial x_{i_2}\cdots \partial x_{i_m}} \|_{C^{\nu,b}_{\rm unif}(\R^n)}$.

For $0<\theta<1$, let
{\small\begin{align*}
&C^{\theta}((t_1,t_2),C_{\rm unif}^{\nu,b}(\R^n))\\
&=\{ u(\cdot)\in C((t_1,t_2),C_{\rm unif}^{\nu,b}(\R^n))\,|\, u(t)\,\, \text{is locally H\"{o}lder continuous with exponent}\,\, \theta\}.
\end{align*}}

Considering chemotaxis models on the whole space $\R^n$, it is important to study the solutions with initial functions in  $L^q(\R^n)$  as well as  $C_{\rm unif}^b(\R^n)$.
 Note that, the local existence of classical solutions of
\eqref{main-eq1} (resp. \eqref{main-eq2}, \eqref{main-eq3}) with initial functions
in $L^q(\R^n)$  or  $C_{\rm unif}^b(\R^n)$ can be proved by quite standard methods.
The following results on the local existence of classical solutions of \eqref{main-eq1}, \eqref{main-eq2}, and \eqref{main-eq3} will be presented in subsection \ref{local-solu}.

\begin{itemize}

\item For given $u_0\in C_{\rm unif}^b(\R^n)$ and $v_0\in C_{\rm unif}^{1,b}(\R^n)$ or
$u_0\in L^q(\R^n)$ and $v_0\in W^{1,q}(\R^n)$ for some $q>n$ and $q\ge 2$, there is
$T_{\max}(u_0,v_0)\in (0,\infty]$ such that \eqref{main-eq1} (resp. \eqref{main-eq2}) has a unique classical solution, denoted by
$(u(t,x;u_0,v_0),v(t,x;u_0,v_0))$,  on $(0,T_{\max}(u_0,v_0))$ satisfying
$$
\lim_{t \to 0}\|u(t,\cdot;u_0)-u_0(\cdot)\|_X+\|v(t,\cdot;u_0,v_0)-v_0(\cdot)\|_Y=0,
$$
where $X=C_{\rm unif}^b(\R^n)$ and $Y=C_{\rm unif}^{1,b}(\R^n)$ or
$X=L^q(\R^n)$ and $Y=W^{1, {q}}(\R^n)$, moreover, if $T_{\max}(u_0,v_0)<\infty$, then
$$
\lim_{t \to T_{\max}(u_0,v_0)-}\|u(t,\cdot;u_0,v_0)\|_X+\|v(t,\cdot;u_0,v_0)\|_Y=\infty
$$
(see Proposition \ref{local-existence-prop1} for detail).

\item   For given $u_0\in C_{\rm unif}^b(\R^n)$  or
$u_0\in L^q(\R^n)$  for some $q>n$ and $q\ge 2$, there is
$T_{\max}(u_0)\in (0,\infty]$ such that \eqref{main-eq3} has a unique classical solution, denoted by,
$(u(t,x;u_0),v(t,x;u_0))$ on $(0,T_{\max}(u_0))$ satisfying
$$
\lim_{t \to 0}\|u(t,\cdot;u_0)-u_0(\cdot)\|_X=0,
$$
where $X=C_{\rm unif}^b(\R^n)$  or
$X=L^q(\R^n)$, moreover, if $T_{\max}(u_0)<\infty$, then
$$
\lim_{t \to T_{\max}(u_0)-}\|u(t,\cdot;u_0)\|_X=\infty
$$
(see Proposition \ref{local-existence-prop2} for detail).

\end{itemize}

\subsection{Statements of the main results}
\label{main-thms}

In this subsection, we state the main results of the paper.
Throughout this subsection, $q$ is a positive number satisfying $q>n$ and $q\ge 2$.
$(u(t,x;u_0,v_0)$, $v(t,x;u_0,v_0))$ denotes the classical solution of \eqref{main-eq1} or \eqref{main-eq2} on  the maximal interval $(0,T_{\max}(u_0,v_0))$ with  $u_0\in C_{\rm unif}^b(\R^n)$ and $v_0\in C_{\rm unif}^{1,b}(\R^n)$ or
$u_0\in L^q(\R^n)$ and $v_0\in W^{1,q}(\R^n)$, and
$u_0\ge 0$, $v_0\ge 0$. $(u(t,x;u_0),v(t,x;u_0))$ denotes the classical solution of \eqref{main-eq3} on the maximal interval $(0,T_{\max}(u_0))$ with   $u_0\in C_{\rm unif}^b(\R^n)$  or
 $u_0\in L^q(\R^n)$, and $u_0\ge 0$.

Our first main result provides a general condition for the global existence and boundedness of positive classical solutions of \eqref{main-eq1}-\eqref{main-eq3}.

\begin{tm}
\label{general-global-existence-thm}
Let  $B(x_0,1)$ be the  unit ball in $\R^n$ centered at $x_0$ with radius $1$.
Let
$$(u(t,x),v(t,x))=(u(t,x;u_0,v_0),v(t,x;u_0,v_0))$$
 or $$(u(t,x),v(t,x))=(u(t,x;u_0), v(t,x;u_0)),
$$
 and
$$T_{\max}=T_{\max}(u_0,v_0)\quad {\rm or}
\quad T_{\max}=T_{\max}(u_0).
$$
If there is $p>\max\{1,\frac{n}{2}\}$  such that
\begin{equation}
\label{general-cond-eq00}
\limsup_{t\to T_{\max}-} \,\, \sup_{x_0\in\R^n}\int_{B(x_0,1)}u^p(t, x)dx<\infty,
\end{equation}
then $T_{\max}=\infty$ and
$
\limsup_{t\to\infty}\|u(t,\cdot)\|_\infty<\infty.
$
\end{tm}

Next theorem provides sufficient conditions for the global existence and boundedeness of classical solutions of \eqref{main-eq1}.

\begin{tm}
\label{main-thm1}
Consider \eqref{main-eq1}.
For given  $u_0\in C_{\rm unif}^b(\R^n)$ and $v_0\in C_{\rm unif}^{1,b}(\R^n)$ or
 $u_0\in L^q(\R^n)$ and $v_0\in W^{1,q}(\R^n)$ and $u_0,v_0\ge 0$,  assume that
\begin{equation}
\label{cond-eq1}
|\chi|\cdot \|v_0\|_\infty <\max\Big\{  b \cdot {C^*_{n}},\, {D^*_{\tau,n}}\Big\} ,
\end{equation}
where
\begin{equation}
\label{c-tau-n-eq1}
{{ C^*_{n}}:=\sup_{\gamma>\max\{1,n/2\}}\frac{\gamma}{\gamma-1}\Big({ C_{\gamma+1,n}}\Big)^{-\frac{1}{\gamma+1}},}
\end{equation}
 $C_{\gamma+1,n}$ is {given in Lemma \ref{maximal-regularity-lm}}, 
and
\begin{equation}
\label{d-tau-n-eq}
D^*_{\tau,n}:=\begin{cases}  \frac{2}{\tau n^* }\left(2\sqrt{\frac{(\tau^*)^2}4+\frac1{\tau n^*}}+j|\tau^*|\right)^{-1}\quad &{\rm if}\quad \sqrt{\frac{(\tau^*)^2}4+\frac1{\tau n^*}}>-j|\tau^*|\cr\cr
\frac{2}{ \tau n^*}\left(\sqrt{\frac{(\tau^*)^2}4+\frac1{\tau n^*}}\right)^{-1}\quad &{\rm if}\quad \sqrt{\frac{(\tau^*)^2}4+\frac1{\tau n^*}}\leq -j|\tau^*|
\end{cases}
\end{equation}
with $n^*:=\max\{1,\frac{n}{2}\}$, $
\tau^*:=\frac{1}{\tau}-1$, and $j:=\text{\rm Sign}(\chi\tau^*)$ is the sign of $\chi\tau^*$.
 Then the classical solution
$(u(t,x;u_0,v_0)$, $v(t,x;u_0,v_0))$ of \eqref{main-eq1} exists for all $t>0$,  and  $\|u(t,\cdot;u_0,v_0)\|_\infty$ and $\|\nabla v(t,\cdot;u_0,v_0)\|_\infty$ stays bounded as $t\to\infty$.
\end{tm}

The following  theorem provides sufficient conditions for the global existence and boundedeness of classical solutions of \eqref{main-eq2}.

\begin{tm}
\label{main-thm2}
Consider \eqref{main-eq2}.
Suppose that $q>n$ and $q\ge 2$.
For given  $u_0\in C_{\rm unif}^b(\R^n)$ and $v_0\in C_{\rm unif}^{1,b}(\R^n)$ or
 $u_0\in L^q(\R^n)$ and $v_0\in W^{1,q}(\R^n)$,  if
\begin{equation}
\label{cond-eq2}
|\chi|\cdot  \mu < b\cdot { C^*_{n},}
\end{equation}
 then $T_{\max}(u_0,v_0)=\infty$ { and $\limsup_{t\to\infty} \|u(t,\cdot;u_0,v_0)\|_\infty<\infty$}, where { $C_n^*$ is as in \eqref{c-tau-n-eq1}.}
\end{tm}

Our last main theorem provides sufficient conditions for the global existence and boundedeness of classical solutions of \eqref{main-eq3}.

\begin{tm}
\label{main-thm3}
Consider \eqref{main-eq3}.
Suppose that $q>n$ and $q\ge 2$.
For given  $u_0\in C_{\rm unif}^b(\R^n)$  or
 $u_0\in L^q(\R^n)$,  if
\begin{equation}
\label{cond-eq3}
\chi \cdot \mu<b\cdot \frac{n}{(n-2)_+},
\end{equation}
 then for any $u_0\in C_{\rm unif}^b(\R^n)$ with $u_0\ge 0$,
$T_{\max}(u_0)=\infty$ and $\limsup_{t\to\infty}\|u(t,\cdot;u_0)\|_\infty<\infty$.
\end{tm}

\subsection{Remarks}
\label{remarks}

\begin{rk}
\label{general-rk1}
{\rm  Consider \eqref{special-eq1}-\eqref{special-eq3} on a bounded domain $\Omega$ complemented with Neumann boundary condition \eqref{special-boundary-cond}. By the Gagliardo-Nirenberg inequality,  it can be proved that the $L^p$-boundedness of a classical solution of \eqref{special-eq1}+\eqref{special-boundary-cond} (resp. \eqref{special-eq2}+\eqref{special-boundary-cond}, \eqref{special-eq3}+\eqref{special-boundary-cond})
for some $p>\max\{1,\frac{n}{2}\}$ implies the $L^\gamma$-boundedness for
 some ${ \gamma>n}$. Then by  certain iteration procedure  (see \cite{Ali}) or semigroup theory, it can be proved that the solution exists globally and stays bounded in $L^\infty$-norm as $t\to\infty$. Hence to prove the global existence and uniform  boundedness of a positive classical solution of \eqref{special-eq1}+\eqref{special-boundary-cond} (resp. \eqref{special-eq2}+\eqref{special-boundary-cond}, \eqref{special-eq3}+\eqref{special-boundary-cond}), it suffices to prove the boundedness of its $L^p$-norm for some $p>\max\{1,\frac{n}{2}\}$.

Consider \eqref{special-eq1}-\eqref{special-eq3} on the whole space $\Omega=\R^n$.
A positive classical solution may not be in $L^p(\R^n)$ for any $1\le p<\infty$, and even it is in $L^p(\R^n)$, its $L^p$-norm cannot stay bounded. It is discovered in this paper that
the boundedness of $L^p_{\rm local}$-norm of a classical solution of \eqref{main-eq1} (resp. \eqref{main-eq2}, \eqref{main-eq3}) for some $p>\max\{1,\frac{n}{2}\}$  implies the global existence and uniform boundedness of the solution (see Theorem \ref{general-global-existence-thm}). This is a novel and important result and implies that,  to prove the global existence and uniform  boundedness of a positive classical solution of  \eqref{main-eq1} (resp. \eqref{main-eq2}, \eqref{main-eq3}),
 it suffices to prove the boundedness of its $L^p_{\rm local}$-norm for some $p>\max\{1,\frac{n}{2}\}$.

 We point out that the condition \eqref{general-cond-eq00}  is    a necessary and sufficient condition for the unique classical solution of \eqref{main-eq1} (resp. \eqref{main-eq2}, \eqref{main-eq3}) exists globally and stays bounded in $L^\infty$-norm as $t\to\infty$.
}
\end{rk}

\begin{rk}
\label{general-rk2}
{\rm
Suppose that $(u(t,x),v(t,x))$ is a positive classical solution of  \eqref{main-eq1} (resp. \eqref{main-eq2}, \eqref{main-eq3}) on the maximal $(0,T_{\max})$. Then $(u(t,x),v(t,x))$ satisfies \eqref{special-eq1} (resp. \eqref{special-eq2}, \eqref{special-eq3}) with {$\Omega=B(x_0,1)$}  for any $x_0\in\R^n$. But $(u,v)$ may not satisfy the Neumann boundary condition \eqref{special-boundary-cond} with $\Omega=B(x_0,1)$.
Because of this, it is  nontrivial to estimate the $L^p_{\rm loc}$-norm of the solution.
Several novel techniques/ideas are developed in this paper to obtain the boundedness of $L^p_{\rm local}$-norm of positive classical solutions of  \eqref{main-eq1} (resp. \eqref{main-eq2}, \eqref{main-eq3})  for some $p>\max\{1,\frac{n}{2}\}$ and to prove that  the boundedness of $L^p_{\rm local}$-norm for some $p>\max\{1,\frac{n}{2}\}$ implies the global existence and uniform boundedness.  For example,
instead of estimating the $L^p_{\rm loc}$-norm, we introduce the idea to
estimate the $L^p$-norm  of $u(t,x)\psi(x)$  for some  $C^\infty$ exponentially decay  function $\psi(x)$ (see Lemma \ref{psi-lm} for such functions).
}
\end{rk}

\begin{rk}
\label{general-rk3} {\rm
Note that
$$
{ C_{n}^*=}\sup_{\gamma>\max\{1,\frac{n}{2}\}}\frac{\gamma}{\gamma-1}\Big(C_{{ \gamma+1,n}}\Big)^{-\frac{1}{\gamma+1}}\ge \frac{n}{(n-2)_+}\Big(C_{\max\{1,\frac{n}{2}\}+1,n}\Big)^{-\frac{1}{\max\{1,\frac{n}{2}\}+1}}.
$$
Hence when $n=1,2$, \eqref{cond-eq1}, \eqref{cond-eq2}, and \eqref{cond-eq3}  are satisfied for any $\chi$ and any $v_0$. This implies that  finite-time blow-up does not occur in \eqref{main-eq1}, \eqref{main-eq2}, and \eqref{main-eq3} when $n=1,2$.
As it is mentioned in the above, such result is known for \eqref{special-eq1}-\eqref{special-eq3} when $\Omega$ is bounded and complemented with Neumann boundary condition.
 It is proved for the first time  in this paper for \eqref{special-eq1}-\eqref{special-eq3} when $\Omega=\R^n$ and initial functions may not be integrable.  }
\end{rk}

\begin{rk}
\label{main-eq1-rk}
{\rm The global existence and boundedness of positive classical solutions of
\eqref{main-eq1} is studied for the first time in this paper.
Note that the condition \eqref{cond-eq1} is equivalent to
\begin{equation}
\label{equivalent-cond-eq1-1}
|\chi|\cdot \|v_0\|_\infty < b \cdot {C^*_{n}}=b\cdot \sup _{\gamma >\max\{1,\frac{n}{2}\}}\frac{\gamma }{\gamma-1} \big(C_{{\gamma +1,n}}\big)^{-\frac{1}{\gamma+1}}
\end{equation}
or
\begin{equation}
\label{equivalent-cond-eq1-2}
|\chi|\cdot \|v_0\|_\infty < {D^*_{\tau,n}}
\end{equation}
holds.
When $n=1,2$,  condition\eqref{equivalent-cond-eq1-1} holds for any $\chi$ and $v_0$.
For fixed $\chi$ and $v_0$, \eqref{equivalent-cond-eq1-1} holds  when $b$ is sufficiently large. {The condition \eqref{equivalent-cond-eq1-1} is independent of $\tau$.}
But it is difficult to provide an explicit formula for the constant  $C^*_{n}$ in \eqref{equivalent-cond-eq1-1}. The constant  $D^*_{\tau,n}$ in \eqref{equivalent-cond-eq1-2}  is given explicitly, which is of great importance in application. But  it is not true that, when $n=1,2$,  condition \eqref{equivalent-cond-eq1-2}
holds for any $\chi$ and $v_0$.  Hence conditions \eqref{equivalent-cond-eq1-1}  and
\eqref{equivalent-cond-eq1-2} complement each other.

Note also that our method can be adapted to the bounded domain case without much effort and  several existing results for bounded domains are then improved. For example,  by the similar arguments of Theorem \ref{main-thm1}, it can be proved that \eqref{cond-eq1} is also a sufficient condition for the global existence and boundedness of the classical solution of \eqref{special-eq1} when $\Omega$ is bounded and complemented with the boundary condition \eqref{special-boundary-cond}.
{When $\tau=1$,
$$
D^*_{1, n}=\sqrt  {{2}/{n}}.
$$
The condition }
\eqref{cond-eq1} is weaker than \eqref{bounded-domain-cond1}. Hence \cite[Theorem 3.3]{WaKhKh} is improved.

It is seen that
$$
D^*_{\tau,n}=
\begin{cases}
O(\frac{1}{\sqrt n})\quad  {\rm as}\quad n\to\infty \quad   {\rm  when }\quad \tau=1\cr\cr
O(\frac{1}{n})\quad {\rm as }\quad  n\to\infty \quad {\rm when}\quad \tau\not =1.
\end{cases}
$$
Hence there is some essential difference between the decay rates of $D^*_{\tau,n}$ as $n\to\infty$ for $\tau=1$ and $\tau\not =1$.

In the appendix, we will provide some discussions on the decay rate of $C_{n}^*$, {namely we show that $C_{n}^*$ decays to $0$ at most exponentially as $n\to\infty$ in Theorem \ref{T.A.1}}.
}
\end{rk}

\begin{rk}
\label{main-eq2-rk}
{\rm
Consider \eqref{main-eq2}.  When $\tau=1$ and $\chi>0$, global existence and boundedness of classical solutions of \eqref{main-eq2} and asymptotic behavior of globally defined bounded solutions
are studied in \cite{ShXu1, ShXu2}. Among others, it is proved  that the unique classical solution of \eqref{main-eq2} with $\tau=1$, $\chi>0$, and $(u_0,v_0)\in C_{\rm unif}^b(\R^n)\times C_{\rm unif}^{1,b}(\R^n)$ ($u_0,v_0\ge 0$) exists globally and stays bounded as $t\to\infty$ provided that
$$
\chi\mu <{4b}/{n}
$$
(see \cite[Theorem 1.2]{ShXu1}).
The global existence and boundedness of positive classical solutions  of \eqref{main-eq2} with $\tau\not =1$ are studied for the first time in this paper.}
\end{rk}

\begin{rk}
\label{main-eq3-rk}
{\rm  Consider \eqref{main-eq3}.
The works \cite{SaSh0, SaSh1, SaSh2, SaShXu} carried out some systematic study on the global existence of positive classical solutions and asymptotic dynamics of globally defined
positive solutions. It is proved in \cite{SaSh0}  that  positive classical solutions of \eqref{main-eq3} with initial functions in $L^q(\R^n)$ for some $q>\max\{n,2\}$ exist globally
provided that { $\chi>0$ and} \eqref{cond-eq3} holds. {It is proved in \cite{SaSh2} that   positive classical solutions of \eqref{main-eq3} with initial functions in $C_{\rm unif}^b(\R^n)$  exist globally and stay bounded for any $\chi<0$ (see \cite[Theorem A]{SaSh2}, which also implies that  positive classical solutions of \eqref{main-eq3} with initial functions in  $L^q(\R^n)$ for some $q>\max\{n,2\}$  exist globally and stay bounded for any $\chi<0$).}
But it remains open whether globally defined positive solutions  \eqref{main-eq3} with initial functions in $L^p$
are uniformly bounded {for all time} and whether  positive classical solutions of \eqref{main-eq3} with initial functions in $C_{\rm unfi}^b(\R^n)$ exist globally
provided that  $\chi>0$ and \eqref{cond-eq3} holds. Theorem \ref{main-thm3} solves this open problems.
}
\end{rk}

\begin{rk}
\label{main-eq2-eq3-rk}
{\rm
Formally, as $\tau\to 0$, \eqref{main-eq2} becomes \eqref{main-eq3}. Theorem \ref{general-global-existence-thm} provides a unified condition for the global existence and boundedness of positive classical solutions of \eqref{main-eq2} and \eqref{main-eq3}.
But it is difficult to identify parameter regions for the global existence of solutions of \eqref{main-eq2} and \eqref{main-eq3} in a same way.
We  find the sufficient conditions  \eqref{cond-eq2} and \eqref{cond-eq3}  for the global existence of classical solutions of \eqref{main-eq2} and \eqref{main-eq3}, respectively,  by different approaches. Both  \eqref{cond-eq2} and \eqref{cond-eq3}  indicate that $\chi\mu$ is small relative to the logistic damping $b$.
But, when $\tau\to 0$,  \eqref{cond-eq2} does not become \eqref{cond-eq3}.}
\end{rk}

The rest of this paper is organized as follows. We provide some preliminary lemmas and propositions in Section 2 for the use in later sections. In Section 3, we  investigate the global existence of classical solutions of \eqref{main-eq1}, \eqref{main-eq2}, and \eqref{main-eq3} in a unified way, and prove Theorem \ref{general-global-existence-thm}.
We prove Theorems \ref{main-thm1}, \ref{main-thm2}, and \ref{main-thm3} in Sections 4, 5, and 6, respectively.  { In the appendix, we provide some discussions on
  the dimension $n$ and $\gamma$ dependencies of the constant $C_{\gamma,n}$  associated to the maximal regularity for parabolic equations in Lemma \ref{maximal-regularity-lm}.}

\section{Preliminary}

In this section, we present  some preliminary  lemmas and propositions to be used
   in the next sections, including basic properties of the analytic semigroup generated by $-\Delta+I$ on $C_{\rm unif}^b(\R^n)$ or $L^p(\R^n)$;  local existence  of classical solutions of \eqref{main-eq1}, \eqref{main-eq2}, and \eqref{main-eq3};  maximal regularity for parabolic initial-boundary
value problems; and a useful  exponentially decaying function.

\subsection{The analytic semigroup generated by $-\Delta+I$ on $C_{\rm unif}^b(\R^n)$}

In this subsection, we present some basic properties for the  analytic semigroup, denoted by $T(t)$,  generated by $-\Delta+I$ on $X:=C_{\rm unif}^b(\R^n)$.
  Observe  that
\begin{equation}
\label{semigroup-eq-0}
(T(t)u)(x)=e^{-t}(G(\cdot,t)\ast u)(x)= \int_{\R^{n}}e^{-t}G(x-y,t)u(y)dy
\end{equation}
for every $u\in X$, $t> 0$,  and $x\in\R^n$, where  $G(t,x)$ is the heat kernel defined by
\begin{equation}
\label{heat-kernel}
G(t,x)={(4\pi t)^{-\frac{n}{2}}}e^{-\frac{|x|^{2}}{4t}}.
\end{equation}

Let
 $X^{\alpha}={\rm Dom}(A^{\alpha})$ be the fractional power spaces associated with $A=I-\Delta$ on $X$  ($\alpha\in [0,\infty)$).
 We have $X^0=C_{\rm unif}^b(\R^n)$ and $X^1=C_{\rm unif}^{2,b}(\R^n)$
and
  the following continuous embeddings
\begin{eqnarray}\label{Fractional power Imbedding-0}
X^{\alpha} \subset C_{\rm unif}^{\lfloor \nu \rfloor, \nu -\lfloor \nu \rfloor,b }(\R^n) \quad \text{if} \quad 0\leq \nu < 2\alpha
\end{eqnarray}
(see \cite[Exercise 9]{Dan Henry}).
Furthermore, for $0<\delta<1$ and $\alpha\ge 0$, there is constant  $C_{\alpha}$,
such that
 \begin{equation}
\label{L-infinity-estimate-1}
 \|A^\alpha T(t)u \|_{C_{\rm unif}^b(\mathbb{R}^{n})}\leq C_\alpha  t^{-\alpha }e^{-(1-\delta)t}
\|u\|_{C_{\rm unif}^b (\mathbb{R}^{n})}
 \end{equation}
 for every $u\in C_{\rm unif}^b(\R^n)$ and  $t>0$ (see \cite[Theorem 1.4.3]{Dan Henry}).

\begin{lem}
(\cite[Lemma 3.2]{SaSh0})\label{L-infinity-bound}
 For every $t>0$, the operator $T(t)\nabla \cdot$ has a unique bounded extension on
$\Big(C_{\rm unif}^b(\R^n)\Big)^n$ satisfying
\begin{equation}
\label{L-infinity-estimate-2}
\|T(t)\nabla \cdot  u\|_{C_{\rm unif}^b(\R^{n})}\leq \frac{n}{\sqrt\pi}t^{-\frac{1}{2}}e^{-t}\|u\|_{C_{\rm unif}^b (\R^{n})} \ \ \ \forall \ u\in\big( C_{\rm unif}^b (\R^{n})\big)^n, \,\, \ \forall \ t>0.
\end{equation}
\end{lem}

\subsection{The analytic semigroup generated by $-\Delta+I$ on $L^p(\mathbb{R}^n)$}

In this subsection, we present some basic properties of the analytic semigroup, denoted by $T_p(t)$,  generated by $A_p:=-\Delta+I$ on $X_p:=L^p(\mathbb{R}^n)$ $(p\ge 1$).
  We also have
\begin{equation}
\label{semigroup-eq}
(T_p(t)u)(x)=e^{-t}(G(\cdot,t)\ast u)(x)= \int_{\R^{n}}e^{-t}G(x-y,t)u(y)dy
\end{equation}
for every $u\in X_p$, $t> 0$,  and $x\in\R^n$, where  $G(t,x)$ is the heat kernel defined
in \eqref{heat-kernel}.  It follows from  the $L^p-L^q$
estimates for the convolution product   that there is  $C_{p, q}>0$ ($1\le p<q<\infty$) such that
\begin{equation}\label{Lp Estimates-2}
 \| T_p(t)u \|_{L^{q}(\mathbb{R}^{n})}\leq C_{p,q} t^{-(\frac{1}{p}-\frac{1}{q})\frac{n}{2}}e^{-t}\|u\|_{L^{p}(\mathbb{R}^{n})},
\end{equation}
and
\begin{equation}\label{Lp Estimates-3}
 \| \nabla T_p(t)u \|_{L^{q}(\mathbb{R}^{n})}\leq C_{p,q} t^{-\frac{1}{2}-(\frac{1}{p}-\frac{1}{q})\frac{n}{2}}e^{- t}\|u\|_{L^{p}(\mathbb{R}^{n})},
\end{equation}
 for every $u\in L^p(\R^n)$ and $t>0$.

\smallskip

Let
 $X_p^{\alpha}={\rm Dom}(A_p^{\alpha})$ be the fractional power spaces associated with $A_p=I-\Delta$ on $X_p$  ($\alpha\in [0,\infty)$).
 We have $X_p^0=X_p=L^p(\R^n)$ and $X_p^1=W^{2,p}(\R^n)$ ($1\leq p<\infty$)
and
  the following continuous embeddings
(see \cite[Theorem 1.6.1]{Dan Henry})
\begin{equation}
\label{Fractional power Imbedding-1}
X_p^{\alpha} \subset C_{\rm unif}^{\lfloor \nu \rfloor, \nu -\lfloor \nu \rfloor,b }(\R^n)  \quad \text{if} \quad 0\leq \nu < 2\alpha -\frac{n}{p},
\end{equation}
\begin{equation}
\label{Fractional power Imbedding-2}
X_p^{\alpha}\subset W^{1,q}(\R^n) \quad \text{if}\ \alpha>\frac{1}{2} \ \text{and}\ \frac{1}{q}>\frac{1}{p}-\frac{(2\alpha-1)}{n},
\end{equation}
and
\begin{equation}
\label{Fractional power Imbedding-3}
X_p^{\alpha}\subset L^{q}(\R^n)\quad \text{if} \ \frac{1}{q}>\frac{1}{p}-\frac{2\alpha}{n}, \ \ q\geq p.
\end{equation}
Furthermore, for $0<\delta<1$ and $\alpha\ge 0$, there is  $C_{\alpha,p,q}>0$  such that
 \begin{equation}\label{Lp Estimates}
 \|A_p^{\alpha} T_p(t)u \|_{L^q}\leq C_{\alpha,p,q} t^{-\alpha-(\frac{1}{p}-\frac{1}{q})\frac{n}{2}}e^{-(1-\delta)t}\|u\|_{L^{p}(\mathbb{R}^{n})} \ \ \text{for}\ \ 1\leq p\leq q{<} +\infty
 \end{equation}
 for every $u\in L^p(\R^n)$ and $t>0$ (see \cite[(2.12)]{SaSh0}).


\begin{lem}(\cite[Lemma 3.1]{SaSh0})\label{L_pbound} Let $p\in[1,\infty)$ and  $\{T_p(t)\}_{t>0}$ be the semigroup in \eqref{semigroup-eq} generated by $A_p$ on $L^p(\R^n)$. For every $t>0$, the operator $T_p(t) \nabla \cdot $ has a unique bounded extension on $\big(L^{p}(\R^{n})\big)^n$ satisfying
\begin{equation}\label{2.3}
\|T_p(t)\nabla \cdot  u\|_{L^{p}(\R^{n})}\leq \tilde C_{p}t^{-\frac{1}{2}}e^{-t}\|u\|_{L^{p}(\R^{n})} \ \ \ \forall \ u\in\big( L^{p}(\R^{n})\big)^n, \ \forall \ t>0,
\end{equation}
 where $\tilde C_p$ depends only on $p$ and $n$.
 Furthermore, for every $q\in [p , \infty]$, we have that $T(t)\nabla\cdot  u\in L^{q}(\R^n)$ with
\begin{equation} \label{2.4}
\|T_p(t)\nabla \cdot u\|_{L^{q}}\leq \tilde C_{p,q}t^{-\frac{1}{2}-\frac{n}{2}(\frac{1}{p}-\frac{1}{q})}e^{-t}\|u\|_{L^{p}(\R^{n})} \ \ \ \forall \ u\in \big(L^{p}(\R^{n})\big)^n, \ \forall \ t>0,
\end{equation}
where $\tilde C_{p,q}$ is constant depending only on $n$, $q$ and $p$.
\end{lem}

\subsection{Local existence of classical solutions}
\label{local-solu}

In this subsection, we study the local existence and uniqueness of classical solutions of \eqref{main-eq1},  \eqref{main-eq2}, and \eqref{main-eq3}.
We have the following propositions on the local existence of classical solutions of \eqref{main-eq1},
\eqref{main-eq2}, and \eqref{main-eq3} with proper initial conditions.

\begin{prop}
\label{local-existence-prop1}
Consider \eqref{main-eq1}  (resp. \eqref{main-eq2}).
\smallskip

\noindent (1)   For given  $u_0\in C_{\rm unif}^b(\R^n)$ and $v_0\in C_{\rm unif}^{1,b}(\R^n)$,  there is $T_{\max}:=T_{\max}(u_0,v_0)\in (0,\infty]$ such that \eqref{main-eq1} (resp. \eqref{main-eq2}) has a unique classical solution
$(u(t,x;u_0,v_0),v(t,x;u_0,v_0))$ on $(0,T_{\max}(u_0,v_0))$  satisfying
\begin{equation}
\label{initial-eq1}
\lim_{t\to 0+}\Big(\|u(t,\cdot;u_0,v_0)-u_0(\cdot)\|_{\infty}+\|v(t,\cdot;u_0,v_0)-v_0(\cdot)\|_{\infty}+\|\nabla v(t,\cdot;u_0,v_0)-\nabla v_0(\cdot)\|_{\infty}\Big)=0,
\end{equation}
\begin{equation}
\label{local-1-eq1}
u(\cdot,\cdot;u_0, v_0) \in C([0, T_{\max} ), C_{\rm unif}^b(\R^n) )\cap C^1((0,T_{\max}),C_{\rm unif}^b(\R^n)),
\end{equation}

\begin{equation}
\label{local-1-eq1-1}
v(\cdot,\cdot;u_0, v_0) \in C([0, T_{\max} ), C_{\rm unif}^{1,b}(\R^n) )\cap C^1((0,T_{\max}),C_{\rm unif}^{1,b}(\R^n)),
\end{equation}

\begin{equation}
\label{local-1-eq2}
u(\cdot,\cdot;u_0, v_0),\, \partial_{x_i} u(\cdot,\cdot;u_0, v_0)\in C^\theta((0,T_{\max}),C^{\nu,b}_{\rm unif}(\R^n)),
\end{equation}

\begin{equation}
\label{local-1-eq2-1-1}
\partial^2_{x_i x_j} u(\cdot,\cdot;u_0, v_0),\, \partial_t u(\cdot,\cdot;u_0, v_0)\in C^\theta((0,T_{\max}),C^{\nu,b}_{\rm unif}(\R^n)),
\end{equation}

\begin{equation}
\label{local-1-eq2-1}
v(\cdot,\cdot;u_0, v_0),\,\, \partial_{x_i} v(\cdot,\cdot;u_0, v_0)\in C^\theta((0,T_{\max}),C^{\nu,b}_{\rm unif}(\R^n))
\end{equation}

\begin{equation}
\label{local-1-eq2-1-1-1}
\partial^2_{x_i x_j} v(\cdot,\cdot;u_0, v_0),\,\, \partial_t v(\cdot,\cdot;u_0, v_0)\in C^\theta((0,T_{\max}),C^{\nu,b}_{\rm unif}(\R^n))
\end{equation}
for all $i,j=1,2,\cdots,n$, $0<\theta\ll 1$, and  $0<\nu\ll 1$.
Moreover,  if $T_{\max}(u_0,v_0)<\infty$, then
\begin{equation}
\label{final-eq1}
\lim_{t\to T_{\max}(u_0,v_0)-}\Big(\|u(t,\cdot;u_0,v_0)\|_{\infty}+\|v(t,\cdot;u_0,v_0)\|_{\infty}+ \|\nabla v(t,\cdot;u_0,v_0)\|_{\infty}\Big)=\infty.
\end{equation}
If $u_0(x)\ge 0$ and $v_0(x)\ge 0$ for $x\in\R^n$, then $u(t,x;u_0,v_0)\ge 0$ and $v(t,x;u_0,v_0)\ge 0$ for $t\in [0,T_{\max}(u_0,v_0))$ and $x\in\R^n$.

\medskip

\noindent (2)  Suppose that $q>n$ and $q\ge 2$.  For given $u_0\in L^q(\R^n)$ and $v_0\in W^{1,q}(\R^n)$, there is $T_{\max}(u_0,v_0)\in (0,\infty]$ such that \eqref{main-eq1}
(resp. \eqref{main-eq2})  has a unique classical solution
$(u(t,x;u_0,v_0)$, $v(t,x;u_0,v_0))$  on $(0,T_{\max}(u_0,v_0))$
satisfying
\begin{equation}
\label{initial-eq2}
\lim_{t\to 0+}(\|u(t,\cdot;u_0,v_0)-u_0(\cdot)\|_{L^q}+\|v(t,\cdot;u_0,v_0)-v_0(\cdot)\|_{L^q}+\|\nabla v(t,\cdot;u_0,v_0)-\nabla v_0(\cdot)\|_{L^q})=0,
\end{equation}
\begin{equation}
\label{local-2-eq1}
u(\cdot,\cdot;u_0, v_0) \in C([0, T_{\max} ), L^q(\R^n) )\cap { C((0, T_{\max} ), C_{\rm unif}^b(\R^n) )\cap C^1((0,T_{\max}),C_{\rm unif}^b(\R^n))}
\end{equation}
\smallskip
\begin{equation}
\label{local-2-eq1-1}
v(\cdot,\cdot;u_0, v_0) \in C([0, T_{\max} ), W^{1,q}(\R^n))\cap { C((0, T_{\max} ), C_{\rm unif}^{1,b}(\R^n) )\cap C^1((0,T_{\max}),C_{\rm unif}^{1,b}(\R^n)),}
\end{equation}
and \eqref{local-1-eq2}-\eqref{local-1-eq2-1-1-1}.
Moreover,  if $T_{\max}(u_0,v_0)<\infty$, then
\begin{equation}
\label{final-eq2}
\lim_{t\to T_{\max}(u_0,v_0)-}(\|u(t,\cdot;u_0,v_0)\|_{L^q}+\|v(t,\cdot;u_0,v_0)\|_{L^q}+ \|\nabla v(t,\cdot;u_0,v_0)\|_{L^q})=\infty.
\end{equation}
If $u_0(x)\ge 0$ and $v_0(x)\ge 0$ for $x\in\R^n$, then $u(t,x;u_0,v_0)\ge 0$ and $v(t,x;u_0,v_0)\ge 0$ for $t\in [0,T_{\max}(u_0,v_0))$ and $x\in\R^n$.
\end{prop}

\begin{prop}
\label{local-existence-prop2}
Consider \eqref{main-eq3}.
\smallskip

\noindent (1)  For given  $u_0\in C_{\rm unif}^b(\R^n)$,  there is $T_{\max}:=T_{\max}(u_0)\in (0,\infty]$ such that \eqref{main-eq3} has a unique classical solution
$(u(t,x;u_0),v(t,x;u_0))$  on $(0,T_{\max}(u_0))$
satisfying
$$
\lim_{t\to 0+}
\|u(t,\cdot;u_0)-u_0(\cdot)\|_{L^\infty}=0,
$$
and
\begin{equation}
\label{local-3-eq1}
u(\cdot,\cdot;u_0) \in C([0, T_{\max} ), C_{\rm unif}^b(\R^n) )\cap C^1((0,T_{\max}),C_{\rm unif}^b(\R^n)),
\end{equation}

\begin{equation}
\label{local-3-eq2}
u(\cdot,\cdot;u_0),\, \partial_{x_i} u(\cdot,\cdot;u_0)\in C^\theta((0,T_{\max}),C^{\nu,b}_{\rm unif}(\R^n)),
\end{equation}

\begin{equation}
\label{local-3-eq3}
\partial^2_{x_i x_j} u(\cdot,\cdot;u_0),\, \partial_t u(\cdot,\cdot;u_0)\in C^\theta((0,T_{\max}),C^{\nu,b}_{\rm unif}(\R^n))
\end{equation}
for all $i,j=1,2,\cdots,n$, $0<\theta\ll 1$, and  $0<\nu\ll 1$.
Moreover,
 if $T_{\max}(u_0)<\infty$, then
$$
\lim_{t\to T_{\max}(u_0)-}\|u(t,\cdot;u_0)\|_{L^\infty}=\infty.
$$
If $u_0(x)\ge 0$ for $x\in\R^n$, then $u(t,x;u_0)\ge 0$ and $v(t,x;u_0)\ge 0$ for $t\in [0,T_{\max}(u_0,v_0))$ and $x\in\R^n$.

\medskip

\noindent (2)  Suppose that $q>n$ and $q\ge 2$. For given $u_0\in L^q(\R^n)$, there is $T_{\max}(u_0)\in (0,\infty]$ such that \eqref{main-eq3} has a unique classical solution
$(u(t,x;u_0),v(t,x;u_0))$  on $(0,T_{\max}(u_0))$ satisfying
$$
\lim_{t\to 0+}\|u(t,\cdot;u_0)-u_0(\cdot)\|_{L^q}=0,
$$
\begin{equation}
\label{local-4-eq1}
u(\cdot,\cdot;u_0) \in C([0, T_{\max} ), L^q(\R^n) )\cap  { C((0, T_{\max} ), C_{\rm unif}^b(\R^n) )\cap C^1((0,T_{\max}),C_{\rm unif}^b(\R^n)),}
\end{equation}
and \eqref{local-3-eq2}, \eqref{local-3-eq3}.
Moreover,  if $T_{\max}(u_0)<\infty$, then
$$
\lim_{t\to T_{\max}(u_0)-} \|u(t,\cdot;u_0,v_0)\|_{L^q}=\infty.
$$
If $u_0(x)\ge 0$ for $x\in\R^n$, then $u(t,x;u_0)\ge 0$ and $v(t,x;u_0)\ge 0$ for $t\in [0,T_{\max}(u_0,v_0))$ and $x\in\R^n$.
\end{prop}

Proposition \ref{local-existence-prop2}   follows from \cite[Theorem 1.1, Theorem 1.3]{SaSh0}.  Proposition \ref{local-existence-prop1} can be proved by the similar arguments of
\cite[Theorem 1.1, Theorem 1.3]{SaSh0}. In the following, we provide the outline of the proof of Proposition \ref{local-existence-prop1}.

\begin{proof}[Outline of the proof of Proposition \ref{local-existence-prop1}]

$\quad$
\smallskip

\noindent (1) We provide the outline of the proof of Proposition \ref{local-existence-prop1}(1) for \eqref{main-eq1} in the following five  steps. The proof of   Proposition \ref{local-existence-prop1}(1) for \eqref{main-eq2} follows the same steps.

\smallskip

\noindent {\bf Step 1.} In this step, we prove that \eqref{main-eq1} has a  unique mild solution
$(u(t,x;u_0,v_0),v(t,x;u_0,v_0))$  on $[0,T]$  satisfying \eqref{initial-eq1} for some $T>0$,
that is,
 $$[0,T]\ni t\mapsto (u(t,\cdot;u_0,v_0),v(t,\cdot;u_0,v_0))\in C_{\rm unif}^b(\R^n)\times C_{\rm unif}^{1,b}(\R^n)\quad \text{is continuous},
$$
 and
\begin{equation}
\label{mild-solution-eq1}
\begin{cases}
u(t,\cdot;u_0,v_0)=T(t)u_0-\chi \int_0^t T(t-s)\nabla \cdot (u(s,\cdot;u_0,v_0)\nabla v(s,\cdot;u_0,v_0))ds\cr
\qquad\qquad \quad\quad\,\,  +\int_0^t T(t-s) u(s,\cdot;u_0,v_0)\big(1+a-b u(s,\cdot;u_0,v_0)\big)ds,\cr
v(t,\cdot;u_0,v_0)=T(\frac{t}{\tau})v_0 + \frac{1}{\tau}\int_0^t T(\frac{t-s}{\tau}) \big(1-u(s,\cdot;u_0,v_0) \big)v(s,\cdot;u_0,v_0) ds
\end{cases}
\end{equation}
for $t\in [0,T]$.

\smallskip

To this end,  first, for given $T>0$ and $R>0$,  let
$$
\mathcal{X}_T= C([0,T],  C^b_{\rm unif}(\R^n))\times C([0,T], C_{\rm unif}^{1,b}(\R^n))
$$
be equipped with the sup-norm, and
    $$
\mathcal{S}_{R,T}=\{(u,v) \in \mathcal{X}_T\, :\,  \sup_{s\in [0,T]}\|u(s)\|_{C^b_{\rm unif}(\R^n)} \leq R,  \sup_{s\in [0,T]}\|v(s)\|_{C_{\rm unif}^{1,b}(\R^n)} \leq R\}.
$$ Then $\mathcal{S}_{R,T}$ is a closed subset of $\mathcal{X}_T$. Define a map  on $\mathcal{S}_{R,T}$ as follows
\begin{align*}
\Psi (u,v)(t)& = \begin{pmatrix}
 \Psi_1(u,v)(t)\\
\Psi_2(u,v)(t)
\end{pmatrix}\\
& = \begin{pmatrix}
 T(t)u_0 -\chi\int_0^t T(t-s)\nabla \cdot (u(s)\nabla v(s))ds + \int_0^t T(t-s)u(s)(1+a-bu(s))ds\\
T(\frac{t}{\tau})v_0 +\frac{1}{\tau} \int_0^t T(\frac{t-s}{\tau})(1-u(s))v(s) ds.
\end{pmatrix}.
\end{align*}

Next, we show that $\Psi$ is a  well defined map from $\mathcal{S}_{R,T}$ to $\mathcal{X}_T$. Moreover, for any given $0<\beta<\frac{1}{2}$ and  $(u,v)\in \mathcal{S}_{R,T}$, $(0,T]\ni t \mapsto \Psi_1(u,v)(t)\in X^\beta$ and $(0,T]\ni t\mapsto \Psi_2(u,v)(t)\in X^{2\beta}$ are locally H\"older continuous.

We then prove that for given $R>\max\{\|u_0\|_{C_{\rm unif}^b(\R^n)},\|v_0\}_{C_{\rm unif}^{1,b}(\R^n)}\}$, there is $T:=T(R)>0$ such that
 the map $\Psi$ maps $\mathcal{S}_{R,T}$  into itself  and  is a contraction map. Hence
$\Psi$ has a unique fixed point $u(\cdot,\cdot;u_0,v_0),v(\cdot,\cdot;u_0,v_0) \in \mathcal{S}_{R,T}$ and then  \eqref{main-eq1} has a   unique mild solution
$(u(t,x;u_0,v_0)$, $v(t,x;u_0,v_0))$  on $[0,T]$  satisfying \eqref{initial-eq1} for some $T>0$.
Moreover, $(0,T]\ni t\mapsto  u(t,\cdot;u_0,v_0)\in X^\beta$ and $(0,T]\ni t\mapsto v(t,\cdot;u_0,v_0)\in X^{2\beta}$ is locally H\"older continuous for any $0<\beta<\frac{1}{2}$.

\smallskip

\noindent {\bf Step 2.} In this step, by the standard extension arguments, we prove that there is
$T_{\max}=T_{\max}(u_0,v_0)\in (0,\infty]$ such that \eqref{main-eq1} has a unique mild solution  $(u(t,x;u_0,v_0),v(t,x;u_0,v_0))$ satisfying \eqref{mild-solution-eq1} for
$t\in [0,T_{\max})$ and $(0,T_{\max})\ni t\mapsto  u(t,\cdot;u_0,v_0)\in X^\beta$ and $(0,T_{\max})\ni t\mapsto v(t,\cdot;u_0,v_0)\in X^{2\beta}$ is locally H\"older continuous for any $0<\beta<\frac{1}{2}$. Moreover, if $T_{\max}<\infty$, then  \eqref{final-eq1} holds.

\smallskip

\noindent {\bf Step 3.} In  this step, we prove that the mild solution $(u(t,x;u_0,v_0), v(t,x;u_0,v_0))$ obtained in the above two steps is a classical solution of \eqref{main-eq1} satisfying \eqref{local-1-eq1}-\eqref{local-1-eq2-1-1-1}.

\smallskip

To this end, first, fix any $t_1\in (0,T_{\max}(u_0,v_0))$.  {By \eqref{Fractional power Imbedding-0}},
$$ (-t_1,T_{\max}(u_0,v_0)-t_1)\ni t\mapsto u(t+t_1,\cdot;u_0,v_0)\in C^{\nu,b}_{\rm unif}(\R^n)$$
is locally H\"older continuous for some $\nu>0$.

Next,
 consider  the initial value problem
\begin{equation}
\label{2.13}
\begin{cases}
\tau\frac{\partial}{\partial t}\tilde{v}=(\Delta-1) \tilde{v}+\big(1-u( t+t_1,x;u_0,v_0)\big)\tilde{v}, \quad x\in\R^n,\,\, 0<t<T_{\max}(u_0,v_0)-t_1\\
\tilde{v}(0,x)=v_1(x):=v(t_1,x;u_0,v_0), \quad  x\in \R^n.
\end{cases}
\end{equation}
By \cite[Theorem 11 and Theorem 16 in Chapter 1]{Friedman}, \eqref{2.13} has a unique classical solution
$\tilde v(t,x)$ on $(0,T_{\max}-t_1)$ with $\lim_{t\to 0+} \|\tilde v(t,\cdot)-v_1\|_{C_{\rm unif}^b(\R^n)}=0$.
 By a priori interior estimates
for parabolic equations (see \cite[Theorem 5]{Friedman}), we  have
that
 \begin{equation}
\label{class0}
\tilde v(\cdot,\cdot)\in C^1((0,T_{\max}-t_1), C_{\rm unif}^{1,b}(\R^n)),
\end{equation}
 and  the mappings
\begin{equation}\label{class1}
    t\mapsto \tilde v(t,\cdot)\in C_{\rm unif}^{\nu,b}(\R^n),\quad
t\mapsto  \frac{\partial\tilde  v}{\partial x_i}(t,\cdot)\in C_{\rm unif}^{\nu,b}(\R^n),
\end{equation}
\begin{equation}\label{class2}
    t\mapsto  \frac{\partial^2 \tilde v}{\partial x_i\partial x_j}(t,\cdot)\in C_{\rm unif}^{\nu,b}(\R^n),\quad
 t\mapsto  \frac{\partial\tilde  v}{\partial t}(t,\cdot)\in C_{\rm unif}^{\nu,b}(\R^n)
\end{equation}
are locally H\"older continuous in $t\in (0,T_{\max}-t_1)$ for  $i,j=1,2,\cdots,N$ and $0<\nu\ll 1$. By \cite[Lemma 3.3.2]{Dan Henry},
 $\tilde v=\tilde v(t,\cdot)$ is also a mild solution of \eqref{2.13} and then satisfies the following integral equation,
\begin{eqnarray*}\label{utilde_Eq}
\tilde{v}(t,\cdot) = T\Big(\frac{t}{\tau}\Big)v_1 + {\frac{1}{\tau}}\int_0^t T\Big(\frac{t-s}{\tau}\Big)\big(1-u(s+t_1,\cdot;u_0,v_0)\big)\tilde v(s,\cdot) ds\quad \forall\, t\in [0,T_{\max}-t_1).
\end{eqnarray*}
Note that for any $t\in [0,T_{\max}-t_1)$,
\begin{eqnarray*}
    v(t+t_1,\cdot;u_0,v_0)
   = T\Big(\frac{t}{\tau}\Big)v_1 + {\frac{1}{\tau}}\int_0^t T\Big(\frac{t-s}{\tau}\Big)\big(1-u(s+t_1,\cdot;u_0,v_0)\big)v(s+t_1,\cdot;u_0,v_0)ds.
\end{eqnarray*}
It is  not difficult to prove that  $ v(t+t_1,\cdot;u_0,v_0)=\tilde v(t,\cdot)  $ for every $t\in [0,T_{\max}-t_1)$.
Hence,  by \eqref{class1} and \eqref{class2}, the mappings
$$ t\to \nabla v(\cdot,t+t_1,\cdot;u_0,v_0)\in C^\nu_{\rm unif}(\R^n), t\to \Delta v(\cdot,t+t_1,\cdot;u_0,v_0)\in C^\nu_{\rm unif}(\R^n)$$
are locally H\"older continuous in $t\in (-t_1,T_{\max}-t_1)$ for $0<\nu\ll 1$.

 Now, consider  the initial value problem
\begin{equation}
\label{2.133}
\begin{cases}
\frac{\partial}{\partial t}\tilde{u}=(\Delta-1) \tilde{u}+F(t, x,\tilde{u},\nabla \tilde u), \quad x\in\R^n,\,\, 0<t<T_{\max}-t_1\\
\tilde{u}(0,x)=u_{1}(x):=u(t_1,x;u_0,v_0), \quad  x\in \R^n,
\end{cases}
\end{equation}
where
\begin{align*}
F(t,x,\tilde u,\nabla \tilde u)&=-\chi\nabla v(t+t_1,x;u_0,v_0)\cdot \nabla \tilde{u}\\
&\quad +\Big( - \chi\Delta v( t+t_1,x;u_0,v_0)+1+a-bu(t+t_1,x;u_0,v_0)\Big)\tilde{u}.
\end{align*}
By the similar arguments as in the above, \eqref{2.133}  has a unique classical solution
$\tilde u(t,x)$ on $(0,T_{\max}-t_1)$ with $\lim_{t\to 0+} \|\tilde u(t,\cdot)-u_1\|_{C_{\rm unif}^b(\R^n)}=0$, and  $\tilde u(t,\cdot)$ satisfies \eqref{class0}-\eqref{class2} with $\tilde v$ being replaced by $\tilde u$.  Moreover, $u(t,\cdot;u_0,v_0)=\tilde u(t,\cdot)$
for $t\in [0,T_{\max}-t_1)$.

Therefore,  $(u(t,x;u_0,v_0),v(t,x;u_0,v_0))$ is a classical solution of \eqref{main-eq1}   satisfying   \eqref{local-1-eq1}-\eqref{local-1-eq2-1-1-1}.

\smallskip

\noindent {\bf Step 4.}  In this step,   by the uniqueness of mild solutions of \eqref{main-eq1}, we prove that \eqref{main-eq1} has a unique  classical solution  satisfying   \eqref{local-1-eq1}-\eqref{local-1-eq2-1-1-1}.

\smallskip

\noindent{\bf Step 5.} In this step, by the comparison principle for parabolic equations, we proved that
$$
u(t,x;u_0,v_0)\ge 0\quad {\rm and}\quad v(t,x;u_0,v_0)\ge 0\quad \forall\, t\in [0,T_{\max}(u_0,v_0),\,\, x\in\R^n
$$
provided that $u_0(x),v_0(x)\ge 0$ for $x\in\R^n$. Proposition \eqref{local-existence-prop1}(1) for \eqref{main-eq1} is thus proved.

\medskip

\noindent (2)   We provide the outline of the proof of Proposition \ref{local-existence-prop1}(2) for \eqref{main-eq1} in the following five  steps. The proof of   Proposition \ref{local-existence-prop1}(2) for \eqref{main-eq2} follows the same steps.

\smallskip

\noindent {\bf Step 1$^{'}$.}  In this step, we prove that \eqref{main-eq1} has a  unique mild solution
$(u(t,x;u_0,v_0),v(t,x;u_0,v_0))$  on $[0,T]$  satisfying \eqref{initial-eq2} for some $T>0$, that is,
$$[0,T]\ni t\mapsto (u(t,\cdot;u_0,v_0),v(t,\cdot;u_0,v_0))\in L^q(\R^n)\times W^{1,q}(\R^n)\quad \text{is continuous},
$$
and
\begin{equation}
\label{mild-solution-eq2}
\begin{cases}
u(t,\cdot;u_0,v_0)=T_q(t)u_0-\chi \int_0^t T_q(t-s)\nabla \cdot (u(s,\cdot;u_0,v_0)\nabla v(s,\cdot;u_0,v_0))ds\cr
\qquad\qquad \quad\quad\,\,  +\int_0^t T_q(t-s) u(s,\cdot;u_0,v_0)\big(1+a-b u(s,\cdot;u_0,v_0)\big)ds,\cr
v(t,\cdot;u_0,v_0)=T_q(\frac{t}{\tau})v_0 + \frac{1}{\tau}\int_0^t T_q(\frac{t-s}{\tau})\big(1-u(s,\cdot;u_0,v_0) \big)v(s,\cdot;u_0,v_0) ds
\end{cases}
\end{equation}
for $t\in [0,T]$.

To this end, as in Step 1 of the outline of the proof of Proposition \ref{local-existence-prop1}(1),  first,  for given $T>0$ and $R>0$,
 let
$$\mathcal{X}_{q,T}= C([0,T]; L^q(\R^n))\times C([0,T]; W^{1,q}(\R^n))
$$
 be equipped with the sup-norm, and
 $$
\mathcal{S}_{q,R,T}=\{(u,v) \in X: \sup_{s\in [0,T]}\|u(s)\|_{L^q(\R^n)} \leq R,  \sup_{s\in [0,T]}\|v(s)\|_{W^{1,q}(\R^n)} \leq R\}.$$
 Define  $\Psi_q$  on $\mathcal{S}_{q, R,T}$ by
\begin{align*}
\Psi _q(u,v)(t)& = \begin{pmatrix}
 \Psi_{1,q}(u,v)(t)\\
\Psi_{2,q}(u,v)(t)
\end{pmatrix}\\
& = \begin{pmatrix}
 T_q(t)u_0 -\chi\int_0^t T_q(t-s)\nabla \cdot (u(s)\nabla v(s))ds + \int_0^t T_q(t-s)u(s)(1+a-bu(s))ds\\
T_q(\frac{t}{\tau})v_0 +\frac{1}{\tau} \int_0^t T_q(\frac{t-s}{\tau})(1-u(s))v(s) ds.
\end{pmatrix}.
\end{align*}

Next, we show that the map $\Psi_q$  is a  well defined map from $\mathcal{S}_{q,R,T}$ to $\mathcal{X}_{q,T}$. Moreover, for any given $0<\beta<\frac{1}{2}-\frac{n}{2q}$, $0<\gamma<\frac{1}{2}$,  and  $(u,v)\in \mathcal{S}_{R,T}$, $(0,T]\ni t \mapsto \Psi_1(u,v)(t)\in X^\beta$ and $(0,T]\ni t\mapsto \Psi_2(u,v)(t)\in X^{2\gamma}$ are locally H\"older continuous.

We then prove that for given $R>\max\{\|u_0\|_{L^q(\R^n)},\|v_0\|_{W^{1,q}(\R^n)}\}$, there is $T:=T(R)>0$ such that
 the map $\Psi_q$ maps $\mathcal{S}_{q,R,T}$  into itself  and  is a contraction map. Hence $\Psi_q$ has a unique fixed point $(u(\cdot,\cdot;u_0,v_0),v(\cdot,\cdot;u_0,v_0))$ in $\mathcal{S}_{q,R,T}$ and then \eqref{main-eq1} has a   unique mild solution
$(u(t,x;u_0,v_0),v(t,x;u_0,v_0))$  on $[0,T]$  satisfying \eqref{initial-eq2} for some $T>0$.
Moreover, $(0,T]\ni t\mapsto  u(t,\cdot;u_0,v_0)\in X_q^\beta$ and $(0,T]\ni t\mapsto v(t,\cdot;u_0,v_0)\in X_q^{2\gamma}$ is locally H\"older continuous for any $0<\beta<\frac{1}{2}-\frac{n}{2q}$ and $0<\gamma<\frac{1}{2}$.

\smallskip

\noindent {\bf Step 2$^{'}$.} In this step, by the standard extension arguments, we prove that there is
$T_{\max}=T_{\max}(u_0,v_0)\in (0,\infty]$ such that \eqref{main-eq1} has a unique mild solution  $(u(t,x;u_0,v_0),v(t,x;u_0,v_0))$ satisfying \eqref{mild-solution-eq2} for
$t\in [0,T_{\max})$ and $(0,T_{\max})\ni t\mapsto  u(t,\cdot;u_0,v_0)\in X_q^\beta$ and $(0,T_{\max})\ni t\mapsto v(t,\cdot;u_0,v_0)\in X_q^{2\gamma}$ is locally H\"older continuous for any $0<\beta<\frac{1}{2}-\frac{n}{2q}$ and $0<\gamma<\frac{1}{2}$. Moreover, if $T_{\max}<\infty$, then  \eqref{final-eq2} holds.

\smallskip

\noindent {\bf Step 3$^{'}$.}  In this step, using the fact that $(0,T_{\max}(u_0,v_0))\ni t\mapsto v(t,\cdot;u_0,v_0)\in X_q^\gamma$ is locally H\"older continuous for any $0<\gamma<1$, we prove that
$(0,T_{\max}(u_0,v_0))\ni t \mapsto u(t,\cdot;u_0,v_0)\in X_q^\beta$ is locally H\"older continuous  for any $0<\beta<\frac{1}{2}$. Then by \eqref{Fractional power Imbedding-1},
$u(t,\cdot;u_0,v_0)\in C_{\rm unif}^b(\R^n)$ and $v(t,\cdot;u_0,v_0)\in C_{\rm unif}^{1,b}(\R^n)$ for all $t\in (0,T_{\max}(u_0,v_0))$.

{To be a little more specific,  by \eqref{Fractional power Imbedding-1}, for any
$0<t_1<t_2<T_{\max}(u_0,v_0)$, $[t_1,t_2]\ni t\mapsto \nabla v(t,\cdot;u_0,v_0)\in C_{\rm unif}^b(\R^n)$ is H\"older continuous. Then for any $0<\beta<\frac{1}{2}$ and $0<\delta<1$, there is $C>0$ such that for any  $t\in [t_1,t_2]$, we have
\begin{align*}
&\|A^\beta u(t,\cdot;u_0,v_0)\|_{L^q}\le \|A^\beta T_q(t-t_1) u(t_1,\cdot;u_0,v_0)\|_{L^q}\\
&\qquad\quad +\chi \int_{t_1}^ t \|A^\beta T_q(t-s)\nabla \cdot(u(s,\cdot;u_0,v_0)\nabla v(s,\cdot;u_0,v_0))\|_{L^q}ds\\
&\qquad\quad +\int_{t_1}^t \|A^\beta T_q(t-s) u(s)\big(1+a-b u(s)\big)\|_{L^q}ds\\
&\qquad\le   \|A^\beta T_q(t-t_1) u(t_1,\cdot;u_0,v_0)\|_{L^q}\\
&\qquad\quad +C \chi \max_{t_1\le t\le t_2}\|\nabla v(t,\cdot;u_0,v_0)\|_\infty \int_{t_1}^t
(t-s)^{-\beta-\frac{1}{2}}e^{-(1-\delta)(t-s)}\|u(s,\cdot;u_0,v_0)\|_{L^q} ds\\
&\qquad\quad +C\int_{t^1}^t e^{-(1-\delta)(t-s)}\Big((t-s)^{-\beta}\|u(s)\|_{L^q}
+(t-s)^{-\beta -\frac{n}{2q}}\|u^2(s)\|_{L^\frac{q}{2}}\Big)ds.
\end{align*}
This implies that for any $t\in (t_1,t_2)$, $u(t,\cdots;u_0,v_0)\in X_q^\beta$, and then for any $t\in (0,T_{\max}(u_0,v_0)$, $u(t,\cdots;u_0,v_0)\in X_q^\beta$. Moreover, we can prove that $(0,T_{\max}(u_0,v_0))\ni t\mapsto u(t,\cdot;u_0,v_0)\in X_q^\beta$ is locally H\"older continuous.
}

\smallskip

\noindent {\bf Step 4$^{'}$.} In this step, we prove that \eqref{main-eq1} has a unique classical solution satisfying   \eqref{initial-eq2}-\eqref{local-2-eq1-1} and \eqref{local-1-eq2}-\eqref{local-1-eq2-1-1-1}.

In fact,  by Step 3$^{'}$, for any $t_1\in (0, T_{\max}(u_0,v_0))$,    $u(t_1,\cdot;u_0,v_0)\in C_{\rm unif}^b(\R^n)$ and $v(t_1,\cdot;u_0,v_0)\in C_{\rm unif}^{1,b}(\R^n)$.  Then by Proposition \ref{local-existence-prop1}(1),
the mild solution $(u(t,x;u_0,v_0),v(t,x;u_0,v_0))$ obtained in the above is a classical solution of \eqref{main-eq1} satisfying  \eqref{local-2-eq1}-\eqref{local-2-eq1-1} and \eqref{local-1-eq2}-\eqref{local-1-eq2-1-1-1}.
Note that a classical solution of \eqref{main-eq1}  satisfying   \eqref{initial-eq2}-\eqref{local-2-eq1-1} and \eqref{local-1-eq2}-\eqref{local-1-eq2-1-1-1} is also a mild solution of \eqref{main-eq1} satisfying \eqref{initial-eq2}. Then, by the uniqueness of mild solution of \eqref{main-eq1},  \eqref{main-eq1} has a unique classical solution satisfying   \eqref{initial-eq2}-\eqref{local-2-eq1-1} and \eqref{local-1-eq2}-\eqref{local-1-eq2-1-1-1}.

\smallskip

\noindent {\bf Step 5$^{'}$.} In this step, by the comparison principle for parabolic equations, we prove that
$$
u(t,x;u_0,v_0)\ge 0\quad {\rm and}\quad v(t,x;u_0,v_0)\ge 0\quad \forall\, t\in [0,T_{\max}(u_0,v_0)),\,\, x\in\R^n
$$
provided $u_0(x)\ge 0$ and $v_0(x)\ge 0$ for all $x\in\R^n$. Proposition \ref{local-existence-prop1}(2) is thus proved.
\end{proof}

\subsection{Maximal regularity for parabolic initial-boundary
value problems}

In this subsection, we present a lemma on the maximal regularity for parabolic equations on
$\R^n$ to be used in the proofs of Theorems \ref{main-thm1} and \ref{main-thm2}.

\begin{lem}
\label{maximal-regularity-lm}
    Let $\tau>0, \alpha>0$,  and  $\gamma \in(1, \infty) $ and $v_0\in L^\gamma(\R^n)\cap L^\infty(\R^n)$.  There exists {$C_{\gamma,n}$ {independent of $\tau$} and $\alpha$} such that for any  $T\in (0,\infty)$, if  $g \in L^\gamma((0,T), L^\gamma(\R^n))$ and $v(\cdot,\cdot)\in W^{1,\gamma}((0,T),L^{\gamma}(\R^n))\cap L^{\gamma}((0,T), W^{2,\gamma}(\R^n))$ solves  the following initial boundary value problem,
\begin{equation}
\label{pdelaplace}
\begin{cases}
\tau v_t =\Delta v -\alpha  v + g,\quad x\in \R^n,\,\,  0<t<T\cr
v(0,x) = v_0(x),
\end{cases}
\end{equation}
then for any $t_0\in (0,T)$,
\begin{equation}
    \begin{aligned}
\label{maximal-regularity-eq1}
&\int_{t_0}^T \int_{\R^n}e^{\frac{\alpha\gamma t}{\tau}}\Big( |v (t,x)|^\gamma +|\nabla v(t,x)|^\gamma +|\Delta v(t,x)|^{\gamma}\Big)dxdt\\
&\qquad\qquad \le C_{\gamma,n} \int_{t_0}^T \int_{\R^n}e^{\frac{\alpha\gamma t}{\tau}}|g(t,x)|^{\gamma}dx dt+C_{\gamma,n}(T+\tau^\gamma t_0^{1-\gamma})\|v_0(\cdot)\|^{\gamma}_{L^\gamma(\R^n)} .
\end{aligned}
\end{equation}

\end{lem}

\begin{proof}
First assume that  $v_0 = 0$. Let
$$
\tilde g(t,x)=
\begin{cases}
g(t,x)\quad &{\rm for}\quad t\in (0,T),\,\, x\in\R^n\cr\cr
0\quad &{\rm for}\quad t>T,\,\, x\in\R^n.
\end{cases}
$$
 By \cite[Theorem 3.1]{HiPr},  the initial value problem
\begin{equation*}
\begin{cases}
\tau \tilde v_t =\Delta \tilde v-\alpha\tilde v  + \tilde g,\quad x\in \R^n,\,\,  0<t<\infty \cr
\tilde v(0,x) = 0
\end{cases}
\end{equation*}
has a unique solution
 $\tilde v(\cdot,\cdot)\in W^{1,\gamma}((0,\infty),L^{\gamma}(\R^n))\cap L^{\gamma}((0,\infty), W^{2,\gamma}(\R^n))$.
This implies that
the initial value problem
\eqref{pdelaplace} has a unique solution  in $W^{1,\gamma}((0,T),L^{\gamma}(\R^n))\cap L^{\gamma}((0,T), W^{2,\gamma}(\R^n))$.

Let
\[
w(t,x):=e^{{ \frac{{ \alpha t}}{\tau}}}\tilde v(t,x).
\]
Then $w(t,x)$ solves
\begin{equation}\label{pdelaplace-1}
\begin{cases}
\tau w_t =\Delta  w + e^{{ \frac{{ \alpha t}}{\tau}} }\tilde g,\quad x\in \R^n,\,\,  0<t<\infty \cr
 w(0,x) = 0
\end{cases}
\end{equation}
By the above arguments,  \eqref{pdelaplace-1}  has a unique solution in $W^{1,\gamma}((0,\infty),L^{\gamma}(\R^n))$ $\cap$ $L^{\gamma}((0,\infty), W^{2,\gamma}(\R^n))$. By the closed graph theorem, there is ${ C_{\gamma,\tau}'}>0$ independent of $T,\alpha$ such that
\begin{equation}
\label{est1}
  \int_0^T \int_{\R^n}e^{{ \frac{\alpha\gamma t}{\tau}}} \Big( |v(t,x)|^\gamma+|  \nabla v(t,x)|^\gamma+|\Delta v(t,x)|^\gamma\Big)dxdt\le C_{\gamma,\tau}' \int_0^T \int_{\R^n} e^{{ \frac{\alpha\gamma t}{\tau}}}| g(t,x)|^\gamma dxdt.
\end{equation}

Now, assume $v_0\not =0$.  Let $\bar{v}$ solve \eqref{pdelaplace} with $v_0$ replaced by $0$, and then $\bar{v}$ satisfies \eqref{est1}. Note that $w:=v-\bar{v}$ satisfies
\begin{equation}
\label{pdelaplace-2}
\begin{cases}
\tau w_t =\Delta w -\alpha  w,\quad t
\in (0,T),\, x\in \R^n\cr
w(0,x) = v_0.
\end{cases}
\end{equation}
By classical results of heat equations, $w$ is explicit:
\[
w(t,x)= e^{-{ \frac{\alpha t}{\tau}}}(G * v_0)(t/\tau,x)
\]
where $G$ is the heat kernel \eqref{heat-kernel}, and we know for some absolute constant $C>0$
\[
|\nabla G(t,x)|\leq \frac{|x|}{2t} G(t,x)\leq \frac{C}{\sqrt{t}}G(2t,x),
\]
and
\[
|G_t(t,x)|\leq \left(\frac{n}{{2 t}}+\frac{|x|^2}{4t^2}\right)G(t,x)\leq \frac{Cn}{t}G(t,x)+\frac{C}{t} {  G(2t,x)}.
\]
Then, by Young's convolution inequality, it is not hard to see that for all $t>0$,
\beq\label{est2}
\begin{aligned}
&\|w(t,\cdot)\|^{\gamma}_{L^\gamma(\R^n)} +\|\nabla w(t,\cdot)\|^{\gamma}_{L^\gamma(\R^n)}+ \|\Delta w(t,\cdot)\|^{\gamma}_{L^\gamma(\R^n)}\\
&\qquad\qquad\leq C^\gamma e^{-{ \frac{\alpha \gamma t}{\tau}}} \left[\Big\|(1+\frac{\tau}{t})G\big(\frac{2t}\tau,\cdot\big)\Big\|_{L^1(\R^n)}^\gamma+\Big\|(1+\frac{n\tau}{t})G\big(\frac{t}\tau,\cdot\big)\Big\|_{L^1(\R^n)}^\gamma \right]
\|v_0(\cdot)\|^{\gamma}_{L^\gamma(\R^n)}\\
&\qquad\qquad\leq C^\gamma e^{- \frac{\alpha \gamma t}{\tau}}\left[1+ \big(\frac{n\tau}{t}\big)^{\gamma}\right]\|v_0(\cdot)\|^{\gamma}_{L^\gamma(\R^n)}.
\end{aligned}
\eeq
Multiplying \eqref{est2} with $e^{{ \frac{\alpha\gamma t}{\tau}}}$ and integrating over $t\in (t_0,T)$ yield for some absolute $C>0$,
\begin{align*}
\int_{t_0}^T \int_{\R^n}e^{\frac{\alpha\gamma t}{\tau}}\Big( |w |^\gamma +|\nabla w|^\gamma +|\Delta w|^{\gamma}\Big)dxdt \le C^\gamma n^\gamma\Big(T+\tau^\gamma t_0^{1-\gamma}\Big)\|v_0(\cdot)\|^{\gamma}_{L^\gamma(\R^n)} .
\end{align*}
The above and \eqref{est1} yield the existence of $C_{\gamma,\tau,n}>0$ such that
 \begin{equation}
    \begin{aligned}
\label{maximal-regularity-eq1-1}
&\int_{t_0}^T \int_{\R^n}e^{\frac{\alpha\gamma t}{\tau}}\Big( |v (t,x)|^\gamma +|\nabla v(t,x)|^\gamma +|\Delta v(t,x)|^{\gamma}\Big)dxdt\\
&\qquad\qquad \le C_{\gamma,\tau,n} \int_{t_0}^T \int_{\R^n}e^{\frac{\alpha\gamma t}{\tau}}|g(t,x)|^{\gamma}dx dt+C_{\gamma,\tau,n}(T+\tau^\gamma t_0^{1-\gamma})\|v_0(\cdot)\|^{\gamma}_{L^\gamma(\R^n)} .
\end{aligned}
\end{equation}

In the following, we always assume that $C_{\gamma,\tau,n}$ is the
smallest positive constant such that  the inequality \eqref{maximal-regularity-eq1-1} holds.
We claim that $C_{\gamma,\tau,n}=C_{\gamma,1,n}$ for all $\tau>0$. Indeed, by a scaling argument, one can consider $\tilde{v}(t,x):=v(\tau t,x)$ which solves
\[
\tilde{v}_t=\Delta\tilde{v}-\alpha \tilde v+g(\tau t,x),\quad \tilde{v}(0,x)=v_0.
\]
Then applying \eqref{maximal-regularity-eq1-1} (with $\tau=1$, and $t_0,T$ replaced by ${\frac{t_0}{\tau}, \frac{T}{\tau}}$ respectively) to $\tilde{v}$ yields
\begin{align*}
&\int_{{\frac{t_0}{ \tau}}}^{{ \frac{T}{\tau}}} \int_{\R^n}e^{{\alpha\gamma \tilde t}}\Big( |\tilde v (\tilde t,x)|^\gamma +|\nabla \tilde v(\tilde t,x)|^\gamma +|\Delta \tilde v(\tilde t,x)|^{\gamma}\Big)dxd\tilde t\\
&\qquad \le C_{\gamma,1,n} \int_{{\frac{t_0}{ \tau }}}^{{\frac{T}{\tau}}}\int_{\R^n}e^{{\alpha\gamma t}}|g(\tau\tilde  t,x)|^{\gamma}dx d\tilde t+C_{\gamma,1,n}\Big({\frac{T}{\tau }}+\big({\frac{t_0}{\tau }}\big)^{1-\gamma}\Big)\|v_0(\cdot)\|^{\gamma}_{L^\gamma(\R^n)}.
\end{align*}
{This together with the  variable change $\tilde t=\frac{t}{\tau}$ implies that
\begin{align*}
&\frac{1}{\tau}\int_{{t_0}}^{{ T}} \int_{\R^n}e^{{\frac{\alpha\gamma t}{\tau}}}\Big( | v (t,x)|^\gamma +|\nabla  v(t,x)|^\gamma +|\Delta  v(t,x)|^{\gamma}\Big)dxdt\\
&\qquad \le\frac{1}{\tau} C_{\gamma,1,n} \int_{{t_0}}^{{ T}}\int_{\R^n}e^{{\frac{\alpha\gamma t}{\tau}}}|g( t,x)|^{\gamma}dx dt+C_{\gamma,1,n}\Big({\frac{T}{\tau }}+\big({\frac{t_0}{\tau }}\big)^{1-\gamma}\Big)\|v_0(\cdot)\|^{\gamma}_{L^\gamma(\R^n)}.
\end{align*}}
This implies \eqref{maximal-regularity-eq1} with $C_{\gamma,n}=C_{\gamma,1,n}$.
\end{proof}

\subsection{A useful  exponentially decaying function}

In this subsection, we present a lemma on the existence of some special exponentially decaying functions to be used in later sections.

\begin{lem}
\label{psi-lm}
For any {$\kappa_1>0$, there are $\k_2>0$} and
 $\psi\in C^\infty(\R^n)$  such that
\begin{equation}\label{psi-eq00}
    0<\psi(x)\le e^{-\k_2|x|},\quad |\nabla \psi(x)|\le \k_1 \psi(x),\quad |\Delta \psi(x)|\le \k_1\psi \quad  \forall\, x\in\R^n.
\end{equation}
\end{lem}

\begin{proof}
First, for given $\k_1>0$, choose $\k_2>0$ such that
$$
\k_2\le \min\{\k_1/\sqrt n,\, \sqrt{\k_1/(2n)}\}.
$$
Next, let
$$
\psi(x)=\frac{1}{e^{\k_2 x_1}+e^{-\k_2 x_1}}\, \times \, \frac{1}{e^{\k_2 x_2}+e^{-\k_2 x_2}}\,\times \, \cdots \,\times\, \frac{1}{e^{\k_2 x_n}+e^{-\k_2 x_n}}.
$$
Then $\psi\in C^\infty(\R^n)$,
$$
\psi_{x_i}(x)=-\psi(x) \frac{\k_2 e^{\k_2  x_i}-\k_2  e^{-\k_2  x_i}}{e^{\k_2  x_i}+e^{-\k_2  x_i}},
$$
and
$$
\psi_{x_ix_i}=2\psi(x) \Big(\frac{\k_2 e^{\k_2  x_i}-\k_2  e^{-\k_2  x_i}}{e^{\k_2  x_i}+e^{-\k_2  x_i}}\Big)^2
-\k_2^2\psi(x).
$$
It then follows that
$$
0<\psi(x)\le e^{-\k_2   |x|},
\qquad |\nabla \psi(x)|\le \k_2 \sqrt n \psi(x)\le \k_1\psi,
$$
and
$$
|\Delta \psi(x)|\le 2\k_2^2n \psi(x)\le \k_1\psi
$$
for all $x\in\R^n$.  The lemma is thus proved.
\end{proof}

\section{A unified method for global existence  of classical solutions}

In this section, we investigate the global existence of classical solutions of \eqref{main-eq1}, \eqref{main-eq2}, and \eqref{main-eq3} in a unified way, and prove Theorem \ref{general-global-existence-thm}.
To do so, we first prove some fundamental properties of classical solutions of \eqref{main-eq1}, \eqref{main-eq2}, and \eqref{main-eq3} in subsection 3.1. We then prove Theorem \ref{general-global-existence-thm} in subsection 3.2.

\subsection{Fundamental propositions}

In this subsection,  we prove some fundamental properties of classical solutions of
\eqref{main-eq1}, \eqref{main-eq2}, or \eqref{main-eq3}. Throughout this subsection,
we assume that $u_0\in C_{\rm unif}^b(\R^n)$ and $v_0\in C_{\rm unif}^{1,b}(\R^n)$ or
$u_0\in L^q(\R^n)$ and $v_0\in W^{1,q}(\R^n)$ for some $q>n$ and $q\ge 2$, and
$u_0\ge 0$, $v_0\ge 0$.  { Note that, by Morrey's inequality, }
$$
{\|v_0\|_{L^\infty(\R^n)}<\infty.}
$$
We assume that  $(u(t,x;u_0,v_0), v(t,x;u_0,v_0))$  is  the classical solution
of \eqref{main-eq1} or \eqref{main-eq2} on the time interval $(0,T_{\max}(u_0,v_0))$ satisfying the initial
condition
$$
\lim_{t\to 0+}\Big(\|u(t,\cdot;u_0,v_0)-u_0(\cdot)\|_X+\|v(t,\cdot;u_0,v_0)-v_0(\cdot)\|_Y\Big)=0,
$$
where  $X= C_{\rm unif}^b(\R^n)$ and $Y= C_{\rm unif}^{1,b}(\R^n)$ or
$X= L^q(\R^n)$ and $Y=W^{1,q}(\R^n)$, and that $(u(t,x;u_0),v(t,x;u_0))$ be the classical solution of \eqref{main-eq3} on $(0,T_{\max}(u_0))$ satisfying the initial condition
$$
\lim_{t\to 0+}\|u(t,\cdot;u_0)-u_0(\cdot)\|_X=0.
$$

We first  show that local $L^p$-bound of $u$ implies local $W^{1,p}$ boundedness of $v$.

\begin{prop}
\label{basic-prop2} Consider \eqref{main-eq1}-\eqref{main-eq3}.
  Let $(u(t,x;u_0,v_0),v(t,x;u_0,v_0))$ be the solution of \eqref{main-eq1} or \eqref{main-eq2},
and $(u(t,x;u_0)$, $v(t,x;u_0))$ be the solution of \eqref{main-eq3}.
Put
$$
(u(t,x),v(t,x))=(u(t,x;u_0,v_0),v(t,x;u_0,v_0))\quad {\rm or}\quad
(u(t,x),v(t,x))=(u(t,x;u_0),v(t,x;u_0))
$$
and
$$
T_{\max}=T_{\max}(u_0,v_0)\quad {\rm or}\quad T_{\max}(u_0).
$$
For given $p\ge 1$, assume that
\begin{equation}
\label{prop2-eq0}
\sup_{t\in [0,T_{\max}),x_0\in\R^n}\int_{B(x_0,1)} u^p(t,x)dx<\infty\quad{\rm and}\quad  \sup_{x_0\in\R^n}\int_{B(x_0,1)}{ |\nabla v_0(x)|^p }dx<\infty.
\end{equation}
Then the following hold.

\begin{itemize}
\item[(1)]
\begin{equation}
\label{prop2-eq1}
\sup_{t\in [0,T_{\max}),x_0\in\R^n}\int_{B(x_0,1)}\big( v^p(t,x)+|\nabla v(t,x)|^p\big)dx<\infty.
\end{equation}

\item[(2)] In addition,  if $p>\frac{n}{4}$,  then
\begin{equation}
\label{prop2-eq2}
\sup_{t\in [0,T_{\max}),x_0\in\R^n}\int_{B(x_0,1)} v^{2p}(t,x)dx<\infty.
\end{equation}

\item[(3)] In addition, if $p>\frac{n}{2}$ and   $ \sup_{x_0\in\R^n}\int_{B(x_0,1)}|\nabla v_0(x)|^{2p}<\infty$, then
\begin{equation}
\label{prop2-eq3}
\sup_{t\in [0,T_{\max}),x_0\in\R^n}\int_{B(x_0,1)} |\nabla v (t,x)|^{2p} dx<\infty.
\end{equation}
\end{itemize}
\end{prop}

\begin{proof}
 First of all, we    note that,
to  prove \eqref{prop2-eq1}, \eqref{prop2-eq2}, and \eqref{prop2-eq3},
 it  suffices to prove
\begin{equation}
\label{proof-prop2-eq1}
\sup_{t\in[0,T_{\max}),x_0\in\R^n}\int_{\R^n}\Big(v^p(t,x+x_0)\psi^p(x)+|\nabla v|^p(t,x+x_0)\psi^p(x)\Big)dx<\infty,
\end{equation}
\begin{equation}
\label{proof-prop2-eq2}
\sup_{t\in [0,T_{\max}),x_0\in\R^n}\int_{\R^n} v^{2p}(t,x+x_0)\psi^{2p}(x)dx<\infty,
\end{equation}
and
\begin{equation}
\label{proof-prop2-eq3}
\sup_{t\in [0,T_{\max}),x_0\in\R^n}\int_{\R^n} |\nabla v|^{2p}(t,x+x_0)\psi^{2p}(x)dx<\infty,
\end{equation}
respectively,
for some  $\psi\in C^\infty(\R^n)$   satisfying \eqref{psi-eq00}, that is,
\begin{equation}
\label{psi-eq00-1}
  0<\psi(x)\le e^{-\k_2|x|},\quad |\nabla \psi(x)|\le \k_1 \psi(x),\quad |\Delta \psi(x)|\le \k_1\psi \quad  \forall\, x\in\R^n
\end{equation}
for some $\k_1,\k_2>0$.
If no confusion occurs, we put
$$
u=u(t,x+x_0)\quad {\rm and}\quad v=v(t,x+x_0).
$$
If $(u,v)$ is the solution of \eqref{main-eq1}, then
 \begin{equation}
\label{proof-prop1-eq1-0}
\|v(t,\cdot)\|_\infty\le \|v_0\|_\infty\quad \forall\, t\in [0,T_{\max}).
\end{equation}

Next, we note that, in the case that $(u,v)$ is the  solution of \eqref{main-eq1}, we have for any $\lambda>0$,
\begin{align*}
\tau \frac{d}{dt}(v \psi)
&=\Delta(v\psi )-  \lambda v \psi+\lambda v\psi -2\nabla v\cdot\nabla \psi-v\Delta \psi-\psi uv,\quad x\in\R^n.
\end{align*}
Hence
\begin{align}
\label{main-eq1-psi}
v(t,\cdot)\psi&=e^{(\Delta-\lambda I)\frac{ t}{\tau}} v_0\psi\nonumber\\
&  +\frac{1}{\tau}\int_0^t e^{(\Delta-\lambda I)\frac{t-s}{\tau}}\Big(\lambda v(s,\cdot)\psi-2\nabla v(s,\cdot)\cdot\nabla \psi-v(s,\cdot)\Delta\psi-\psi u(s,\cdot)v(s,\cdot)\Big)ds.
\end{align}
In the case  that $(u(t,x),v(t,x))=(u(t,x;u_0,v_0),v(t,x;u_0,v_0))$ is the  solution of \eqref{main-eq2}, we have
\begin{align*}
\tau \frac{d}{dt}(v \psi)
&=\Delta(v\psi)-  \lambda  v \psi-2\nabla v\cdot\nabla \psi-v\Delta\psi+\mu \psi u.
\end{align*}
Hence
\begin{equation}
\label{main-eq2-psi}
v(t,\cdot)\psi=e^{(\Delta-\lambda I)\frac{t}{\tau}} v_0\psi+\frac{1}{\tau}\int_0^t e^{(\Delta-\lambda I)\frac{t-s}\tau}
\Big(-2\nabla v(s,\cdot)\cdot\nabla \psi-v(s,\cdot)\Delta \psi+\mu u(s,\cdot)\psi \Big)ds.
\end{equation}
In the case that $(u,v)$ is  a solution of \eqref{main-eq3}, we have
\begin{equation*}
\Delta (v\psi)-\lambda v \psi -2\nabla v\cdot\nabla \psi-v\Delta \psi +\mu u\psi=0.
\end{equation*}
Hence, {using the resolvent of $(\Delta -\lambda I)$ in $\R^n$ yields}
\begin{equation}
\label{main-eq3-psi}
v(t,\cdot)\psi(\cdot)=\int_0^\infty e^{(\Delta -\lambda I)s}\Big(-2\nabla v(t,\cdot)\cdot\nabla \psi-v(t,\cdot)\Delta \psi+\mu u(t,\cdot)\psi\Big)ds.
\end{equation}

In the following, we provide some estimates for $\|v\psi\|_{L^{p'}}$ and $\|(\nabla v )\psi\|_{L^{p'}}$ for $p'\ge p$, which will play essential roles in the proofs of \eqref{proof-prop2-eq1}, \eqref{proof-prop2-eq2}, and \eqref{proof-prop2-eq3}.

  Note that
\begin{equation}
\label{proof-prop2-eq00}
e^{(\Delta-\lambda I)t}=e^{(\Delta -I)t} e^{(1-\lambda)t}\quad \forall\, t\ge 0.
\end{equation}
In the case that $(u,v)$ is the  solution of \eqref{main-eq1},  we have that for any $p\le p'$ and $t\in [0,T_{\max})$,
\begin{align}
\label{main-eq1-estimate-eq1}
\|v(t,\cdot)\psi\|_{L^{p'}} \le \|v_0\|_\infty\|\psi\|_{L^{p'}},
\end{align}
and by \eqref{main-eq1-psi},
\begin{align*}
\| (\nabla  v (t,\cdot))\psi\|_{L^{p'}}&=\|\nabla (v (t,\cdot) \psi)-v(t,\cdot)\nabla \psi\|_{L^{p'}}\nonumber\\
&\le \|v(t,\cdot)\nabla \psi\|_{L^{p'}}+ \|\nabla e^{(\Delta-\lambda I)\frac{t}{\tau}} (v_0\psi)\|_{L^{p'}}\nonumber\\
&\quad +\frac{1}{\tau}
\int_0^t \left\|\nabla e^{(\Delta-\lambda I)\frac{t-s}{\tau}}\Big(\lambda v\psi-2\nabla v\cdot\nabla \psi-v\Delta\psi-\psi uv\Big)\right\|_{L^{p'}}ds\nonumber
\end{align*}
Using \eqref{Lp Estimates-2},  \eqref{Lp Estimates-3}, \eqref{psi-eq00-1},  \eqref{proof-prop1-eq1-0}  and \eqref{proof-prop2-eq00}, we get
\beq\label{main-eq1-estimate-eq2}
\begin{aligned}
&\| (\nabla  v (t,\cdot))\psi\|_{L^{p'}}    \leq   {\k_1\|v_0\|_\infty}\|\psi\|_{L^{p'}}+C_{ p'} e^{- \frac{ \lambda t}{\tau}}\|\nabla(v_0\psi)\|_{ L^{p'}} \\
&\quad + \frac{2 \k_1 C_{p,p'}}{\tau}\sup_{r\in [0,t]} \|(\nabla v(r,\cdot)) \psi\|_{L^p} \int_0^t e^{-\lambda \frac{t-s}{\tau}}\Big( \frac{t-s}{\tau}\Big)^{-\frac{1}{2}-\big(\frac{1}{p}-\frac{1}{p'}\big)\frac{n}{2}} ds\\
&\quad + \frac{  C_{p,p'} \|v_0\|_\infty}{\tau}  \Big( ( \lambda+\k_1) \|\psi\|_{L^p}+{\sup_{r\in [0,t]}}\|u(r,\cdot)\psi\|_{L^p}\Big) \int_0^t e^{-\lambda \frac{t-s}{\tau}}\Big( \frac{t-s}{\tau}\Big)^{-\frac{1}{2}-\big(\frac{1}{p}-\frac{1}{p'}\big)\frac{n}{2}}ds.
\end{aligned}
\eeq
In the case  that $(u(t,x),v(t,x))=(u(t,x;u_0,v_0),v(t,x;u_0,v_0))$ is the  solution of \eqref{main-eq2},  by   \eqref{Lp Estimates-2},  \eqref{Lp Estimates-3},  \eqref{main-eq2-psi}, and \eqref{proof-prop2-eq00},   we have that for any $p\le p'<\infty$ and $t\in [0,T_{\max})$,
\begin{align}
\label{main-eq2-estimate-eq1}
&\quad\, \|v(t,\cdot)\psi\|_{L^{p'}}\nonumber\\
&\le 
{C_p' e^{- \frac{ \lambda t}{\tau}}\|v_0\psi\|_{ L^{p'}}}+\frac{C_{p,p'} \mu }{\tau}\sup_{r\in [0,t]}\|u(r,\cdot)\psi\|_{L^p} \int_0^t  e^{-\lambda \frac{t-s}{\tau}}
\Big(\frac{t-s}{\tau}\Big)^{-\big(\frac{1}{p}-\frac{1}{p'}\big)\frac{n}{2}} ds\nonumber\\
&\quad + \frac{C_{p,p'}\kappa_1}\tau \sup_{r\in [0,t]}\Big( 2 \| \nabla v(r,\cdot) \psi\|_{L^p}  +   \|v(r,\cdot)\psi \|_{L^p}\Big)\int_0^t  e^{-\lambda \frac{t-s}{\tau} } \Big(\frac{t-s}{\tau}\Big)^{-\big(\frac{1}{p}-\frac{1}{p'}\big)\frac{n}{2}}ds,
\end{align}
and, similarly as before,
\begin{align}
\label{main-eq2-estimate-eq2}
&\quad\,\,\|\nabla (v(t,\cdot)\psi)\|_{L^{p'}}=\|\nabla (v(t,\cdot)\psi)-v(t,\cdot)\nabla \psi\|_{L^{p'}}\nonumber\\
&\le \k_1\|v(t,\cdot)\psi\|_{L^{p'}}+\|{\nabla e^{(\Delta-\lambda I)\frac{ t}\tau}(v_0\psi)}\|_{L^{p'}}\nonumber\\
& +\frac{C_{p,p'}}{\tau}\int_0^t e^{-\lambda \frac{t-s}\tau}\Big(\frac{t-s}{\tau}\Big)^{-\frac{1}{2}-\big(\frac{1}{p}-\frac{1}{p'}\big)\frac{n}{2}}\Big( 2\k_1\|{(\nabla v(s,\cdot))}\psi\|_{L^p}+\k_1\|v(s,\cdot)\psi\|_{L^p}+\mu \|u(s,\cdot)\psi\|_{L^p}\Big)ds\nonumber\\
&\le  \k_1{\|v(t,\cdot)\psi\|_{L^{p'}}} +
{ C_{p'}e^{-\frac{\lambda t}{\tau}}\|\nabla (v_0\psi)\|_{L^{p'}}} \nonumber\\
& +\frac{C_{p,p'}}{\tau}  \mu \sup_{r\in [0,t]}\|u(r,\cdot)\psi\|_{L^p} \int_0^t e^{-\lambda \frac{t-s}\tau}\Big(\frac{t-s}{\tau}\Big)^{-\frac{1}{2}-\big(\frac{1}{p}-\frac{1}{p'}\big)\frac{n}{2}}ds \nonumber\\
& +\frac{C_{p,p'}\kappa_1}{\tau}\sup_{r\in [0,t]} \Big( 2\|(\nabla v(r,\cdot))\psi\|_{L^p}+\|v(r,\cdot)\psi\|_{L^p}\Big)\int_0^t e^{-\lambda \frac{t-s}\tau}\Big(\frac{t-s}{\tau}\Big)^{-\frac{1}{2}-\big(\frac{1}{p}-\frac{1}{p'}\big)\frac{n}{2}}ds.
\end{align}

In the case that $(u,v)$ is  a solution of \eqref{main-eq3},  by  \eqref{Lp Estimates-2},  \eqref{Lp Estimates-3},   \eqref{main-eq3-psi}, and \eqref{proof-prop2-eq00},  we have
that for any $p { \le p'}<\infty$ and $t\in[0,T_{\max})$,
\begin{align}
\label{main-eq3-estimate-eq1}
&\quad\,\|v(t,\cdot)\psi(\cdot)\|_{L^{p'}}\nonumber\\
&\le C_{p,p'}  \int_0^\infty e^{-\lambda s} s^{-\big(\frac{1}{p}-\frac{1}{p'}\big)\frac{n}{2}}\Big(2 \k_1\|(\nabla v (t,\cdot)) \psi\|_{L^p}+\k_1\|v(t,\cdot) \psi\|_{L^p}+\mu \|u(t,\cdot)\psi\|_{L^p}\Big)ds\nonumber\\
&\le C_{p,p'}  \mu \|u(t,\cdot)\psi\|_{L^p}  \int_0^\infty e^{-\lambda s} s^{-\big(\frac{1}{p}-\frac{1}{p'}\big)\frac{n}{2}}ds\nonumber\\
&\quad  + C_{p,p'}  \Big(2 \k_1 \|\nabla v (t,\cdot) \psi\|_{L^p}+\k_1
\|v(t,\cdot) \psi\|_{L^p}\Big) \int_0^\infty e^{-\lambda s} s^{-\big(\frac{1}{p}-\frac{1}{p'}\big)\frac{n}{2}}ds,
\end{align}
and
\begin{align}
\label{main-eq3-estimate-eq2}
&\quad\,\,\|(\nabla v(t,\cdot)) \psi\|_{L^{p'}}=\|\nabla (v(t,\cdot)\psi)-v(t,\cdot)\nabla \psi\|_{L^{p'}}\nonumber\\
&\le \|v(t,\cdot) \nabla \psi\|_{L^{p'}}+\|\nabla (v(t,\cdot)\psi)\|_{L^{p'}}\nonumber\\
&\le \k_1\|v(t,\cdot)\psi\|_{L^{p'}}+C_{p,p'}\int_0^\infty e^{-\lambda s} s^{-\frac{1}{2}-\big(\frac{1}{p}-\frac{1}{p'}\big)\frac{n}{2}}\left\|-2\nabla v(t,\cdot)\cdot\nabla \psi-v(t,\cdot)\Delta \psi+\mu u(t,\cdot)\psi\right\|_{L^p} ds\nonumber\\
&\le \k_1{\|v(t,\cdot)\psi\|_{L^{p'}}}+C_{p,p'} \mu \|u(t,\cdot)\psi\|_{L^p}\int_0^\infty e^{-\lambda s}s^{-\frac{1}{2}-\big(\frac{1}{p}-\frac{1}{p'}\big)\frac{n}{2}} ds\nonumber\\
&\quad + C_{p,p'}\k_1 \Big(2 \|\nabla v(t,\cdot;u_0)\psi\|_{L^p}+\|v(t,\cdot)\psi\|_{L^p}\Big) \int_0^\infty e^{-\lambda s}s^{-\frac{1}{2}-\big(\frac{1}{p}-\frac{1}{p'}\big)\frac{n}{2}} ds.
\end{align}

We now prove (1), (2), and (3).

\smallskip

(1) As it is mentioned in the above, to prove \eqref{prop2-eq1}, it suffices to prove \eqref{proof-prop2-eq1}.

Suppose that  $\psi\in C^\infty(\R^n)$ satisfies \eqref{psi-eq00-1} for some $\k_1,\k_2>0$.
We claim that there are constant $C>0$ independent of $\k_1$ and constant $\tilde C>0$ such that
\begin{align}
\label{proof-prop2-eq4}
\| v(t,\cdot)\psi\|_{L^p}+\|\nabla v(t,\cdot) \psi\|_{L^p}\le \tilde  C +C \k_1  \sup_{s\in [0,t]}\Big(\|v(s,\cdot) \psi\|_{L^p}+\|\nabla v(s,\cdot) \psi\|_{L^p}\Big) \quad\forall\, t\in [0,T_{\max}).
\end{align}

Before proving \eqref{proof-prop2-eq4}, we claim that \eqref{proof-prop2-eq1} follows from
    \eqref{proof-prop2-eq4} by choosing $0<\k_1\ll 1$.  In fact, assume that \eqref{proof-prop2-eq4} holds. Then for any $0\leq t<T_{\max}$, we have
\begin{align*}
\| v(t,\cdot)\psi\|_{L^p}+\|\nabla v(t,\cdot) \psi\|_{L^p}
\le \tilde C+ C\k_1  \sup_{s\in [0,T_{\max})} \Big(\|v(s,\cdot) \psi\|_{L^p}+\|\nabla v(s,\cdot) \psi\|_{L^p}\Big).
\end{align*}
This implies that
$$
\Big(1-C\k_1 \Big)\sup_{t\in [0,T_{\max})} \Big(\|v(t,\cdot)\psi\|_{L^p}+\|\nabla v(t,\cdot)\psi\|_{L^p}\Big)\le \tilde C.
$$
\eqref{proof-prop2-eq1} then follows by choosing $\psi\in C^\infty(\R^n)$ satisfying  \eqref{psi-eq00} with $0<\k_1\ll 1$.

\smallskip

We now prove \eqref{proof-prop2-eq4}.
In the case that $(u,v)$ is the  solution of \eqref{main-eq1},  \eqref{proof-prop2-eq4} follows from \eqref{main-eq1-estimate-eq1} and \eqref{main-eq1-estimate-eq2} with $p'=p$.
In the case that $(u,v)$ is the solution of \eqref{main-eq2},  \eqref{proof-prop2-eq4} follows from \eqref{main-eq2-estimate-eq1} and \eqref{main-eq2-estimate-eq2} with $p'=p$. In the case that $(u,v)$ is the solution of \eqref{main-eq3},  \eqref{proof-prop2-eq4} follows from \eqref{main-eq3-estimate-eq1} and \eqref{main-eq3-estimate-eq2} with $p'=p$. The  proof of (1) is thus completed.

\smallskip

(2) As it is mentioned in the above, to prove \eqref{prop2-eq2}, it suffices to prove \eqref{proof-prop2-eq2}.
Note that when $p>\frac{n}{4}$ and $p'=2p$,
$$
\int_0^\infty e^{-\lambda s} s^{-\big(\frac{1}{p}-\frac{1}{p'}\big)\frac{n}{2}}ds<\infty.
$$
 Then,
in the case  $(u,v)$ is the solution of \eqref{main-eq1},  \eqref{proof-prop2-eq2} follows from   \eqref{main-eq1-estimate-eq1} with $p'=2p$.
In the case that  $(u,v)$ is the solution of \eqref{main-eq2},   \eqref{proof-prop2-eq2} follows from \eqref{proof-prop2-eq1} and \eqref{main-eq2-estimate-eq1} with $p'=2p$.  In the case that  $(u,v)$ is the solution of \eqref{main-eq3},   \eqref{proof-prop2-eq2} follows from \eqref{proof-prop2-eq1} and \eqref{main-eq3-estimate-eq1} with $p'=2p$.

\smallskip

(3)   As it is mentioned in the above, to prove \eqref{prop2-eq3}, it suffices to prove \eqref{proof-prop2-eq3}.  Note that when $p>\frac{n}{2}$ and $p'=2p$,
$$
\int_0^\infty e^{-\lambda s}s^{-\frac{1}{2}-\big(\frac{1}{p}-\frac{1}{p'}\big)\frac{n}{2}} ds<\infty.
$$
Hence,
in the case  $(u,v)$ is the solution of \eqref{main-eq1},  \eqref{proof-prop2-eq3} follows from  \eqref{proof-prop2-eq1} and  \eqref{main-eq1-estimate-eq2} with $p'=2p$.
In the case that  $(u,v)$ is the solution of \eqref{main-eq2},   \eqref{proof-prop2-eq3} follows from \eqref{proof-prop2-eq1}, \eqref{proof-prop2-eq2},  and \eqref{main-eq2-estimate-eq2} with $p'=2p$.  In the case that  $(u,v)$ is the solution of \eqref{main-eq3},   \eqref{proof-prop2-eq3} follows from \eqref{proof-prop2-eq1}, \eqref{proof-prop2-eq2},  and \eqref{main-eq3-estimate-eq2} with $p'=2p$.
\end{proof}

In the following proposition, we show that the $W^{1,p}$ estimate of $v$ obtained from the previous proposition can be used to improve the $L^p$-norm of $u$.

\begin{prop}\label{basic-prop3}
 Consider \eqref{main-eq1}-\eqref{main-eq3}.
  Let $(u(t,x;u_0,v_0),v(t,x;u_0,v_0))$ be the solution of \eqref{main-eq1} or \eqref{main-eq2},
and $(u(t,x;u_0)$, $v(t,x;u_0))$ be the solution of \eqref{main-eq3}.
Put
$$
(u(t,x),v(t,x))=(u(t,x;u_0,v_0),v(t,x;u_0,v_0))\quad {\rm or}\quad
(u(t,x),v(t,x))=(u(t,x;u_0),v(t,x;u_0))
$$
and
$$
T_{\max}=T_{\max}(u_0,v_0)\quad {\rm or}\quad T_{\max}(u_0).
$$
Assume that there is $p>\max\{1,\frac{n}{2}\}$  such that
\begin{equation}
\label{prop2-eq00}
\sup_{t\in [0,T_{\max}),x_0\in\R^n}\int_{B(x_0,1)} u^p(t,x)dx<\infty\quad{\rm and}\quad  \sup_{x_0\in\R^n}\int_{B(x_0,1)} (u^{2p}_0(x)+ |\nabla v_0(x)|^{2p})dx<\infty.
\end{equation}
Then there is  $\gamma>n$ such that
\begin{equation}
\label{general-thm-proof-eq1}
\sup_{t\in [0,T_{\max})),x_0\in\R^n}\int_{B(x_0,1)} u^\gamma(t,x)dx<\infty.
\end{equation}
\end{prop}

\begin{proof} First of all, note that
if $p>n$, nothing needs to be proved.  When $n=1$, $p>\max\{1,\frac{n}{2}\}$ implies that $p>1=n$. Therefore, in the following, we assume that
$$
n\ge 2\quad {\rm and}\quad \frac{n}{2}<p\le n.
$$
Let $\psi$ be as in \eqref{psi-eq00}.  Then $u_0\in L^\gamma(B(x_0,1))$ for any $x_0\in\R^n$ and $\gamma>n$,  and
to prove \eqref{general-thm-proof-eq1}, it suffices to prove
\begin{equation}
\label{general-thm-proof-eq2}
\sup_{t\in[0,T_{\max}(u_0,v_0)),x_0\in\R^n}\int_{\R^n} u^\gamma (t,x+x_0;u_0,v_0)\psi(x)dx<\infty
\end{equation}
for some $\gamma>n$.
If no confusion occurs, we put
$$
u=u(t,x+x_0;u_0,v_0),\quad v=v(t,x+x_0;u_0,v_0).
$$

Next, fix $\gamma\in \big(\max\{n,\frac{2p(p-1)}{n}\},2p\big)$.
Note that, by Young's inequality,
\begin{align}
\label{new-eq0}
\frac{1}{\gamma}\frac{d}{dt}\int_{\R^n}u^\gamma \psi &=\int_{\R^n} u^{\gamma-1} \psi u_t\nonumber\\
&=\int_{\R^n} u^{\gamma -1}\psi \Delta u-\int_{\R^n}u^{\gamma-1}\psi \nabla\cdot(\chi u\nabla v)+a\int_{\R^n} u^\gamma \psi- b\int_{\R^n}u^{\gamma+1}\psi\nonumber\\
&=-(\gamma-1)\int_{\R^n}u^{\gamma-2}\psi|\nabla u|^2-\int_{\R^n}u^{\gamma-1}\nabla \psi\cdot\nabla u +\chi(\gamma-1)\int_{\R^n}
u^{\gamma-1}\psi\nabla u\cdot\nabla v\nonumber\\
&\quad  +\chi\int_{\R^n} u^\gamma \nabla v\cdot\nabla \psi+a\int_{\R^n} u^\gamma \psi- b\int_{\R^n}u^{\gamma+1}\psi\nonumber\\
&\le -\Big(\frac{\gamma-1}{2}-\frac{\k_1}{2}\Big)\int_{\R^n}u^{\gamma-2}\psi|\nabla u|^2+\frac{\chi^2(\gamma-1)}{2}\int_{\R^n}u^\gamma|\nabla v|^2\psi\nonumber\\
&\quad +\chi\int_{\R^n} u^\gamma \nabla v\cdot\nabla \psi +(a+\frac{\k_1}{2})\int_{\R^n} u^\gamma \psi- b\int_{\R^n}u^{\gamma+1}\psi\nonumber\\
&\le -\Big(\frac{\gamma-1}{2}-\frac{\k_1}{2}\Big)\int_{\R^n}u^{\gamma-2}\psi|\nabla u|^2+\Big(\frac{\chi^2(\gamma-1)}{2}+\frac{\k_1|\chi|}{2}\Big)\int_{\R^n}u^\gamma|\nabla v|^2\psi\nonumber\\
&\quad +\big(a+\frac{\k_1(|\chi|+1)}{2}\big)\int_{\R^n} u^\gamma \psi- b\int_{\R^n}u^{\gamma+1}\psi.
\end{align}
By H\"older inequality, we have
\beq\lb{3.25}
\int_{\R^n}u^\gamma |\nabla v|^2\psi\le \Big(\int_{\R^n} u^{\frac{\gamma p}{p-1}}\psi^{\frac{\alpha p}{p-1}}\Big)^{\frac{p-1}{p}}\Big(\int_{\R^n}|\nabla v|^{2p}\psi^{\beta p}\Big)^{\frac{1}{p}},
\eeq
for any $\alpha,\beta\in (0,1)$ with $\alpha+\beta=1$.

In the following, we estimate $\Big(\int_{\R^n} u^{\frac{\gamma p}{p-1}}\psi^{\frac{\alpha p}{p-1}}\Big)^{\frac{p-1}{p}}$.
Let {$\{x_k\}_{k\ge 1}=\{ x/\sqrt{n}\,|\, x\in\Z^n\}$} and so
$$
\R^n=\cup_{k=1}^\infty B(x_k, 1):=\cup_{k=1}^\infty R_k,
$$
where $R_k$ is an open ball centred at $x_k$ with radius $1$ and each $x\in\R^n$ belongs to at most  $M$ balls $R_k$ with $M$ depending only on $n$.
Note that there is $C_1>1$ such that
\beq\lb{3.26}
\frac{1}{C_1}\psi(x)\le \psi(x_k)\le C_1\psi(x)\quad \forall\, x\in R_k,\,\,\, k=1,2,\cdots.
\eeq
Let  $\xi_k(x)$ be  the characteristic function on $R_k$. Then we have
\begin{align}
\label{new-eq1}
\Big(\int_{\R^n} u^{\frac{\gamma p}{p-1}}\psi^{\frac{\alpha p}{p-1}}\Big)^{\frac{p-1}{p}}
&\le \Big(\int_{\R^n}\sum_{k=1}^\infty \xi_k(x)u^{\frac{\gamma p}{p-1}}\psi^{\frac{\alpha p}{p-1}}(x)\Big)^{\frac{p-1}{p}}
\le \sum_{k=1}^\infty \Big(\int_{\R^n}\xi_k(x) u^{\frac{\gamma p}{p-1}}\psi^{\frac{\alpha p}{p-1}}(x)\Big)^{\frac{p-1}{p}}\nonumber\\
&\le C_1^\alpha \sum_{k=1}^\infty \psi^{\alpha }(x_k) \big(\int_{R_k}u^{\frac{\gamma p}{p-1}} \big)^{\frac{p-1}{p}}.
\end{align}

Now we estimate  $\big(\int_{R_k}u^{\frac{\gamma p}{p-1}} \big)^{\frac{p-1}{p}}$.
To this end,
let
\begin{equation}
\label{theta-eq}
 \theta:=\frac{n\gamma-n p+n}{n\gamma -np+2p}.
\end{equation}
Note that $\gamma>n\geq p$. Thus it can be verified directly that
$$
\frac{p-1}{2p}=\theta\Big(\frac{1}{2}-\frac{1}{n}\Big)+(1-\theta)\frac{\gamma}{2p}
\quad {\rm and}\quad \frac{p-1}{p}<\theta<1.
$$
By  Gagliardo-Nirenberg inequality (see \cite[Theorem 1]{Nir}),  there is $C>0$ such that
\begin{align*}
\Big(\int_{R_k} u^{\frac{\gamma p}{p-1}} dx\Big)^{\frac{p-1}{p}}=\|u^{\frac{\gamma}{2}}\|_{L^{\frac{2p}{p-1}}(R_k)}^2\le C \Big(\int_{R_k}\big|\nabla ( u^{\gamma/2})|^2 dx\Big)^{\theta}\Big(\int_{R_k} u^p\Big)^{\frac{(1-\theta)\gamma }{p}}+C\Big(\int_{R_k}u^p\Big)^{\gamma/p}
\end{align*}
for all $k\ge 1$.
This together with  the assumption \eqref{prop2-eq0} implies that  there is $C_2>0$ such that
$$
\Big(\int_{R_k} u^{\frac{\gamma p}{p-1}} dx\Big)^{\frac{p-1}{p}} \le C_2 \Big(\int_{R_k}\big|\nabla ( u^{\gamma/2})|^2 dx\Big)^{\theta}+C_2.
$$

Taking
$\alpha:=\frac{1+\theta}{2}$, and then \eqref{new-eq1} and \eqref{3.26} yield
\begin{align}
\label{new-eq1-2}
&\Big(\int_{\R^n} u^{\frac{\gamma p}{p-1}}\psi^{\frac{\alpha p}{p-1}}\Big)^{\frac{p-1}{p}}\le C_1^{\alpha} \sum_{k=1}^\infty \psi^{\alpha }(x_k) \Big(\int_{R_k}u^{\frac{\gamma p}{p-1}} \Big)^{\frac{p-1}{p}}\nonumber\\
&\qquad\qquad\le C_1^{\alpha} C_2\sum_{k=1}^\infty \psi^{\frac{1-\theta}{2}}(x_k)\Big(\int_{R_k}|\nabla (u^{\gamma/2})|^2 \psi(x_k) dx\Big)^\theta+C_1^\alpha C_2\sum_{k=1}^\infty \psi^\alpha (x_k)\nonumber\\
&\qquad\qquad\le C_1^{\alpha+\theta}C_2\sum_{k=1}^\infty \psi^{\frac{1-\theta}{2}}(x_k)\Big(\int_{R_k}|\nabla (u^{\gamma/2})|^2 \psi(x) dx\Big)^\theta+C_1^\alpha C_2\sum_{k=1}^\infty \psi^\alpha (x_k).
\end{align}
Since $0<\theta<1$, by Young's inequality, for any $\eps>0$, there is $C_\eps>0$ such that
$$
C_1^{\alpha+\theta}C_2\psi^{\frac{1-\theta}{2}}(x_k)\Big(\int_{R_k}\big|\nabla ( u^{\gamma/2})|^2 \psi(x)dx\Big)^{\theta}\le \eps \int_{R_k}\big|\nabla ( u^{\gamma/2})|^2\psi(x) dx+C_\eps \psi^{\frac{1}{2}}(x_k).
$$
This, together  with \eqref{new-eq1-2}, implies that there is $C>0$ such that
\begin{align}
\label{new-eq1-4}
\Big(\int_{\R^n} u^{\frac{\gamma p}{p-1}}\psi^{\frac{\alpha p}{p-1}}\Big)^{\frac{p-1}{p}}&\le \eps   \sum_{k=1}^\infty \int_{R_k}|\nabla(u^{\gamma/2})|^2\psi(x)dx+C_\eps \sum_{k=1}^\infty\psi^{\frac{1}{2}}(x_k)+C_1^\alpha C_2\sum_{k=1}^\infty \psi^\alpha (x_k). \nonumber\\
&\le \eps \int_{\R^n}|\nabla(u^{\gamma/2})|^2\psi(x)dx +C_\eps+C
\end{align}
{where in the last inequality, we also used the definition of $x_k$ and $\psi$.}

By \eqref{new-eq0}, \eqref{3.25} and \eqref{new-eq1-4}, we have
\begin{align}
\label{new-eq0-2}
&\frac{1}{\gamma}\frac{d}{dt}\int_{\R^n}u^\gamma \psi
\le -\Big(\frac{\gamma-1}{2}-\frac{\k_1}{2}\Big)\int_{\R^n}u^{\gamma-2}\psi|\nabla u|^2\nonumber\\
&\qquad +\Big(\frac{\chi^2(\gamma-1)}{2}+\frac{\k_1|\chi|}{2}\Big)\Big(\eps \int_{\R^n}|\nabla(u^{\gamma/2})|^2\psi dx + (C_\eps+C)\Big)\Big(\int_{\R^n}|\nabla v|^{2p}\psi^{\beta p}\Big)^{\frac{1}{p}}\nonumber\\
&\qquad +\Big(a+\frac{k(|\chi|+1)}{2}\Big)\int_{\R^n} u^\gamma \psi- b\int_{\R^n}u^{\gamma+1}\psi.
\end{align}
Note that
$$
\int_{\R^n}u^{\gamma-2}|\nabla u|^2\psi=\frac{4}{\gamma^2}\int_{\R^n}|\nabla (u^{\gamma/2})|^2\psi.
$$
By Proposition \ref{basic-prop2} (3),
$$
\sup_{t\in[0,T_{\max}(u_0,v_0)),x_0\in\R^n}\int_{\R^n}|\nabla v|^{2p}\psi^{\beta p}<\infty.
$$
By H\"older's inequality, we have
$$
\int_{\R^n} u^\gamma \psi\le \Big(\int_{\R^n} u^{\gamma+1} \psi\Big)^{\frac{\gamma}{\gamma+1}}\Big(\int_{\R^n}\psi\Big)^{\frac{1}{\gamma+1}}.
$$
Hence for $C_3:=(\frac{\chi^2(\gamma-1)}{2}+\frac{\k_1|\chi|}{2})$  we have
\begin{align}
\label{new-eq0-3}
&\frac{1}{\gamma}\frac{d}{dt}\int_{\R^n}u^\gamma \psi
\le -{\Big(\frac{2(\gamma-1)}{\gamma^2}-\frac{2\k_1}{\gamma^2}-C_3\eps \Big)}\int_{\R^n}\psi|\nabla(u^{\frac{\gamma}{2}})|^2\nonumber\\
&\qquad \quad+
 \Big(a+\frac{k(|\chi|+1)}{2}\Big)\int_{\R^n} u^\gamma \psi- {b}{\Big(\int_{\R^n}\psi\Big)^{-\frac1\gamma}}\Big(\int_{\R^n} u^{\gamma}\psi\Big)^{\frac{\gamma+1}{\gamma}} +CC_3(C_\eps+1).
\end{align}
Choose $0<\k_1\ll 1$ and $0<\eps\ll 1$ such that
$\frac{2(\gamma-1)}{\gamma^2}-\frac{2\k_1}{\gamma^2}-C_3\eps >0$.
Then by \eqref{new-eq0-3} and the comparison principle for scalar ODEs, we have
$$
\sup_{t\in[0,T_{\max}(u_0,v_0)),x_0\in\R^n} \int_{\R^n}u^\gamma(t,x+x_0;u_0,v_0)\psi(x)<\infty,
$$
that is, \eqref{general-thm-proof-eq2} holds, and then \eqref{general-thm-proof-eq1} holds.
\end{proof}

The last proposition of the subsection says that if $u$ is locally bounded in $L^p$, then $|\nabla v|$ is uniformly bounded in $L^\infty$.

\begin{prop}
\label{basic-prop4}
  Let $(u(t,x;u_0,v_0),v(t,x;u_0,v_0))$ be the solution of \eqref{main-eq1} or \eqref{main-eq2},
and $(u(t,x;u_0)$, $v(t,x;u_0))$ be the solution of \eqref{main-eq3}.
Put
$$
(u(t,x),v(t,x))=(u(t,x;u_0,v_0),v(t,x;u_0,v_0))\quad {\rm or}\quad
(u(t,x),v(t,x))=(u(t,x;u_0),v(t,x;u_0))
$$
and
$$
T_{\max}=T_{\max}(u_0,v_0)\quad {\rm or}\quad T_{\max}(u_0).
$$
{If for some $p>n$ we have}
\begin{equation}
\label{prop3-eq0}
\sup_{t\in [0,T_{\max}), x_0\in\R^n}\|u\|_{L^p(B(x_0,1)}<\infty,
\end{equation}
then
\begin{equation}
\label{prop3-eq1}
\sup_{t\in [t_0,T_{\max})}\|\nabla v(t,\cdot)\|_{\infty}<\infty\quad\text{ for any $0<t_0<T_{\max}$.}
\end{equation}
\end{prop}

\begin{proof}  In the case that    $(u(t,x), v(t,x))$ is the solution of \eqref{main-eq1} or \eqref{main-eq2}, \eqref{prop3-eq1} follows from
 \cite[Theorem 11.1]{LaSoUr}.
In the case that $(u(t,x), v(t,x))$ is the solution of \eqref{main-eq3}, \eqref{prop3-eq1} follows from \cite[Theorem 9.11]{GiTr}.
\end{proof}

\subsection{Proof of Theorem \ref{general-global-existence-thm}}

In this subsection,  we prove  Theorem \ref{general-global-existence-thm}.
First of all, put
$$
(u(t,x),v(t,x))=(u(t,x;u_0,v_0),v(t,x;u_0,v_0))\quad {\rm or}\quad
(u(t,x),v(t,x))=(u(t,x;u_0),v(t,x;u_0))
$$
and
$$
T_{\max}=T_{\max}(u_0,v_0)\quad {\rm or}\quad T_{\max}=T_{\max}(u_0).
$$
By Propositions \ref{local-existence-prop1} and \ref{local-existence-prop2},
 $u(t,\cdot)\in C_{\rm unif}^b(\R^n)$ and $v(t,\cdot)\in C_{\rm unif}^{1,b}(\R^n)$ for any $t\in (0,T_{\max})$.
Without loss of generality, we assume that $u_0\in C_{\rm unif}^b(\R^n)$ and
$v_0\in C_{\rm unif}^{1,b}(\R^n)$.

Next, by \eqref{general-cond-eq00} and   Propositions \ref{basic-prop2} -- \ref{basic-prop4},
\begin{equation}
\label{proof-thm0-eq1}
\sup_{t\in [0,T_{\max})}\|\nabla v(t,\cdot)\|_\infty<\infty.
\end{equation}
Note that
\begin{align*}
u(t,\cdot)&={e^{(\Delta -I)t}u_0}-\chi\int_0^t e^{(\Delta-I)(t-s)}\nabla \cdot(u(s,\cdot)\nabla v(s,\cdot))ds\nonumber\\
&\quad +\int_0^t e^{(\Delta-I)(t-s)}u(s,\cdot)(1+a-b u(s,\cdot))ds
\end{align*}
for $t\in [0,T_{\max})$. Note that  there is $M>0$ such  that
$$
u(s,x)(1+a-bu(s,x))\le M\quad \forall \, s\in [0,T_{\max}),\,\, x\in\R^n.
$$
Hence
\begin{align*}
0\le u(t,\cdot)\le {e^{(\Delta -I)t}u_0}-{\chi\int_0^t e^{(\Delta-I)(t-s)}\nabla \cdot(u(s,\cdot)\nabla v(s,\cdot))ds}+\int_0^t e^{(\Delta-I)(t-s)} Mds
\end{align*}
and then
\begin{align}
\label{proof-thm0-eq2}
\| u(t,\cdot)\|_\infty & \le \underbrace{\|e^{(\Delta -I)t}u_0\|_\infty}_{I_0(t)}+\underbrace{|\chi| \int_0^t \|e^{(\Delta-I)(t-s)}\nabla \cdot(u(s,\cdot)\nabla v(s,\cdot))\|_\infty ds}_{I_1(t)}\nonumber\\
&\quad +\underbrace{\int_0^t\| e^{(\Delta-I)(t-s)} M\|_\infty ds}_{I_2(t)}
\end{align}

It is direct to see that
$$
I_0(t)\le \|u_0\|_\infty,\quad I_2(t)\le M\quad \forall\, t\in [0,T_{\max}).
$$
As for $I_1(t)$, by  Lemma \ref{L-infinity-bound}, we have
$$
{ I_1(t)\le\frac{n  |\chi|}{\sqrt \pi}}  \sup_{r\in [0,T_{\max})}\|\nabla v(r,\cdot)\|_\infty \int_0^t e^{-(t-s)}(t-s)^{-\frac{1}{2}}\|u(s,\cdot)\|_\infty ds\quad \forall\, t\in [0,T_{\max}).
$$
It then follows from \eqref{proof-thm0-eq2} and \eqref{proof-thm0-eq1} that
\begin{equation}
\label{proof-thm0-eq3}
\|u(t,\cdot)\|_\infty\le \|u_0\|_\infty+M+C \int_0^t e^{-(t-s)}(t-s)^{-\frac{1}{2}}\|u(s,\cdot)\|_\infty ds
\end{equation}
for all $t\in [0,T_{\max})$.
This together with generalized Gronwall's inequality implies that
$$
\|u(t,\cdot)\|_\infty \le C \Big(\|u_0\|_\infty+M\Big) \int_0^t e^{-(t-s)}(t-s)^{-\frac{1}{2}}ds\quad \forall\, t\in [0,T_{\max}),
$$
which yields $
\sup_{t\in [0,T_{\max})}\|u(t,\cdot)\|_\infty <\infty
$
and then $T_{\max}=\infty$.
The theorem is thus proved.

\section{Global existence of classical solutions of \eqref{main-eq1}}

In this section, we study  the global existence of classical solutions of \eqref{main-eq1} and prove Theorem \ref{main-thm1}.

\begin{proof}[Proof of Theorem \ref{main-thm1}]
First of all,  for given   $u_0\in C_{\rm unif}^b(\R^n)$ and $v_0\in C_{\rm unif}^{1,b}(\R^n)$ or
$u_0\in L^p(\R^n)$ and $v_0\in W^{1,p}(\R^n)$, by Proposition \ref{local-existence-prop1},  for any $t\in (0,T_{\max}(u_0,v_0))$,
$u(t,\cdot;u_0,v_0)\in C_{\rm unif}^b(\R^n)$ and $v(t,\cdot;u_0,v_0)\in C_{\rm unif}^{1,b}(\R^n)$. It then suffices to prove the theorem for the case that $u_0\in C_{\rm unif}^b(\R^n)$ and $v_0\in C_{\rm unif}^{1,b}(\R^n)$.

Next, assume that  $u_0\in C_{\rm unif}^b(\R^n)$ and $v_0\in C_{\rm unif}^{1,b}(\R^n)$.  By Theorem \ref{general-global-existence-thm}, it suffices to prove that there exists some $p>\max\{1,\frac{n}{2}\}$ such that
\begin{equation}
\label{proof-thm1-eq1}
\sup_{t\in [0,T_{\max}(u_0,v_0)), x_0\in\R^n}\int_{B(x_0,1)} u^p(t,x;u_0,v_0)dx<\infty
\end{equation}
provided that
 \begin{equation}
\label{thm1-cond-eq1}
 |\chi|\cdot \|v_0\|_\infty< b \cdot\sup_{\gamma>\max\{1,n/2\}}\frac{\gamma}{\gamma-1}\Big(C_{\gamma+1,n}\Big)^{-\frac{1}{\gamma+1}}
\end{equation}
 or
\begin{equation}
\label{thm1-cond-eq2}
|\chi|\cdot \|v_0\|_\infty< { D^*_{\tau,n}}.
\end{equation}
Let $\psi\in C^\infty(\R^n)$  satisfy \eqref{psi-eq00} for some $\k_1,\k_2>0$.
To prove \eqref{proof-thm1-eq1}, it then suffices to prove
\begin{equation}
\label{proof-thm1-eq1-1}
\sup_{t\in [0,T_{\max}(u_0,v_0)),x_0\in\R^n} \int_{\R^n}u^p(t,x+x_0;u_0,v_0) \psi(x)dx<\infty.
\end{equation}
 If no confusion occurs, we put
$$
u=u(t,x+x_0;u_0,v_0) \quad {\rm and}\,\, v=v(t,x+x_0;u_0,v_0).
$$

We prove \eqref{proof-thm1-eq1-1}   in two steps.

\medskip

\noindent {\bf Step 1.}  {\it In this step, we  assume that
$|\chi|\cdot \|v_0\|<  b \cdot\sup_{\gamma>\max\{1,\frac{n}2\}}\frac{\gamma}{\gamma-1}\Big(C_{{ \gamma+1,n}}\Big)^{-\frac{1}{\gamma+1}}$, which is equivalent to  $b>\Big(\inf_{\gamma >\max\{1,\frac{n}{2}\}} \frac{\gamma-1}{\gamma}\big(C_{{\gamma+1,n}}\big)^{\frac{1}{\gamma+1}}\Big)|\chi|\|v_0\|_\infty$,
 and prove \eqref{proof-thm1-eq1-1}.}

\medskip

First, by the equation and Young's inequality, for any $p>1$ we have
\begin{align*}
\frac{1}{p} &\frac{d}{dt}\int_{\R^n}u^p\psi
=-(p-1)\int_{\R^n}u^{p-2}|\nabla
u|^2\psi-\int_{\R^n}u^{p-1}\nabla u \cdot \nabla\psi\\
&\quad  +\chi(p-1)\int_{\R^n}u^{p-1}\nabla u \cdot (\nabla
v )\psi +  \chi\int_{\R^n}u^{p}\nabla v \cdot \nabla
\psi +a\int_{\R^n}u^p\psi-b\int_{\R^n}u^{p+1}\psi\\
&\le -(p-1)\int_{\R^n}u^{p-2}|\nabla
u|^2\psi + \frac{\k_1}{2} \int_{\R^n}u^{p-2}|\nabla u|^2\psi + \frac{\k_1}{2} \int_{\R^n}u^p\psi\\
&\quad +\frac{\chi(p-1)}{p}\int_{\R^n} \nabla (u^p)\cdot(\nabla v)\psi +\frac{|\chi| \k_1 p}{p+1}\int_{\R^n} u^{p+1}\psi+\frac{|\chi| \k_1 }{p+1}\int_{\R^n} |\nabla v|^{p+1}\psi\\
&\quad +a\int_{\R^n} u^p\psi-b\int_{\R^n} u^{p+1}\psi\\
&=-\big(p-1-\frac{\k_1}{2})\int_{\R^n}u^{p-2}|\nabla u|^2\psi  -\frac{\chi (p-1)}{p}\int_{\R^n} u^p(\Delta v) \psi-\frac{\chi(p-1)}{p}\int_{\R^n} u^p\nabla v\cdot\nabla \psi\\
&\quad +\frac{|\chi|\k_1}{p+1}\int_{\R^n}|\nabla v|^{p+1}\psi+\big(a+\frac{\k_1}{2}\big)\int_{\R^n} u^p\psi-\big(b-\frac{|\chi|\k_1 p}{p+1}\big)\int_{\R^n} u^{p+1}\psi\\
&\le -\big(p-1-\frac{\k_1}{2})\int_{\R^n}u^{p-2}|\nabla u|^2\psi  +\frac{|\chi| (p-1)}{p}\int_{\R^n}u^p|\Delta v| \psi-\frac{p+1}{\tau p}\int_{\R^n}u^p\psi\\
&\quad +\Big(\frac{|\chi|\k_1}{p+1}+\frac{|\chi|\k_1 (p-1)}{p (p+1)}\Big)\int_{\R^n}|\nabla v|^{p+1}\psi+\big(a+\frac{\k_1}{2}+\frac{p+1}{\tau p}\big)\int_{\R^n} u^p\psi\\
&\quad -\Big(b-\frac{|\chi|\k_1 p}{p+1}-\frac{|\chi|\k_1(p-1)}{p+1}\Big)\int_{\R^n} u^{p+1}\psi.
\end{align*}

 Let $r>0$ { be determined later}. By Young's inequality again, we have
  \[\frac{|\chi|(p-1)}{p}\int_{\R^n}u^p |\Delta v|\psi \leq r \int_{\R^n}u^{p+1}\psi +\underbrace{\frac{1}{p}\big(\frac{p-1}{p+1}\big)^{p+1} }_{A_p} r^{-p}|\chi|^{p+1}\int_{\R^n}|\Delta v|^{p+1}\psi\]
and for any $\eps>0$,
 \begin{align*}
 \int_{\R^n}\Big[a+\frac{\k_1}{2}+\frac{p+1}{\tau p}\Big]u^p\psi \leq \eps\int_{\R^n}u^{p+1}\psi +\underbrace{\frac{1}{p+1}\Big[\frac{p+1}{p}\eps\Big]^{-p}\Big[a+\frac{\k_1}{2}+\frac{p+1}{\tau p}\Big]^{p+1}\int_{\R^n}\psi}_{C_{\eps}}.
 \end{align*}
We then have
\begin{align*}
\frac{1}{p}\frac{d}{dt}\int_{\R^n}u^p\psi+\frac{p+1}{\tau p}\int_{\R^n}u^p\psi&\le   -\big(p-1-\frac{\k_1}{2})\int_{\R^n}u^{p-2}|\nabla u|^2\psi  +A_p r^{-p}|\chi|^{p+1}\int_{\R^n}|\Delta v|^{p+1} \psi \\
&\quad +\Big(\frac{|\chi|\k_1}{p+1}+\frac{|\chi|\k_1 (p-1)}{p (p+1)}\Big)\int_{\R^n}|\nabla v|^{p+1}\psi\\
&\quad -\big(b-\frac{|\chi|\k_1 p}{p+1}-\frac{|\chi|\k_1(p-1)}{p+1}-r-\eps\big)\int_{\R^n} u^{p+1}\psi+C_{\eps}.
\end{align*}
This implies that for any $0<t_0<t<T_{\max}$,
\begin{equation}\label{f-s-eq0}
\int_{\R^n}u^p(t,\cdot)\psi dx\le  e^{\frac{(p+1) (t_0-t)}{\tau}}\int_{\R^n}u^p(t_0,\cdot)\psi dx+p\int_{t_0}^t
e^{\frac{(p+1) (s-t)}{\tau}} f(s)ds,
\end{equation}
where
\begin{align}
\label{f-s-eq1}
f(s)&:=A_p r^{-p}|\chi|^{p+1}\int_{\R^n}|\Delta v(s,\cdot)|^{p+1} \psi dx+\Big(\frac{|\chi|\k_1}{p+1}+\frac{|\chi|\k_1 (p-1)}{p (p+1)}\Big)\int_{\R^n}|\nabla v(s,\cdot)|^{p+1}\psi dx\nonumber  \\
&\quad -\Big(b-\frac{|\chi|\k_1 p}{p+1}-\frac{|\chi|\k_1(p-1)}{p+1}-r-\eps \Big)\int_{\R^n} u^{p+1}(s,\cdot)\psi dx+C_{\eps}.
\end{align}

 Let $\psi_1 := \psi^{\frac{1}{p+1}}$. Then  $v\psi_1$ solves
$$
\tau (v\psi_1)_t =\Delta(v\psi_1)-v\psi_1+\left(v\psi_1 -2\nabla v\cdot \nabla\psi_1 + v\Delta \psi_1 - uv \psi_1\right).
$$
By Lemma \ref{maximal-regularity-lm} with $\gamma:=p+1$ and $\alpha:=1$, we have
\begin{equation}
\begin{aligned}
\label{f-s-eq2}
&\int_{t_0}^t e^{\frac{(p+1) s}{\tau}} \int_{\R^n} (v(s,x)\psi_1(x))^{p+1}+|\nabla (v(s,x)\psi_1(x))|^{p+1}+|\Delta (v(s,x) \psi_1(x))|^{p+1} dx ds\\
&\qquad\le C_{{p+1,n}}\int_{t_0}^t  \int_{\R^n} e^{\frac{(p+1)s}{\tau}} \Big[ (v\psi_1 -2\nabla v\cdot \nabla\psi_1 + v\Delta \psi_1 - uv\psi_1)(s,x)\Big]^{p+1}dx ds\\
&\qquad+ C_{{ p+1,n}}(t+\tau^{p+1}t_0^{-p})\|v_0(\cdot)
\psi_1\|^{p+1}_{L^{p+1}(\R^n)} 
\end{aligned}
\end{equation}
Note that
\begin{align*}
\int_{\R^n}(|\nabla v|\psi_1)^{p+1}\le \int_{\R^n}\big(|\nabla(v\psi_1)|+v|\nabla \psi_1|\big)^{p+1}\le \int_{\R^n} \big(|\nabla (v\psi_1)|+\frac{\k_1}{p+1} v\psi_1\big)^{p+1}
\end{align*}
and
\begin{align*}
\int_{\R^n}|\Delta v\psi_1|^{p+1}&\le \int_{\R^n}\Big(|\Delta(v\psi_1)|+2|\nabla v||\nabla \psi_1|+v|\Delta \psi_1|\Big)^{p+1}\\
&\le \int_{\R^n}\Big[|\Delta (v\psi_1)|+\frac{2\k_1}{p+1}|\nabla v|\psi_1+\big(\frac{p \k_1^2}{(p+1)^2}+\frac{\k_1}{p+1}\big)v\psi_1\Big]^{p+1}.
\end{align*}
Then for any $\eps>0$, by choosing $0<\k_1\ll 1$, we have
\beq\lb{4.8}
\begin{aligned}
    \int_{\R^n}(|\nabla v|\psi_1)^{p+1}&\le (1+\eps)\int_{\R^n}|\nabla(v\psi_1)|^{p+1}+\frac{\eps}{2}\int_{\R^n}(v\psi_1)^{p+1},\\
\int_{\R^n}(|\Delta v|\psi_1)^{p+1}&\le (1+\eps)\int_{\R^n}|\Delta(v\psi_1)|^{p+1}
+\eps\int_{\R^n}(|\nabla v|\psi_1)^{p+1}+\frac{\eps}{2}\int_{\R^n}(v\psi_1)^{p+1},
\end{aligned}
\eeq
and
\beq\lb{4.9}
\begin{aligned}
&\int_{\R^n} \left( (v\psi_1 -2\nabla v\cdot\nabla\psi_1+ v\Delta \psi_1 - uv \psi_1)(s,x)\right)^{p+1}dx\\
&\qquad\quad\le (1+\eps)\int_{\R^n}(v\psi_1)^{p+1}+
\eps \int_{\R^n}(|\nabla v|\psi_1)^{p+1}+(\|v_0\|^{p+1}_\infty+\eps)\int_{\R^n}(u\psi_1)^{p+1}.
\end{aligned}
\eeq
Using \eqref{4.8} yields
\begin{align*}
\int_{t_0}^t& \int_{\R^n}e^{\frac{(p+1) s}{\tau}}\big( (v\psi_1)^{p+1}
+(|\nabla v|\psi_1)^{p+1}+(|\Delta v|\psi_1)^{p+1}\big)(s,x)dx ds\nonumber\\
&\le (1+\eps)^2\int_{t_0}^t \int_{\R^n} e^{\frac{(p+1) s}{\tau}}\big((v\psi_1)^{p+1}+|\nabla (v\psi_1)|^{p+1}+|\Delta (v \psi_1)|^{p+1} \big)(s,x)dx ds.
\end{align*}
By  \eqref{f-s-eq2} and \eqref{4.9}, the above
\begin{align*}
&\le (1+\eps)^2C_{{ p+1,n}}\int_{t_0}^t  \int_{\R^n} e^{\frac{(p+1) s}{\tau}}\Big[ (v\psi_1-2\nabla v\cdot\nabla\psi_1 + v\Delta \psi_1 - uv\psi_1)(s,x)\Big]^{p+1}dx ds\nonumber\\
&\qquad + (1+\eps)^2C_{{ p+1,n}}(t+\tau^{p+1}t_0^{-p})\|v_0(\cdot)
\psi_1\|^{p+1}_{L^{p+1}(\R^n)}\\
&\le (1+\eps)^2 C_{{ p+1,n}}\int_{t_0}^t \int_{\R^n}e^{\frac{(p+1) s}{\tau}}\left((1+\eps)(v\psi_1)^{p+1}+\eps(|\nabla v|\psi_1)^{p+1}+(\|v_0\|_\infty^{p+1}+\eps)(u\psi_1)^{p+1}\right)\\
&\qquad + (1+\eps)^2C_{{ p+1,n}}(t+\tau^{p+1}t_0^{-p})\|v_0(\cdot)
\psi_1\|^{p+1}_{L^{p+1}(\R^n)}.
\end{align*}
Recall that $\psi_1^{p+1}=\psi$. This implies that
\begin{equation*}
\begin{aligned}
&\int_{t_0}^t\int_{\R^n}e^{\frac{(p+1) s}{\tau}}\big(|\nabla v|^{p+1}\psi+|\Delta v|^{p+1}\psi \big)dx ds\\
&\quad \le  (1+\eps)^2 C_{{ p+1,n}}\int_{t_0}^t   \int_{\R^n}  e^{\frac{(p+1) s}{\tau}}\Big(\eps |\nabla v|^{p+1}\psi+(\|v_0\|_\infty^{p+1}+\eps)u^{p+1}\psi\Big)dxds+C^*_{1}(t),
\end{aligned}
\end{equation*}
where
\begin{align}
    \label{C*1}
C^*_{1}(t):=&(1+\eps)^3 C_{{p+1,n}}\|v_0\|_\infty^{p+1} \int_{t_0}^t \int_{\R^n}e^{\frac{(p+1) s}{\tau}} \psi dxds \nonumber\\
&+(1+\eps)^2 C_{{p+1,n}}(t+\tau^{p+1}t_0^{-p})\|v_0(\cdot)
\psi_1\|^{p+1}_{L^{p+1}(\R^n)}.
\end{align}
Now by taking $\eps$ to be small such that $(1+\eps )^2\eps C_{{ p+1,n}}\leq \frac12$, we obtain from this
\begin{equation}\label{4.1}
\int_{t_0}^t\int_{\R^n}e^{\frac{(p+1) s}{\tau}}|\nabla v|^{p+1}\psi dx ds \le 2C^*_{2}\int_{t_0}^t   \int_{\R^n} e^{\frac{(p+1) s}{\tau}} u^{p+1}\psi dxds +2C^*_{1}(t)
\end{equation}
where
\[
C^*_{2}:={(1+\eps)^2 C_{{ p+1,n}}(\|v_0\|_\infty^{p+1}+\eps)}.
\]
We also get
\begin{equation}\label{f-s-eq3'}
 \begin{aligned}
\int_{t_0}^t\int_{\R^n}e^{\frac{(p+1) s}{\tau}}|\Delta v|^{p+1}\psi dx ds
\le  C^*_{2}\int_{t_0}^t   \int_{\R^n}e^{\frac{(p+1) s}{\tau}} u^{p+1}\psi dxds +C^*_{1}(t).
\end{aligned}
\end{equation}

Choose $0<\k_1\ll 1$ such that
$$
\frac{|\chi|\k_1}{p+1}+\frac{|\chi|\k_1 (p-1)}{p (p+1)}< \eps
\quad {\rm and}\quad
\frac{|\chi|\k_1 p}{p+1}+\frac{|\chi|\k_1(p-1)}{p+1}<\eps.
$$
Then it follows from \eqref{f-s-eq0}, \eqref{f-s-eq1}, and \eqref{f-s-eq3'} that
\begin{equation}
    \begin{aligned}
\label{f-s-eq4}
&\quad \int_{\R^n}u^p(t,x)\psi dx\\
&\le  e^{\frac{(p+1)(t_0-t)}{\tau}}\int_{\R^n}u^p(t_0,x)\psi dx+pA_p  r^{-p}|\chi|^{p+1}\int_{t_0}^t\int_{\R^n}e^{\frac{(p+1)(s-t)}{\tau}}|\Delta v|^{p+1} \psi dxds \\
&\quad+p\eps\int_{t_0}^t\int_{\R^n}e^{\frac{(p+1)(s-t)}{\tau}}|\nabla v|^{p+1}\psi dxds   -p\big(b-r-2\eps \big)\int_{t_0}^t\int_{\R^n}e^{\frac{(p+1)(s-t)}{\tau}}u^{p+1}\psi dxds\\
&\quad +C_{\eps} p\int_{t_0}^t  e^{\frac{(p+1)(s-t)}{\tau}}ds\\
&\le  \int_{\R^n}u^p(t_0,\cdot)\psi dx+  \left(pA_p r^{-p}|\chi|^{p+1}+2p\eps\right)\left[C^*_{2}\int_{t_0}^t   \int_{\R^n}e^{\frac{(p+1)(s-t )}\tau} u^{p+1}\psi dxds +e^{-\frac{(p+1) t}\tau}C^*_{1}(t)\right]\\
&\quad -p\big(b-r-2\eps\big)\int_{t_0}^t\int_{\R^n} e^{\frac{(p+1)(s-t )}\tau}u^{p+1}\psi dxds+{ \frac{C_{\eps}p\tau}{p+1}}\\
&=-p\Big(b-r-4\eps  -A_pr^{-p}|\chi|^{p+1}C^*_{2}\Big)\int_{t_0}^t \int_{\R^n}e^{(p+1)(s-t )}u^{p+1}\psi dxds+C^*_{3}(t),
\end{aligned}
\end{equation}
where
\begin{align*}
C^*_{3}(t):= \frac{C_{\eps}p\tau}{p+1}+ \left(pA_pr^{-p}|\chi|^{p+1}+2p\eps\right)e^{-\frac{(p+1)t}\tau}C^*_{1}(t)+\left\|u(t_0,\cdot)\psi^{1/p}\right\|_{L^{p}}^{p}.
\end{align*}
Direct computation yields $C^*_{1}(t)\leq Ce^{\frac{(p+1)t}\tau}+C(t+\tau^{p+1}t_0^{-p})$ by \eqref{C*1}  for some $C>0$. It is then clear that $C^*_{3}(t)$ is uniformly bounded for all $t\geq t_0$, depending on $r,t_0,p,\tau,\chi,\eps,u(t_0)$ and $\|v_0\|_\infty$ of course.

Now note that $\left|C^*_{2}- C_{{ p+1,n}}\|v_0\|_\infty^{p+1}\right|=O(\eps)$ and
$$
\min_{r>0} \left(A_p  r^{-p}|\chi|^{p+1}C_{{ p+1,n}}\|v_0\|_\infty^{p+1}+r\right)=\Big(1+\frac1p\Big)A_p^\frac{1}{p+1}\big(p\,C_{{ p+1, n}}\big)^{\frac{1}{p+1}}|\chi|\, \|v_0\|_\infty.
$$
So, in view of the condition $b>\Big[\inf_{p >\max\{1,\frac{n}{2}\}}  { \left(1+\frac1p\right)} A_p^\frac{1}{p+1}\big(p\, C_{{ p+1, n}}\big)^{\frac{1}{p+1}}\Big]|\chi|\|v_0\|_\infty$, there are $r>0$ and
$p>\max\{1,\frac{n}{2}\}$ such that
$$
b>r+A_pr^{-p}|\chi|^{p+1}  C_{{ p+1, n}}\|v_0\|_\infty^{p+1}.
$$
Fix such $r>0$ and $p>\max\{1,\frac{n}{2}\}$. Then by selecting $\eps\ll 1$, we obtain
$$
b-r-4\eps -A_pr^{-p}|\chi|^{p+1}C^*_{2}>0.
$$
Then by \eqref{f-s-eq4},
$$
\sup_{t\in [t_0,T_{\max}),x_0\in\infty}\int_{\R^n}u^p(t,x+x_0)\psi(x)<\infty.
$$
This implies that  \eqref{proof-thm1-eq1-1} holds provided that \eqref{thm1-cond-eq1}
holds.

\medskip

\noindent {\bf Step 2.} {\it In this step, we  assume that
$|\chi|\cdot  \|v_0\|_\infty< { D^*_{\tau,n}}$ and prove \eqref{proof-thm1-eq1-1}.}
\smallskip

To this end,
define $\varphi(s) := e^{\sigma s^2}$ for $0\leq s \leq \|v_0\|_{\infty}$,  where
 $\sigma>0$ is to be determined later.
Recall $n^*=\max\{1,\frac{n}{2}\}$,
$\tau^*=\frac{1}{\tau}-1$, and $ j=\text{Sign}(\chi\tau^*)$.
Due to the condition, we can choose a {$p>\max\{1,\frac{n}{2}\}$} such that
\begin{equation}
\label{new-d-tau-n-eq}
|\chi|\cdot \|v_0\|_\infty<\begin{cases}  \frac{2}{\tau p (2\alpha+j|\tau^*|)}\quad &{\rm if}\quad \alpha>-j|\tau^*|\cr\cr
\frac{2}{ \tau p\alpha}\quad &{\rm if}\quad \alpha\leq -j|\tau^*|
\end{cases}
\end{equation}
where $\alpha:=\sqrt{\frac{(\tau^*)^2}4+\frac1{\tau p}}.$
Also let $\psi\in C^\infty(\R^n)$  satisfy \eqref{psi-eq00} for some $\k_1,\k_2>0$.
To prove \eqref{proof-thm1-eq1-1}, it then suffices to prove
\begin{equation}
\label{proof-thm1-eq2}
\sup_{t\in [0,T_{\max}(u_0,v_0)),x_0\in\R^n} \int_{\R^n}u^p(t,x+x_0;u_0,v_0) \varphi(v(t,x+x_0;u_0,v_0)) \psi(x)dx<\infty.
\end{equation}

We now prove \eqref{proof-thm1-eq2}.  If no confusion occurs,  for given $x_0\in\R^n$, we put
$$
u=u(t,x+x_0;u_0,v_0) \quad {\rm and}\,\, v=v(t,x+x_0;u_0,v_0).
$$
Observe that
 \[
\begin{aligned}
&\quad\,\,\frac{1}{p}\frac{d}{dt}\int_{\R^n}u^{p}\varphi(v)\psi(x)\nonumber\\
 &= \int_{\R^n} u_tu^{p-1}\varphi(v) \psi(x)+ \frac{1}{p}\int_{\R^n}u^p\varphi^{\prime}(v)\psi(x)v_t \nonumber\\
& = \int_{\R^n}(\Delta u - \chi\nabla \cdot (u\nabla v) + au - bu^2) u^{p-1}\varphi(v) \psi(x)+\frac{1}{{\tau}p}\int_{\R^n}u^p\varphi^{\prime}(v)(\Delta v - uv ) \psi(x)\nonumber \\
&=  -(p-1) \int_{\R^n}u^{p-2}|\nabla u|^2\varphi(v)\psi(x) - \int_{\R^n}\varphi^{\prime}(v) u^{p-1}\nabla u \cdot(\nabla v) \psi(x)-\int_{\R^n}u^{p-1}\nabla u\cdot(\nabla \psi) \varphi(v)\nonumber\\
&\quad + \chi(p-1)\int_{\R^n} u^{p-1}\varphi(v)\nabla u\cdot (\nabla v)  \psi(x)
 + \chi\int_{\R^n}u^p|\nabla v|^2\varphi^{\prime}(v)\psi(x)+
\chi\int_{R^n} u^p \varphi(v)\nabla v\cdot\nabla \psi\nonumber\\
&\quad  + \int_{\R^n}(au^{p} -bu^{p+1})\varphi(v)\psi(x) - \frac{1}{{\tau}p}\int_{\R^n}u^p\varphi^{\prime\prime}(v)|\nabla v|^2 \psi(x)-{\frac{1}{\tau}} \int_{\R^n}u^{p-1}\varphi^{\prime}(v)\nabla u\cdot (\nabla v)\psi(x) \nonumber\\
&\quad -\frac{1}{{\tau}p}\int_{\R^n} u^p\varphi'(v)\nabla v\cdot \nabla \psi - \frac{1}{{\tau}p}\int_{\R^n}u^{p+1}v\varphi^{\prime}(v)\psi(x).
\end{aligned}
\]
Since $v\geq 0$, $\varphi^{\prime}(s)=2\sigma s\varphi(s) \geq 0$ for all $s\geq 0$, and $\psi(x) >0$, we have
\begin{align*}
    &\quad\,\,\frac{1}{p}\frac{d}{dt}\int_{\R^n}u^{p}\varphi(v) \psi\\
   &\leq   -(p-1) \int_{\R^n}u^{p-2}|\nabla u|^2\varphi(v) \psi- \int_{\R^n}({ 2\sigma v+\frac{2\sigma}{\tau}v} -\chi(p-1))\varphi(v) u^{p-1}\nabla u \cdot (\nabla v)\psi\nonumber  \\
&\quad+ \chi\int_{\R^n}u^p|\nabla v|^2\varphi^{\prime}(v)\psi + \int_{\R^n}(au^{p} -bu^{p+1})\varphi(v)\psi - \frac{1}{{\tau}p}\int_{\R^n}u^p\varphi^{\prime\prime}(v)|\nabla v|^2\psi \nonumber\\
&\quad  -\int_{\R^n} u^{p-1}\varphi(v)\nabla u\cdot\nabla \psi +\int_{\R^n} (\chi - \frac{2\sigma v}{{\tau p}})u^p\varphi(v)\nabla v\cdot\nabla \psi.
\end{align*}
By Young's inequality, we have for any  $\eps,\eps'\in(0,1)$,
\begin{align*}
&-\int_{\R^n}({ 2\sigma v+\frac{2\sigma}{\tau}v} -\chi(p-1))\varphi(v) u^{p-1}\nabla u \cdot (\nabla v)\psi \leq
(1-\eps)(p-1)\int_{\R^n}u^{p-2}\varphi(v)|\nabla u|^2 \psi \\
&\qquad\qquad\qquad\qquad+\int_{\R^n}\frac{({ 2\sigma v+\frac{2\sigma}{\tau}v}  -\chi(p-1))^2}{4(p-1)(1-\eps)}u^p\varphi(v)|\nabla v|^2\psi ,
\end{align*}
$$
-\int_{\R^n}u^{p-1}\varphi(v)\nabla u\cdot\nabla \psi\le \eps (p-1)\int_{\R^n} u^{p-2}|\nabla u|^2\varphi(v)\psi +\frac{\k_1^2}{4\eps (p-1)}\int_{\R^n}u^p  \varphi(v)\psi,
$$
and
$$
\int_{\R^n}(\chi - \frac{2\sigma v}{{\tau}p})u^p\varphi(v)\nabla v\cdot\nabla \psi\le \eps' \int_{\R^n} u^p |\nabla v|^2 \varphi(v)\psi+{\frac{\k_1^2}{4\eps'} \int_{\R^n}}\Big(\chi - \frac{2\sigma v}{{\tau}p}\Big)^2u^p  \varphi(v)\psi,
$$
Then we have
\begin{align}
\label{proof-thm1-eq4}
    &\quad\,\,\frac{1}{p}\frac{d}{dt}\int_{\R^n}u^{p}\varphi(v) \psi\nonumber\\
   &\leq \int_{\R^n}\Big[\frac{({ 2\sigma v+\frac{2\sigma}{\tau}v} -\chi(p-1))^2}{4(p-1)(1-\eps)} + \eps' \Big]u^p\varphi(v)|\nabla v|^2\psi \nonumber  \\
&\quad + \chi\int_{\R^n}u^p|\nabla v|^2\varphi^{\prime}(v) \psi - \frac{1}{{\tau}p}\int_{\R^n}u^p\varphi^{\prime\prime}(v)|\nabla v|^2  \psi + \int_{\R^n}(au^{p} -bu^{p+1})\varphi(v)\psi\nonumber\\
&\quad +\Big[ \frac{\k_1^2}{4\eps (p-1)}+ \Big(\chi - \frac{2\sigma v}{{\tau}p}\Big)^2\frac{\k_1^2}{4\eps'}\Big]\int_{\R^n}u^p  \varphi(v)\psi\nonumber\\
&= \int_{\R^n}\Big[\frac{({ 2\sigma v+\frac{2\sigma}{\tau}v}-\chi(p-1))^2}{4{(p-1)}(1-\eps)} + \eps' +2\sigma\chi v - \frac{1}{{\tau}p}(2\sigma +4\sigma^2 v^2)\Big] u^p\varphi(v)|\nabla v|^2  \psi \nonumber\\
&\quad + \int_{\R^n}(au^{p} -bu^{p+1})\varphi(v)\psi +\int_{\R^n}\Big[ \frac{\k_1^2}{4\eps (p-1)}+ \Big(\chi - \frac{2\sigma v}{{\tau}p}\Big)^2\frac{\k_1^2}{4\eps'}\Big]u^p  \varphi(v)\psi.
\end{align}

We claim that, if \eqref{new-d-tau-n-eq} holds, then with $\sigma$ defined in \eqref{sigma} and \eqref{sigma1} below we have for all $v\in [0,\|v_0\|_\infty]$,
\beq\lb{4.17}
\frac{({2\sigma v+\frac{2\sigma}{\tau}v}-\chi(p-1))^2}{4{(p-1)}} +2\sigma\chi v - \frac{1}{{\tau}p}(2\sigma +4\sigma^2 v^2)<0.
\eeq
Indeed, \eqref{4.17} is equivalent to
\beq\lb{4.18}
A\sigma^2 v^2+B< C\sigma+D\sigma v
\eeq
where
\[
A:=\frac{(\tau^*)^2}{p-1}+\frac{4}{\tau p(p-1)}>0,\quad B:=\frac{\chi^2(p-1)}{4}>0,\quad C:=\frac{2}{\tau p}>0,\quad D:=-\chi \tau^*.
\]

First, assume $D< \sqrt{AB} $ and in this case we take
\beq\lb{sigma}
\sigma:={\sqrt{B}(2\sqrt{AB}-D)}/{(C\sqrt{A})}.
\eeq
In order to have \eqref{4.18} valid, since it is a convex quadratic form of $v$, it suffices to have \eqref{4.18} with $v=\|v_0\|_\infty$ and $v=0$. When $v=0$, it suffices to have
\[
B\leq {\sqrt{B}(2\sqrt{AB}-D)}/{\sqrt{A}}
\]
which is certainly true as $D\leq \sqrt{AB}$.
When $v=\|v_0\|_\infty$,  \eqref{4.18} becomes
\beq\lb{4.20}
\frac{\sqrt{AB}(2\sqrt{AB}-D)}{C}\|v_0\|_\infty^2-D\|v_0\|_\infty+\frac{C(-\sqrt{AB}+D)}{2\sqrt{AB}-D}<0.
\eeq
Note that $AB=|\chi|^2\alpha^2$ and $D=-j|\chi||\tau^*|$ with $j=\text{Sign}(\chi\tau^*)$.
Thus \eqref{4.20} reduces to
\beq\lb{4.25}
|\chi|\|v_0\|_\infty<\frac{1}{\tau p}\left(\frac{-j|\tau^*|+\sqrt{|\tau^*|^2+4\alpha^2+4j\alpha|\tau^*|}}{2\alpha^2+\alpha j|\tau^*|}\right)=\frac{2}{\tau p(2\alpha+j|\tau^*|)}.
\eeq

Next, assume that $D\geq \sqrt{AB}$, take $\eps''>0$ and
\beq\lb{sigma1}
\sigma:=B/C+\eps''> {\sqrt{B}(2\sqrt{AB}-D)}/{(C\sqrt{A})}.
\eeq
Therefore \eqref{4.18} holds when $v=0$. When $v=\|v_0\|_\infty$, since $D\geq \sqrt{AB}$, after taking $\eps''$ to be sufficiently small,
\[
A\sigma^2\|v_0\|_\infty^2+B<D\sigma\|v_0\|_\infty+C\sigma \Longleftarrow \|v_0\|_\infty <\frac{C}{\sqrt{AB}}\Longleftrightarrow |\chi|\|v_0\|_\infty<\frac{2}{\tau p \alpha}.
\]
This and \eqref{4.25}  reduce to \eqref{new-d-tau-n-eq}. Overall, we proved the claim \eqref{4.17}.


\medskip

Now by choosing $0<\eps,\eps'\ll 1$ and  \eqref{proof-thm1-eq4}, we then have
\beq\label{proof-thm1-eq5}
\begin{aligned}
    \frac{1}{p}\frac{d}{dt}\int_{\R^n}u^{p}\varphi(v) \psi&\leq \int_{\R^n}(au^{p} -bu^{p+1})\varphi(v)\psi +\int_{\R^n}\Big( \frac{\k_1^2}{4\eps (p-1)}+ \Big(\chi - \frac{2\sigma v}{{\tau }p}\Big)^2\frac{\k_1^2}{4\eps'}\Big)u^p  \varphi(v)\psi\\
&\le M\int_{\R^n}u^{p}\varphi(v)\psi -b\int_{\R^n}u^{p+1}  \varphi(v)\psi
\end{aligned}
\eeq
where $M=\max_{0\le v\le \|v_0\|_\infty} \Big(a +  \frac{\k_1^2}{4\eps (p-1)}+\Big(\chi - \frac{2\sigma v}{{\tau }p}\Big)^2\frac{\k_1^2}{4\eps'}\Big)$. Note hat
\begin{align*}
\int_{\R^n} u^p \varphi(v) \psi=\int_{\R^n} u^p \varphi^{\frac{p}{p+1}}\psi^{\frac{p}{p+1}} \cdot {\varphi}^{\frac{1}{p+1}}\psi^{\frac{1}{p+1}}
\le \Big(\int_{\R^n} u^{p+1}\varphi(v)\psi\Big)^{\frac{p}{p+1}}\Big(\int_{\R^n}\varphi(v)\psi\Big)^{\frac{1}{p+1}}.
\end{align*}
This and $\varphi(v)\leq e^{\sigma \|v_0\|_\infty^2}$  with $\sigma$ defined in \eqref{sigma} and \eqref{sigma1} imply that
$$
-b\int_{\R^n}u^{p+1}\varphi(v)\psi\le -{b}\Big(\int_{\R^n}\varphi(v)\psi\Big)^{-\frac{1}{p}}\Big(\int_{\R^n} u^p\varphi(v)\psi\Big)^{\frac{p+1}{p}}\le -bK_p \Big(\int_{\R^n} u^p\varphi(v)\psi\Big)^{\frac{p+1}{p}},
$$
where $K_p:= (e^{\sigma\|v_0\|_\infty^2}\int_{\R^n}\psi)^{-\frac{1}{p}}$.  This together with \eqref{proof-thm1-eq5} implies that
\begin{align*}
\frac{1}{p}\frac{d}{dt}\int_{\R^n}u^{p}\varphi(v)\psi
&\le M\int_{\R^n} u^p \varphi(v)\psi  - bK_p\Big(\int_{\R^n} u^p\varphi(v)\psi\Big)^{\frac{p+1}{p}}
\end{align*}
By the comparison principle for ODEs, we have
\begin{equation}
\label{proof-thm1-eq6}
\int_{\R^n}u^{p}\psi \leq \int_{\R^n}u^{p}\varphi(v)\psi \le \max\left[ \int_{\R^n}u_0^{p}\varphi(v_0)\psi, \Big(\frac{M}{bK_p}\Big)^p\right].
\end{equation}
Note that $M$ and $K_p$ are independent of $x_0\in\R^n$. Hence \eqref{proof-thm1-eq2} holds, and the theorem is thus proved.
\end{proof}

\section{Global  existence of classical solutions of \eqref{main-eq2}}

In this section, we study  the global existence of classical solutions of \eqref{main-eq2}
and prove Theorem \ref{main-thm2}.

\begin{proof}[Proof of Theorem \ref{main-thm2}]
First, note that by Theorem \ref{general-global-existence-thm}, it suffices to prove that there is
$p>\max\{1,\frac{n}{2}\}$ such that
\begin{equation}
\label{proof-thm2-eq1}
\limsup_{t\to T_{\max}}\sup_{x_0\in\R^n}\int_{B(x_0,1)}(u(t,x+x_0;u_0,v_0))^p dx<\infty.
\end{equation}
Let $\psi\in C^\infty(\R^n)$ be as in \eqref{psi-eq00}. To prove \eqref{proof-thm2-eq1},
it is equivalent to prove
\begin{equation}
\label{proof-thm2-eq2}
\limsup_{t\to T_{\max}}\sup_{x_0\in\R^n}\int_{\R^n} u^p(t,x+x_0)\psi(x)dx<\infty.
\end{equation}

In the following, we prove \eqref{proof-thm2-eq2} under the condition \eqref{cond-eq2}, that is,
$$|\chi| \cdot \mu< b\cdot  \sup _{\gamma >\max\{1,\frac{n}{2}\}}\frac{\gamma }{\gamma-1} \big(C_{{ \gamma +1,n}}\big)^{-\frac{1}{\gamma+1}}.
$$
Fix $p>\max\{1,\frac{n}{2}\}$ such that
$$
|\chi| \cdot \mu < \frac{bp}{p-1}\Big(C_{p+1,n}\big)^{-\frac{1}{p+1}}.
$$

 \noindent By the arguments of \eqref{f-s-eq0}, with the term $\int\frac{p+1}{\tau p}u^p\psi dx$ replaced by $\int\frac{\lambda(p+1)}{\tau p}u^p\psi dx$,
we have
\begin{equation}\label{proof-thm2-eq3}
\int_{\R^n}u^p(t,\cdot)\psi dx\le  e^{\frac{\lambda(p+1) (t_0-t)}{\tau}}\int_{\R^n}u^p(t_0,\cdot)\psi dx+p\int_{t_0}^t
e^{\frac{\lambda(p+1) (s-t)}{\tau}} f(s)ds,
\end{equation}
where, for some $r>0$ to be determined,
\begin{align}
\label{f-s-eq}
f(s)&:=A_p r^{-p}|\chi|^{p+1}\int_{\R^n}|\Delta v(s,\cdot)|^{p+1} \psi dx+\Big(\frac{|\chi|\k_1}{p+1}+\frac{|\chi|\k_1 (p-1)}{p (p+1)}\Big)\int_{\R^n}|\nabla v(s,\cdot)|^{p+1}\psi dx\nonumber  \\
&\quad -\big(b-\frac{|\chi|\k_1 p}{p+1}-\frac{|\chi|\k_1(p-1)}{p+1}-r-\eps \big)\int_{\R^n} u^{p+1}(s,\cdot)\psi dx+C_{\eps,\lambda}
\end{align}
and
$$C_{\eps,\lambda} := \frac{1}{p+1}\Big[\frac{p+1}{p}\eps\Big]^{-p}\Big[a+\frac{\k_1}{2}+\frac{\lambda(p+1)}{\tau p}\Big]^{p+1}\int_{\R^n}\psi\quad\text{and}\quad A_p:=\frac{1}{p}\big(\frac{p-1}{p+1}\big)^{p+1}. $$

Let $\psi_1 = \psi^{\frac{1}{p+1}}$. Then  $v\psi_1$ solve
$$
\tau(v\psi_1)_t =\Delta(v\psi_1)  - \lambda v\psi_1 -2\nabla v\cdot\nabla\psi_1 + v\Delta \psi_1 +\mu u \psi_1.
$$
By Lemma \ref{maximal-regularity-lm} with $\gamma:=p+1$ and $\alpha:=\lambda$, we have
\begin{align}
\label{f-s-eq3}
&\int_{t_0}^t \int_{\R^n} e^{\frac{\lambda (p+1)s}{\tau}}\left[(v\psi_1)^{p+1}+|\nabla (v\psi_1)|^{p+1}+|\Delta (v \psi_1)|^{p+1} \right]dx ds\nonumber\\
&\qquad\le C_{p+1,  n}\int_{t_0}^t \int_{\R^n}e^{\frac{\lambda (p+1)s}{\tau}}  \left[ -2\nabla v\nabla\psi_1 + v\Delta \psi_1 +\mu u\psi_1\right]^{p+1}dx ds\nonumber\\
&\qquad\quad+ C_{p+1, n} (t+\tau^{p+1}t_0
^{-p})\|v_0\psi_1\|_{L^{p+1}(\R^n)}^{p+1}.
\end{align}
Hence, by \eqref{4.8} (which holds the same here), for any $\eps>0$, there is $0<\k_1\ll1$ such that
\begin{align*}
\int_{t_0}^t& \int_{\R^n}e^{\frac{\lambda (p+1)s}{\tau}}\left[ (v\psi_1)^{p+1}
+(|\nabla v|\psi_1)^{p+1}+(|\Delta v|\psi_1)^{p+1}\right]dx ds\nonumber\\
&\le (1+\eps)^2\int_{t_0}^t \int_{\R^n} e^{\frac{\lambda (p+1)s}{\tau}}\left[(v\psi_1)^{p+1}+|\nabla (v\psi_1)|^{p+1}+|\Delta (v \psi_1)|^{p+1} \right]dx ds\nonumber\\
&\le (1+\eps)^2 C_{p+1,n}\int_{t_0}^t  \int_{\R^n} e^{\frac{\lambda (p+1)s}{\tau}}\Big( -2\nabla v\cdot\nabla\psi_1+ v\Delta \psi_1 +\mu u \psi_1\Big)^{p+1}dx ds\nonumber\\
&\qquad + (1+\eps)^2 C_{p+1,n} (t+\tau^{p+1}t_0
^{-p})\|v_0(\cdot)\psi_1\|_{L^{p+1}(\R^n)}^{p+1}\\
&\le (1+\eps)^2 C_{p+1,n}\int_{t_0}^t \int_{\R^n} e^{\frac{\lambda (p+1)s}{\tau}} \left[(1+\eps)\mu^{p+1}  (u\psi_1)^{p+1} + \eps (|\nabla v|\psi_1)^{p+1}+ \eps(v \psi_1)^{p+1}\right] dxds\\
&\qquad + (1+\eps)^2C_{p+1,n} (t+\tau^{p+1}t_0
^{-p})\|v_0(\cdot)\psi_1\|_{L^{p+1}(\R^n)}^{p+1}.
\end{align*}
This implies that
\begin{align}\label{delta-v1}
&\int_{t_0}^t \int_{\R^n}e^{\frac{\lambda (p+1)s}{\tau}}(|\Delta v|\psi_1)^{p+1}dx ds\nonumber\\
&\le -\int_{t_0}^t \int_{\R^n}e^{\frac{\lambda (p+1)s}{\tau}}\left[(v\psi_1)^{p+1} +(|\nabla v|\psi_1)^{p+1}\right]dx ds\nonumber\\
&\quad+  (1+\eps)^2 C_{p+1,n}\int_{t_0}^t \int_{\R^n} e^{\frac{\lambda (p+1)s}{\tau}} \left[(1+\eps)\mu^{p+1}  (u\psi_1)^{p+1} + \eps (|\nabla v|\psi_1)^{p+1}+ \eps(v \psi_1)^{p+1}\right] dxds\nonumber\\
&\quad + (1+\eps)^2C_{p+1,n} (t+\tau^{p+1}t_0
^{-p})\|v_0(\cdot)\psi_1\|_{L^{p+1}}^{p+1}.
\end{align}
Choose $\k_1\ll 1$ such that
$$
\frac{|\chi|\k_1}{p+1}+\frac{|\chi|\k_1 (p-1)}{p (p+1)}\le \eps,\quad
\frac{|\chi|\k_1 p}{p+1}+\frac{|\chi|\k_1(p-1)}{p+1}<\eps.
$$
Then by \eqref{proof-thm2-eq3} and \eqref{delta-v1}, we have
\beq
\begin{aligned}
\label{tm2-eqn4}
&\quad\,\int_{\R^n}u^p(t,x)\psi(x)dx \\
&\le-p\left(b-r-2\eps -A_pr^{-p}|\chi|^{p+1}(1+\eps)^3\mu^{p+1}C_{p+1,n}\right)\int_{t_0}^t \int_{\R^n}e^{\frac{\lambda(p+1)(s-t)}{\tau}}u^{p+1}\psi dxds \\
&\quad  -p\left(A_p r^{-p}|\chi|^{p+1}(1-(1+\eps)^2C_{p+1,n}\eps) -p\eps\right)\int_{t_0}^t \int_{\R^n}e^{\frac{\lambda(p+1)(s-t)}{\tau}}|\nabla v|^{p+1}\psi dx\ ds \\
&\quad  -p A_p r^{-p}|\chi|^{p+1}\left(1-C_{p+1,n}(1+\eps)^2{\eps}\right)\int_{t_0}^t\int_{\R^n} e^{\frac{\lambda(p+1)(s-t)}{\tau}}v^{p+1}\psi dx ds +C^*
\end{aligned}
\eeq
where $C^*$ is such that
\begin{align*}
(1+\eps)^2 e^{-\frac{\lambda (p+1)t}{\tau}}C_{p+1,n} (t+\tau^{p+1}t_0
^{-p})\|v_0(\cdot)\psi_1\|_{L^{p+1}}^{p+1} +p\int_{t_0}^t e^{\frac{\lambda (p+1)(s-t)}{\tau}} C_{\eps,\lambda} ds\le C^*.
\end{align*}
It is clear that $C^*$ exists and can be selected independent of $t$.
Next, note that
$$
\min_{r>0} \left(A_p C_{p+1,n} r^{-p}|\chi|^{p+1}\mu^{p+1}+r\right)=\frac{p-1}{p}\left(C_{p+1,n}\right)^{\frac{1}{p+1}}|\chi|\mu.
$$
Because $b>\left(\inf_{{p} >\max\{1,\frac{n}{2}\}}\frac{p-1}{p}C_{p+1,n}^{\frac{1}{p+1}}\right)|\chi|\mu
$, there are $r>0$ and
$p>\max\{1,\frac{n}{2}\}$ such that
$$
b>r+A_pr^{-p}C_{p+1,n}|\chi|^{p+1}\mu^{p+1}.
$$
Fix such $r>0$ and $p>\frac{n}{2}$. Choose $\eps\ll 1$ such that
$$
b-r-2\eps -A_pr^{-p}|\chi|^{p+1}(1+\eps)^3\mu^{p+1} C_{p+1,n}>0
$$
and
$$
A_p r^{-p}|\chi|^{p+1}(1-(1+\eps)^2C_{p+1,n}\eps) -p\eps>0 \qquad and \qquad 1-C_{p+1,n}(1+\eps)^2{\eps}>0.
$$
Then by \eqref{tm2-eqn4},
$$
\sup_{t\in [t_0,T_{\max}),x_0\in\infty}\int_{\R^n}u^p(t,x+x_0)\psi(x)<\infty.
$$
Theorem \ref{main-thm2} is thus proved.
\end{proof}

\section{Global existence of classical solutions of \eqref{main-eq3}}
In this section, we study  the global existence of classical  solutions  of \eqref{main-eq3} and
prove theorem \ref{main-thm3}.
First,   if no confusion occurs, we put
$$
u=u(t,x+x_0;u_0) \quad {\rm and}\,\, v=v(t,x+x_0;u_0).
$$

\begin{proof}[Proof of Theorem \ref{main-thm3}]
By Theorem \ref{general-global-existence-thm}, it suffices to prove that there is $p>\max\{1,\frac{n}{2}\}$ such that
\begin{equation}
\label{proof-main-thm3-eq1}
\limsup_{t\to T_{\max}}\sup_{x_0\in\R^n}\int_{\R^n} u^p(t,x+x_0)\psi(x)dx<\infty
\end{equation}
for some
 $\psi\in C^{\infty}(\R^n)$  satisfying \eqref{psi-eq00} with some $\k_1,\k_2>0$.
If $\chi\le 0$,  \eqref{proof-main-thm3-eq1}  follows from \cite[Theorem A]{SaSh2}.
We then only need to prove   \eqref{proof-main-thm3-eq1} for the case that $\chi>0$.

\smallskip

In the following, we assume that $\chi>0$. 
By \eqref{cond-eq3}, 
there is $p>\max\{1,\frac{n}{2}\}$ such that
\beq\lb{6.3}
b>\mu\chi \big(1-\frac{1}{p}\big)=\frac{\mu \chi(p-1)}{p}.
\eeq
Fix such $p>\max\{1,\frac{n}{2}\}$.
From the equation, we have
\beq\lb{6.2}
\begin{aligned}
\frac{1}{p}\frac{d}{dt}\int_{\R^n} u^p \psi dx&=\int_{\R^n}u^{p-1}\psi\Delta u-\chi\int_{\R^n}u^{p-1}\psi\nabla \cdot(u\nabla v)dx+\int_{\R^n}u^p \psi(a-b u)dx\\
&=-(p -1)\int_{\R^n} u^{p-2}|\nabla u|^2\psi-\int_{\R^n} u^{p -1}\nabla u\cdot\nabla \psi\\
&\quad +\chi (p-1)\int_{\R^n} u^{p -1} \nabla u\cdot (\nabla v)\psi
+\chi\int_{\R^n}u^p \nabla v\cdot\nabla \psi+\int_{\R^n}u^p (a-bu)\psi.
\end{aligned}
\eeq
Note that $\Delta v=\lambda v-\mu u$. Hence
\begin{align*}
&\quad\,\,    \chi(p-1)\int_{\R^n}u^{p-1}\nabla u\cdot (\nabla v)\psi+\chi\int_{\R^n} u^p \nabla v\cdot\nabla \psi\\
&= -\frac{\chi(p-1)}p\int_{\R^n}u^{p}  \nabla v\cdot\nabla \psi-\frac{\chi(p-1)}p\int_{\R^n}u^{p}  (\Delta v) \psi+\chi\int_{\R^n} u^p \nabla v\cdot\nabla \psi\\
&=\frac{\chi}{p}\int_{\R^n}
u^p \nabla v\cdot\nabla \psi-\frac{\lambda\chi(p-1)}{p}\int_{\R^n} u^p v\psi+\frac{\mu\chi(p-1)}{p}\int_{\R^n} u^{p+1}\psi\\
&\le \frac{\chi\kappa_1}{p}\int_{\R^n}
u^p |\nabla v|\psi  +\frac{\mu\chi(p - 1)}{p}\int_{\R^n} u^{p+1}\psi.
\end{align*}
It then follows from \eqref{6.2} and Young's inequality that
\begin{align*}
\frac{1}{p}\frac{d}{dt}\int_{\R^n} u^p \psi dx
&\leq -(p -1)\int_{\R^n} u^{p-2}|\nabla u|^2\psi-\int_{\R^n} u^{p -1}\nabla u\cdot\nabla \psi +\frac{\chi\kappa_1}{p}\int_{\R^n}
u^p |\nabla v|\psi\\
&\quad+\frac{\mu\chi(p - 1)}{p}\int_{\R^n} u^{p+1}\psi+\int_{\R^n}u^p (a-bu)\psi\\
&\le -(p-1)\int_{\R^n} u^{p-2} |\nabla u|^2\psi +{\frac{\k_1}{2}\int_{\R^n} u^{p-2}|\nabla u|^2\psi+\frac{\k_1}{2}\int_{\R^n} u^p \psi} \\
&\quad  + \frac{\chi\kappa_1}{p}\int_{\R^n}u^p |\nabla v|\psi+\frac{\mu\chi(p-1)}{p}\int_{\R^n}u^{p+1}\psi  +\int_{\R^n}u^p(a-bu)\psi\\
&\le  -(p-1-\frac{\k_1}{2})\int_{\R^n} u^{p-2} |\nabla u|^2\psi +\frac{\k_1}{2}\int_{\R^n} u^p \psi +\frac{\mu\chi(p-1)}{p}\int_{\R^n}u^{p+1}\psi\\
&\quad     +\frac{\chi\k_1}{p+1}\int_{\R^n}u^{p+1}\psi
+ \frac{\chi\k_1}{p (p+1)}\int_{\R^n} |\nabla v|^{p+1}\psi+\int_{\R^n}u^p(a-bu)\psi.
\end{align*}
By \eqref{main-eq3-estimate-eq1} and \eqref{main-eq3-estimate-eq2}, for any $p>\max\{1,\frac{n}{2}\}$ and $0<\k_1\ll 1$, there is $K_p>0$ such that
\begin{equation*}
\int_{\R^n}|\nabla v|^{p+1}\psi \le K_p\Big(\int_{\R^n}u^p\psi\Big)^{\frac{p+1}{p}}\le K_p\Big( \int_{\R^n}u^{p+1}\psi \Big)\Big(\int_{\R^n}\psi^{p+1}\Big)^{\frac{1}{p}}.
\end{equation*}
By further taking $\k_1$ to be sufficiently small, we have
\begin{align}
\label{proof-main-thm3-eq2}
\frac{1}{p}\frac{d}{dt}\int_{\R^n} u^p \psi dx
&\le  2a \int_{\R^n} u^p \psi - \frac{1}{2}\left(b-\frac{\mu\chi(p-1)}{p}\right)  \int_{\R^n} u^{p+1}\psi.
\end{align}
This  together with \eqref{6.3} and the comparison principle for scalar ODEs implies that \eqref{proof-main-thm3-eq1} holds. The theorem is thus proved.
\end{proof}

\appendix

\section{Appendix:  Maximal regularity}

There are a large number of literatures discussing maximal regularity for parabolic equations, see e.g., \cite{HiPr,Lieberman,lamberton,DuRo}. However, to the best of our knowledge, none provides explicit dependence on the dimension $n$ and the exponent $\gamma$ of the $L^\gamma$ norms. In this appendix, we  estimate the dimension $n$ and $\gamma$ dependencies of the constant $C_{\gamma,n}$  in Lemma \ref{maximal-regularity-lm}. Below  we take {$\tau=1$ and} $v(0,\cdot)=0$  for simplicity. 
\begin{tm}\lb{T.A.1}
Let $\gamma\in (2,\infty)$. There exists an absolute constant $C>0$ independent of $\gamma$ and $n$ such that the following holds for
\[
C_{\gamma,n}:=C^\gamma (\gamma-2)^{1-\gamma}2^{9n(\gamma-2)}.
\]
For any  $T>0$, if  $g \in L^\gamma((0,T), L^\gamma(\R^n))$ and $v(\cdot,\cdot)\in W^{1,\gamma}((0,T),L^{\gamma}(\R^n))\cap L^{\gamma}((0,T), W^{2,\gamma}(\R^n))$ solve  the following initial boundary value problem,
\begin{equation*}
\begin{cases}
v_t =\Delta v  + g,&\quad x\in \R^n,\,\,  0<t<T\cr
v(0,x) = 0, &\quad x\in \R^n,
\end{cases}
\end{equation*}
then
\begin{equation*}
    \begin{aligned}
\int_{0}^T \int_{\R^n}\Big( |v (t,x)|^\gamma +|\nabla v(t,x)|^\gamma +|\Delta v(t,x)|^{\gamma}\Big)dxdt \le C_{\gamma,n} \int_{0}^T \int_{\R^n}|g(t,x)|^{\gamma}dx dt.
\end{aligned}
\end{equation*}
\end{tm}

The proof is based on Calder\'{o}n-Zygmund decomposition and largely follows the one in \cite{DuRo}. We are only able to bound $C_{\gamma,n}$ exponentially in $n$ and $\gamma$, which might not be optimal.

\begin{proof}
Let us only prove for all $T>0$,
\beq\lb{M1}
\int_{0}^T \int_{\R^n}|\Delta v(t,x)|^{\gamma}dxdt \le C_{\gamma,n} \int_{0}^T \int_{\R^n}|g(t,x)|^{\gamma}dx dt,
\eeq
as the rest will be standard. We will first prove the result for $\gamma=2$, and then apply Calder\'{o}n-Zygmund decomposition to obtain a bound for $\gamma=1$. Finally, we conclude the proof via Marcinkiewicz interpolation theorem and a duality argument. Let us emphasize that all constants below, if not specified or not with subscript $n$ or $\gamma$, do not depend on $n$ and $\gamma$.

\medskip

We divide the proof into four steps.

\medskip

\noindent {\bf Step 1. } Firstly, let us prove for $\gamma=2$. We can assume that $v$ is smooth, and $v$ and $|\nabla v|$ decay sufficiently fast at $x=\infty$. The general result follows by approximations.
Call $u:=(u_1,\ldots,u_n)=\nabla v$ which then solves
\beq\lb{M0}
u_t=\Delta u+\nabla g\quad\text{ with }u(0,\cdot)=0.
\eeq
Let us apply $(\cdot\, u)$ to both sides of \eqref{M0}, integrate over $[0,T]\times \R^n$ for some $T>0$. By integration by parts,
\[
{\frac{1}{2}}\int_{\R^n}u(T,x)^2dx=-\sum_{i=1}^n\int_0^T\int_{\R^n}|\nabla u_i|^2 dxdt-\int_{\R^n}\int_0^T g\,\nabla\cdot u\, dxdt.
\]
Thus, by Young's inequality,
\[
\sum_{i=1}^n\int_{\R^n}\int_0^T|\nabla u_i|^2 dxdt\leq \frac12\int_0^T\int_{\R^n}|g|^2dxdt+\frac12 \int_0^T \int_{\R^n}|\nabla\cdot u|^2dxdt
\]
which implies that
\[
\int_{\R^n}\int_0^T|\Delta v|^2 dxdt=\sum_{i=1}^n\int_{\R^n}\int_0^T|\nabla u_i|^2 dxdt\leq \int_0^T\int_{\R^n}|g|^2dxdt.
\]
Thus the conclusion \eqref{M1} holds with $C_{2,n}\equiv 1$.
\medskip

\noindent {\bf Step 2.} We introduce some notations and preliminary results. Let $\calT$ be the strongly continuous semigroup generated by Laplacian on $L^2(\R^n)$ and so it is given by
\[
(\calT(t)f)(x):=\int_{\R^n}G(t,x-y)f(y)dy
\]
where $G$ is given in \eqref{heat-kernel}. We can write
$
G(t,x-y)= |4\pi t|^{-n/2}p(\frac{|x-y|^2}{t})
$
with $p(s):=\exp(-s/4)$. Moreover, one can extend $\calT$ to an analytic semigroup of angle $\frac{\pi}{4}$ in the complex plane. Actually, any angle $<\frac{\pi}{2}$ works. We claim that there exists $\eps\in (0,1)$ independent of the dimension such that for each $\theta\in (0,\frac{\pi\eps}4)$ we have
\beq\lb{M5}
|G(z,x-y)|\leq (4\pi\,\Re z)^{-n/2} p\big(\frac{|x-y|^2}{2|z|}\big)
\eeq
for all $x,y\in\R^n$ and all $z\in\{|\arg z|<\theta\}$.
The proof is straightforward by the explicit formula of $G(z,x)$. 


\medskip

Consider the following metric $ d^+$ defined on $\R^+\times\R^n$ by
\[
{d}^+((t,x),(s,y)):=(|x-y|^2+|t-s|)^{1/2}.
\]
For $r>0$, the corresponding ball under this metric is
\[
B^+((t,x),r):=\{(s,y)\in \R^+\times\R^n\,:\, {d}^+((t,x),(s,y))<r\}.
\]

Now we define for $x,y\in\R^n$, $t,s,\tau>0$ and $\rho\in\R$,
\beq\lb{M2}
e_\tau(\rho):=\frac{\exp(-|\rho|/\tau)}{2\tau},\quad D(\tau,t,x,s,y):=G(\tau,x-y)e_\tau(t-s).
\eeq
It follows from the proof of \cite[Lemma 4.2]{HiPr} (with $a=1$) (after keeping track of the dependence on $n$) that
\beq\lb{M3}
\sup_{(s,y)\in B^+((\rho,z),\sqrt{\tau})}D(\tau,t,x,s,y)\leq e^2 2^n \inf_{(s,y)\in B^+((\rho,z),\sqrt{\tau})}D(4\tau,t,x,s,y)
\eeq
for all $x,z\in\R^n,t,\rho\geq 0$ and all $\tau>0$.

\medskip

We write $|\Omega|$ or $\mu(\Omega)$ as the standard Lebesgue  measure of $\Omega\subseteq \R^{n+1}$ (or $\Omega\subseteq\R^n$).
We introduce the space-time maximal function $M^+$:
\[
(M^+f)(t,x):=\sup_{a>0,r>0}|Q_{a,r}|^{-1}\int_{Q_{a,r}}|f(t+s,x+y)|dsdy
\]
for $f\in L^p(\R^+\times\R^n)$ and $p\in[1,\infty]$, where $Q_{a,r}:=[-a,a]\times B(r)$. Let $f\in L^p(\R^+\times \R^n)$ for any $p\in [1,\infty]$. Since for each $t>0$, $G(t,x)$ is a radially symmetric and decreasing, and $\int_{\R^n}G(t,x)dx=1$, we have (see \cite{DuRo})
\[
\int_{\R^n}G(t,x-y)|f(t,y)|dy\leq \sup_{r>0}\,|B(r)|^{-1}\int_{B(r)}|f(t,x+y)|dy=: (Mf)(t,x).
\]
Similarly, since for any $\tau>0$, $e_\tau(\rho)$ is a radially symmetric and decreasing, and $\int_{\R}e_\tau(\rho)d\rho=1$, we get
\beq\lb{M4}
\begin{aligned}
\int_0^\infty\int_{\R^n}D(\tau,t,x,s,y)|f(s,y)|dyds\leq
\sup_{a>0}\,(2a)^{-1}\int_{t-a}^{t+a}Mf(s,x) =(M^+f)(t,x) .
\end{aligned}
\eeq

Moreover, it follows from the proof of \cite{StSt} that there exists an absolute constant such that
\beq\lb{M12}
\|M^+f\|_2\leq Cn\|f\|_2\quad\text{ for any $f\in L^2(\R^+\times\R^n)$.}
\eeq

\medskip


Finally, define
\[
\calR f:=\Delta \int_0^t\calT(t-s)f(s)ds.
\]
The goal is to derive an  estimate, which is explicit in $n$ and $\gamma$, for the map $\calR$ from $L^\gamma(\R^+\times\R^n)$ to $L^\gamma(\R^+\times\R^n)$.

\medskip

\noindent {\bf Step 3.} Let us estimate the upper bound of $\calR$ from $L^1(\R^+\times\R^n)$ to $L^1_w(\R^+\times\R^n)$, where $L^1_w(\R^+\times\R^n)$ denotes the weak $L^1$ space. We apply Calder\'{o}n-Zygmund decomposition. For any $f\in L^1(\R^+\times\R^n)$ and fixed $\alpha>0$, there is a disjoint union of open space-time cubes $Q_k$ such that
\begin{enumerate}
    \item for each $Q_k$,
    \[
    \alpha \leq |Q_k|^{-1}\int_{Q_k}|f(t,x)|dxdt\leq 2^{n+1}\alpha.
    \]
    \item $|f(t,x)|\leq \alpha$ almost everywhere in the complement of $\cup_k Q_k$.
\end{enumerate}
Let us cover the cubes by countably many balls $B_k^+:=B^+((t_k,x_k),{r_k})$ for some $(t_k,x_k)\in\R^{n+1}$ and $r_k>0$. Note that for any $(t,x)\in [-\frac12,\frac12]^{n+1}\subset\R^{n+1}$, we have
\[
|x|^2+|t|\leq \frac{1}{4}n+\frac{1}{2}\leq \frac{1}{2}n.
\]
Thus one can cover $[-\frac12,\frac12]^{n+1}$ by $B^+((0,0),{\sqrt{n/2}})$. By Stirling's formula $\sqrt{2\pi n}(\frac{n}e)^n\leq n!\leq e^{\frac1{12}}\sqrt{2\pi n}(\frac{n}e)^n$, the volume of a unit ball in $\R^n$ can be estimated by
\[
|B(1)|=\frac{2(2\pi)^{(n-1)/2}}{n!!}< \frac{2}{\sqrt{n\pi}}(\frac{2\pi e}{n})^{n/2}
\]
Thus we get
\[
|B^+((0,0),\sqrt{n/2})|\leq |B(1)|(\sqrt{n/2})^n 2\sqrt{n/2}\leq \frac{4}{\sqrt{2\pi}} (\pi e)^\frac{n}{2} \leq 2(3^n)\left|[-1/2,1/2]^{n+1}\right|.
\]
In view of this, we can assume
\[
|\cup_k B_k^+|\leq 2(3^n)|\cup_k Q_k|,
\]
and moreover, each $x\in\R^n$ is only contained in at most $2^n$ many balls $B_k^+$.
Next, we take $g(t,x):=f(t,x)$ for $(t,x)\in\R^{n+1}\backslash \cup_k Q_k$ and $g(t,x):=0$ otherwise. Then define
\[
b_0(t,x):=(f(t,x)-g(t,x)-b_0(t,x))\chi_{B_0^+}(t,x),
\]
and for $i=1,\ldots$, define iteratively:
\[
b_i(t,x)=\big(f(t,x)-g(t,x)-\sum_{k=0}^{i-1}b_k(t,x)\big)\chi_{B_i^+}(t,x).
\]
In this way, we obtain a decomposition of $f$ such that the following holds
\begin{itemize}
    \item[(f1)]  $f=g+\sum_i b_i$ with $b_i$ supported in $B_i^+$;

    \item[(f2)] $|g(t,x)|\leq \min\{|f(t,x)|,\alpha\}$ almost everywhere in $\R^{n+1}$;

    \item[(f3)] each $x\in\R^n$ is only contained in at most $2^n$ many balls $B_k^+$;

    \item[(f4)] $\|b_i\|_1\leq 2^{n+1}\alpha |B_i^+|$;

    \item[(f5)] $\sum_i |B_i^+|\leq 2(3^n) \sum_k|Q_k|\leq \frac{2(3^n)}\alpha \|f\|_1$.
\end{itemize}
The advantage of using $B_k^+$ instead of cubes is the Harnack type estimate  \eqref{M3}.
The constants are explicit and these properties are stronger than those in \cite{HiPr}. The rest of Step 3 largely follows the proof of \cite{HiPr}, except that we will keep track of constants' dependencies on $n$ and $\gamma$.

We now further decompose $\sum_i b_i=h+l$ with $b_i=h_i+l_i$ where
\[
h_i(t,x):=\int_0^\infty \int_{\R^n} G(\tau_i,x-y)e_{\tau_i}(t-s)b_i(s,y)dyds
\]
where $\tau_i:=\rho_i^2$ (and $\rho_i$ is the radius of $B^+_i$) and $e_\tau$ is given in \eqref{M2}. To this end, we define
\[
u:=\calR g,\quad w_1:=\calR h, \quad w_2:=\calR l.
\]
For any fixed $\alpha>0$, one has
\begin{align*}
\mu\{(t,x)\in\R^+\times\R^n\,:\,|\calR f(t,x)|>\alpha\} &\leq \mu\{(t,x)\in\R^+\times\R^n\,:\,|u(t,x)|>\alpha\}    \\
&+\mu\{(t,x)\in\R^+\times\R^n\,:\,|w_1(t,x)|>\alpha\}    \\
&+\mu\{(t,x)\in\R^+\times\R^n\,:\,|w_2(t,x)|>\alpha\} \\
&=:I_1+I_2+I_3.
\end{align*}
Below, we prove that each term $I_i$ for $i=1,2,3$ is bounded by $\frac{C(n)}{\alpha}\|f\|_1$.

\medskip
{\bf Step 3.1.} Estimate $I_1+I_2$. It is clear that
\[
I_1=\mu\{(t,x)\in\R^+\times\R^n\,:\,|u(t,x)|>\alpha\} \leq\frac{1}{\alpha^2}\|v\|_2^2.
\]
By Step 1,  $\|v\|_2^2\leq \|g\|_2^2$. Then use (f2) to get
\beq\lb{M13}
I_1 \leq \frac{1}{\alpha^2}\|g\|_2^2\leq \frac{1}{\alpha}\|g\|_1\leq \frac{1}{\alpha}\|f\|_1.
\eeq

As for $I_2$, it follows from \eqref{M3} and (f4) that for some $C>0$ independent of $n$ and $i$,
\begin{align*}
|h_i(t,x)|&\leq \int_{\R^n}\int_0^\infty D(\tau_i,t,x,s,y)|b_i(s,y)|dsdy\leq \|b_i\|_1\sup_{(s,y)\in B^+_i}D(\tau_i,t,x,s,y)\\
&\leq 2^{n+1}\alpha |B_i^+|\sup_{(s,y)\in B^+_i}D(\tau_i,t,x,s,y)\leq e^22^{2n+1}\alpha |B_i^+|\inf_{(s,y)\in B^+_i}D(4\tau_i,t,x,s,y)\\
&\leq e^22^{2n+1}\alpha \int_{\R^n}\int_0^\infty D(4\tau_i,t,x,s,y)\chi_i(s,y)dsdy
\end{align*}
where $\chi_i$ is the characteristic function of $B_i^+$.
This yields
\begin{align*}
\|h\|_2&=\sup_{\|\varphi\|_2\leq 1}\{|(h,\varphi)|\}    \leq \sup_{\|\varphi\|_2\leq 1}\left\{\sum_i|(h_i,\varphi)|\right\} \\
&\leq e^22^{2n+1}\alpha \sup_{\|\varphi\|_2\leq 1}\left\{\sum_i \int_{\R^n}\int_0^\infty\int_{\R^n}\int_0^\infty D(4\tau_i,t,x,s,y)\chi_i(s,y) \varphi(t,x)\,dsdydtdx\right\}.
\end{align*}
Therefore, using \eqref{M4} and that $D(\tau,t,x,s,y)=D(\tau,s,y,t,x)$ yields
\begin{align*}
\|h\|_2
&\leq e^22^{2n+1}\alpha \sup_{\|\varphi\|_2\leq 1}\left\{\sum_i \int_{\R^n}\int_0^\infty (M^+\varphi)(s,y)\chi_i(s,y) \,dsdy\right\}\\
&=e^22^{2n+1}\alpha \sup_{\|\varphi\|_2\leq 1}\{(M^+\varphi,\sum_i\chi_i)\}.
\end{align*}
Since $\varphi\to M^+\varphi$ is a bounded operator on $L^2(\R^+\times\R^n)$ by \eqref{M12}, it follows that
\[
\|h\|_2\leq C4^{n}n\alpha \sup_{\|\varphi\|_2\leq 1}\{(\varphi,\sum_i\chi_i)\}\leq C4^{n}n\alpha \|\sum_i\chi_i\|_2\leq C8^{n}n\alpha (\sum_i |B_i^+|)^{1/2},
\]
where in the last inequality we applied (f3). Then due to (f5),
\[
\|h\|^2_2\leq \left(C8^nn\alpha (3^n\|f\|_1/\alpha)^{1/2}\right)^2\leq  C2^{8n}n^2 \alpha \|f\|_1.
\]
Lastly, using Step 1 again,
\beq\lb{M14}
I_2=\mu\{(t,x)\in\R^+\times\R^n\,:\,|w_1(t,x)|>\alpha\} \leq \frac{1}{\alpha^2}\|w_1\|_2^2\leq \frac{1}{\alpha^2}\|h\|_2^2\leq \frac{C2^{8n}n^2 }{\alpha}\|f\|_1.
\eeq

\medskip

{\bf Step 3.2.} Estimate $I_3$. 
Recall that $w_2=\sum_i \calR(b_i-h_i)$ and set $w_{2,i}:=\calR(b_i-h_i)$. Then
\[
w_{2,i}(t,\cdot)=\int_0^t\Delta \calT(t-s)\left[b_i-\int_0^\infty e_{\tau_i}(s-r)\calT(\tau_i)b_i(r)dr\right]ds.
\]
Below we show that $\|w_{2,i}\|_{1}\leq C2^n \|b_i\|_1$ with $C$ independent of $n$ and $i$.

Note that the spectrum of $\Delta$ lies in the negative real axis. Take $\eps\in (0,1)$ from Step 2 so that  \eqref{M5} holds, and take $\theta\in (\pi/2,\pi/2+\frac{\pi\eps}4)$. Let $\Gamma_\pm=\{re^{\pm i\theta}\,:\, r\geq 0\}$ and $\Gamma_\pm'=\{re^{\pm i\pi\eps/4}\,:\, r\geq 0\}$.
Following the proof of \cite{HiPr}, we have
\beq\lb{M7}
w_{2,i}(t,\cdot)=w_{2,i}^+(t,\cdot)+w_{2,i}^-(t,\cdot):=\int_0^\infty K_+(t-s)b_i(s)ds+\int_0^\infty K_-(t-s)b_i(s)ds
\eeq
where the operators $K_\pm$ are defined as
\[
K_\pm(t-s)=\frac{1}{2\pi i}\int_{\Gamma_\pm}\int_{\Gamma_{\pm}'}\mu e^{-\mu z}e^{\mu(t-s)}\left[1-\int_{-\infty}^te^{\mu(s-r)}\frac{e^{-|r-s|/\tau_i}}{2\tau_i}e^{\mu\tau_i}dr\right]\calT(z)dzd\mu.
\]
Moreover, if we set
\[
\calK_\pm(t-s,x,y)=\frac{1}{2\pi i}\int_{\Gamma_+}\int_{\Gamma_{+}'}\mu e^{-\mu z}e^{\mu\tau_i}u_{\tau_i}(t-s,\mu)G(z,x-y)dzd\mu
\]
where $u_{\tau_i}$ is defined for $\mu\in\Gamma_+$ by
\[
u_{\tau_i}(t-s,\mu)=\left\{
\begin{aligned}
    &e^{\mu(t-s)}\left[e^{-\mu\tau_i}-\frac{1}{2(1-\mu\tau_i)}-\frac{1}{2(1+\mu\tau_i)}\right]+\frac{e^{-(t-s)/\tau_i}}{2(1+\mu\tau_i)},\quad && t\geq s,\\
    &\frac{e^{-(t-s)/\tau_i}}{2(1-\mu\tau_i)}, && t<s,
\end{aligned}
\right.
\]
then by the definition of $\calT$, we can re-write
\[
w_{2,i}^+(t,x)=\int_{\R^n}\int_0^\infty \calK_+(t-s,x,y)b_i(s,y)dsdy.
\]
This yields
\beq\lb{M6}
\begin{aligned}
&\quad \| w_{2,i}^+\|_1\leq \|b_i\|_1\sup_{y\in\R^n}\int_{-\infty}^\infty\int_{\R^n} |\calK_+(\tau,x,y)|dxd\tau    \\
& \leq \|b_i\|_1\sup_{y\in\R^n}\int_{-\infty}^\infty\int_{\R^n} \int_{\Gamma_+}\int_{\Gamma_{+}'}|\mu e^{-\mu z}||e^{\mu\tau_i}||u_{\tau_i}(\tau,\mu)||G(z,x-y)|dzd\mu dxd\tau
\end{aligned}
\eeq

Note that $z\in\Gamma_+'$ and $\Re z\geq \frac12|z|$ by $\frac{\pi\eps}4\leq1$. So we can use \eqref{M5} to have
\[
\int_{\R^n} |G(z,x-y)| dx \leq  \int_{\R^n}(4\pi\, \Re z)^{-n/2} e^{-\frac{|x-y|^2}{8|z|}} dx\leq 2^{n}.
\]
With this, we can argue similarly as done in Section 5 Step III of \cite{HiPr} and obtain from \eqref{M6} that
\[
\| w_{2,i}^+\|_1\leq C 2^{n}\|b_i\|_1.
\]
Notice that in the rest of terms, the dimension $n$ is not involved and so there is no extra dependence on $n$ in the constant.
Similarly, we have
$
\| w_{2,i}^-\|_1\leq C 2^{n}\|b_i\|_1.
$
Thus we get from \eqref{M7} and $w_2=\sum_i w_{2,i}$ that
\[
I_3=\mu\{(t,x)\in \R^+\times\R^n\,:\,|w_2(t,x)|>\alpha\}\leq \frac1{\alpha}\sum_i (\|w_{2,i}^+\|_1+\|w_{2,i}^-\|_1)\leq \frac{C2^n}\alpha\sum_i\|b_i\|_1.
\]
By (f4) and (f5), we obtain
\beq\lb{M15}
I_3\leq {C2^{2n}}\sum_i|B_i^+|\leq \frac{C12^n}\alpha\|f\|_1.
\eeq

Overall, from \eqref{M13}, \eqref{M14} and \eqref{M15}, we proved that for some $C$ independent of $n$,
\[
\mu\{(t,x)\in\R^+\times\R^n\,:\,|\calR f(t,x)|>\alpha\} \leq \frac{C2^{8n}n^2 }{\alpha}\|f\|_1.
\]
Thus $\calR$ maps $L^1(\R^+\times\R^n)$ to $L^1_w(\R^+\times\R^n)$ with operator norm bounded by $C2^{8n}n^2$.

\medskip

\noindent {\bf Step 4.} In view of Step 2 and Step 3, the Marcinkiewicz interpolation theorem yields for any $p\in (1,2)$, the operator $\calR$ maps $L^p(\R^+\times\R^n)$ to $L^p(\R^+\times\R^n)$ with operator norm bounded by
\[
(C_{p,n})^\frac{1}{p}:=2\left[\frac{p}{(p-1)(2-p)}\right]^{1/p}(C2^{8n}n^2)^\frac{2-p}{p} 4^\frac{2p-2}{p}.
\]

To find out $C_{\gamma,n}$ for $\gamma>2$, we use a duality argument. The proof will only be given for smooth solutions and the general result holds by approximations. Let $p:=\frac{\gamma}{\gamma-1}\in (1,2)$. Suppose that for some $T>0$ and some function $h\in L^p([0,T]\times\R^n)$, $\phi$ is a smooth, compactly supported vector valued solution to
\[
-\phi_t=\Delta \phi+\nabla h,\quad\text{ in $[0,T]\times\R^n$ with }\phi(T,\cdot)=0.
\]
Due to the bound of $\calR$ obtained on $L^p(\R^+\times\R^n)$, we know
\beq\lb{M11}
\|\nabla\cdot \phi\|_{p}\leq (C_{p,n})^\frac{1}{p}\|h\|_p .
\eeq

Now let $u$ be from \eqref{M0} and pick $h$ such that
\beq\lb{M8}
\|\nabla\cdot u\|_\gamma=\int_{0}^T\int_{\R^n} (\nabla\cdot u ) h\,dxdt\quad\text{ and }\quad \|h\|_p=1.
\eeq
It follows from the equations
\begin{align*}
0&=\frac{d}{dt}(\int_{0}^T\int_{\R^n} u\cdot\phi\,dxdt)=\int_{0}^T\int_{\R^n} (\Delta u+\nabla g)\cdot \phi+u\cdot(-\Delta\phi-\nabla h)dxdt\\
&  =  \int_{0}^T\int_{\R^n} -g \nabla\cdot \phi+(\nabla\cdot u) h\,dxdt.
\end{align*}
Therefore, by \eqref{M11} and \eqref{M8},
\[
\|\nabla\cdot u\|_\gamma\leq \int_{0}^T\int_{\R^n}|g \nabla\cdot\phi|\,dxdt\leq \|g\|_\gamma\|\nabla\cdot\phi\|_p\leq (C_{p,n})^\frac{1}{p}\|g\|_\gamma,
\]
which implies that
\[
\|\Delta v\|_\gamma^\gamma\leq (C_{p,n})^\frac{\gamma}{p}\|g\|_\gamma^\gamma.
\]
Thus using $p=\frac{\gamma}{\gamma-1}$ again, we obtain
\[
C_{\gamma, n}\leq (C_{p,n})^\frac{\gamma}{p}\leq C^\gamma 2^{8n(\gamma-2)+\gamma}n^{2(\gamma-2)}\left(\frac{\gamma-1}{\gamma-2}\right)^{\gamma-1}\leq  C^\gamma (\gamma-2)^{1-\gamma}2^{9n(\gamma-2)}.
\]
\end{proof}

\end{document}